\documentclass[12pt,a4paper]{article}
\usepackage{geometry}
\usepackage{stackengine} 
\usepackage[utf8]{inputenc}
\usepackage[english]{babel}
\usepackage{amssymb, amsthm, amsmath} 
\usepackage{verbatim}
\usepackage{makeidx}
\usepackage{graphicx}
\usepackage{enotez}
\usepackage{quoting}
\usepackage{url}
\usepackage{eucal}
\usepackage{sectsty}
\usepackage{lipsum}
\subsubsectionfont{\small}
\makeatletter
\def\blfootnote{\xdef\@thefnmark{}\@footnotetext}
\makeatother

\usepackage{xcolor}

\theoremstyle{plain}
\newtheorem{lem}{Lemma}[section]

\newtheorem{teo}[lem]{Theorem}

\newtheorem{propo}[lem]{Proposition}
\theoremstyle{definition}
\newtheorem{exa}[lem]{Example}
\newtheorem{rema}[lem]{Remark}
\newtheorem{defi}[lem]{Definition}

\newtheorem{cons}[lem]{Construction}

\usepackage{tikz}
\usepackage{tikz-cd}

\DeclareMathAlphabet{\mathbbe}{U}{bbold}{m}{n}
\newcommand{\simplexcategory}{\mathbbe{\Delta}}

\usepackage{hyperref}
\hypersetup{
    colorlinks=true,
    linkcolor=blue,
    filecolor=magenta,      
    urlcolor=cyan,
}

\usepackage[color]{changebar}
\cbcolor{red}

\usepackage{stmaryrd}
\newcommand{\actto}{\rightarrow\Mapsfromchar}
\tikzset{  act /.tip = >|}

\tikzset{pullback/.style={minimum size=1.2ex,path picture={
\draw[opacity=1,black,-,#1] (-0.5ex,-0.5ex) -- (0.5ex,-0.5ex) -- (0.5ex,0.5ex);%
}}}

 \usepackage[all]{xy}

\newcommand{\drpullback}{\arrow[phantom]{dr}[very near start,description]{\lrcorner}}
\newcommand{\dlpullback}{\arrow[phantom]{dl}[very near start,description]{\llcorner}}

\newcommand{\drpushout}{\arrow[phantom]{dr}[very near end,description]{\ulcorner}}

\newcommand{\FB}{\operatorname{FB}}
\newcommand{\CEM}{\operatorname{CEM}}
\newcommand{\BCK}{\operatorname{BCK}}
\newcommand{\Map}{\operatorname{map}}

\newcommand{\op}{\operatorname{op}}
\newcommand{\lt}{\operatorname{lt}}

\newcommand{\identity}{\operatorname{id}}
\newcommand{\Dec}{\operatorname{Dec}}
\newcommand{\HDec}{\operatorname{HDec}}

\newcommand{\Fib}{\operatorname{Fib}}

\newcommand{\ter}{\operatorname{t}}
\newcommand{\tps}{\operatorname{tps}}

\newcommand{\Fun}{\operatorname{Fun}}
\newcommand{\comp}{\operatorname{comp}}
\newcommand{\posy}{\operatorname{postf}}

\newcommand{\nerve}{\operatorname{N}}
\newcommand{\fatnerve}{\operatorname{\mathbf{N}}}
\newcommand{\frow}{\pi_{\operatorname{firstrow}}}
\newcommand{\GD}{\int^{\mathbf{K}}}
\newcommand{\GDD}{\int^{\mathbf{D}}}
\newcommand{\RAdCon}{\mathfrak{r}}

\newcommand{\invp}{\operatorname{inv}}
\newcommand{\underlyingset}[1]{\underline{#1}}
\newcommand{\posetify}[1]{\widetilde{#1}}
\newcommand{\collapse}[1]{\overline{#1}}

\newcommand{\kat}[1]{\mathbf{#1}}
\newcommand{\Set}{\kat{Set}}
\newcommand{\FinSet}{\kat{FinSet}}
\newcommand{\Cat}{\kat{Cat}}

\newcommand{\Grpd}{\kat{Grpd}}
\newcommand{\Dcmp}{\kat{Dcmp}}

  \makeatletter
\renewcommand{\tableofcontents}{%
   \begin{center}
\begin{minipage}{150mm}
   \begin{center}
                \bf{\contentsname}
        \end{center}
   \footnotesize
  
                \@starttoc{toc}
       
\end{minipage}
        \end{center}
        \addvspace{3em \@plus\p@}
}

\makeatother

\newtheorem{taller}[lem]{$\!\!$}

\newenvironment{blanko}[1]%
{\begin{taller}{\normalfont\bfseries  #1}\normalfont}%
{\end{taller}}

\newenvironment{blanko*}[1]{\begin{list}{\bf {#1} }%
{\setlength{\labelsep}{0mm}\setlength{\leftmargin}{0mm}%
\setlength{\labelwidth}{0mm}\setlength{\listparindent}{\parindent}%
\setlength{\parsep}{\parskip}\setlength{\partopsep}{0mm}}%
\item%
}{\end{list}}

\begin{document}

\title{\Large\textbf{DIRECTED HEREDITARY SPECIES AND  DECOMPOSITION SPACES}} 
\author{\normalsize{ALEX CEBRIAN} \; and  \; \normalsize{WILSON FORERO} } 
\date{}
\maketitle

\begin{abstract}
We introduce the notion of \emph{directed hereditary species} and show that they have associated monoidal decomposition spaces, comodule bialgebras, and operadic categories. The notion subsumes Schmitt's hereditary species,  Gálvez--Kock--Tonks directed restrictions species, and a directed version of Carlier's construction of monoidal decomposition spaces and comodule bialgebras. In addition to all the examples of Schmitt, Gálvez--Kock--Tonks and Carlier, the new construction covers also the Fauvet--Foissy--Manchon comodule bialgebra of finite topological spaces, the Calaque--Ebrahimi-Fard--Manchon comodule bialgebra of rooted trees, and the Faà di Bruno comodule bialgebra of linear trees.
\end{abstract}

\tableofcontents

\section*{Introduction}
\label{sec:intro}


G\'alvez, Kock, and Tonks \cite{GTK1, GTK2, GTK3} introduced decomposition spaces as a
far-reaching generalization of posets for the purpose of defining incidence
algebras and M\"obius inversion, covering the classical theory for posets by Rota~\cite{Rota1987}, for monoids (Cartier--Foata~\cite{Cartier}), M\"obius categories (Leroux~\cite{Lero}, \cite{Content-Lemay-Leroux}), as well as various constructions with operads (\cite{vanderLaan:math-ph/0311013},
\cite{vanderLaan-Moerdijk:hep-th/0210226}, \cite{Chapoton-Livernet:0707.3725}, \cite{Vallette:0405312}).
Decomposition spaces are certain
simplicial $\infty$-groupoids, and the theory becomes homotopical in
nature. Independently Dyckerhoff and Kapranov
\cite{DK} had arrived at the equivalent notion
of 2-Segal space (see Feller et
al. \cite{unital} for the last piece in
the equivalence), from the viewpoint of representation theory, homological
algebra, and K-theory, where Hall algebras and Waldhausen's $S$-construction
are the main motivating examples.

It was suggested by G\'alvez, Kock, and Tonks~\cite{GKT:comb}
that virtually all combinatorial co- bi-
and Hopf algebra should be incidence algebras of decomposition
spaces, whereas many are not directly incidence algebras of
posets. This idea is difficult to
state as a precise theorem, but several recent papers have
contributed with classes of examples vindicating the principle.

A series of combinatorial co- bi- and Hopf algebra which are
not directly incidence algebras of posets were given in the
seminal paper of Schmitt~\cite{schmitt_1993}, where he identified large
significant families of combinatorial coalgebras coming
from extra structures on combinatorial species in the sense of
Joyal~\cite{Joyal:1981} (see Aguiar--Mahajan~\cite{Aguiar-Mahajan} for
further treatment). The two main such structures are restriction
species, as exemplified by the chromatic Hopf
algebra of graphs~\cite{Schmitt:incidence} (see also \cite{Foissy:1611.04303})
and hereditary species, as for example the Fa\`a di Bruno Hopf algebra
(see also \cite{KockWeber}). Restriction species are
presheaves on the category of finite sets and injections. Hereditary
species are presheaves on the category of finite sets and partial
surjections.

The constructions of these two families
of examples have been assimilated into decomposition-space theory:
G\'alvez, Kock, and Tonks~\cite{GKT:restr}
showed how Schmitt's coalgebra of a restriction species is a
special case of the general incidence-coalgebra construction of
decomposition spaces, and generalised to {\em directed
restriction species} (presheaves on the category of finite
posets and convex inclusions); this generalization includes for example the
Butcher--Connes--Kreimer Hopf algebra from numerical
analysis~\cite{Butcher:1972} and renormalisation
theory~\cite{Kreimer:9707029}, \cite{Connes_1998}.
Shortly after, Carlier~\cite{carlier2019hereditary} showed that
also Schmitt's construction of bialgebras from hereditary species
 is a special case of the incidence bialgebra of 
monoidal decomposition spaces. He went on to establish that
these are {\em comodule bialgebras}, an intricate
interaction between two bialgebra structures which is of
importance in certain areas of analysis (see
Manchon~\cite{Manchon:Abelsymposium}, and also \cite{Foissy:1702.05344} and
\cite{Kock_2021}). Carlier  also discovered that hereditary species provide a new
class of examples of  the operadic categories of
Batanin and Markl~\cite{BMpart2, BM, BMpart1}:
while there is a clear operadic flavour in Schmitt's hereditary species, they are not
always operads. They turn out to be operadic categories. 


In the quest to develop a directed version of the theory of Carlier~\cite{carlier2019hereditary} that covers new examples, we found two different routes called: \emph{connected directed hereditary species} and \emph{non-connected directed hereditary species}. We focus our attention in develop the connected case and show that it induces monoidal decomposition spaces, comodule bialgebras, and operadic categories. 

Although the non-connected variant is also an example of monoidal decomposition spaces, comodule bialgebras, and operadic categories, it is only designed to generalise all of the Carlier results on Schmitt hereditary species, but we do not have any interesting example other than the discrete. The connected variant is more interesting, as it covers two important examples, that hitherto were isolated constructions of (comodule) bialgebras, and which are now subsumed as
examples of the general theory of connected directed hereditary species. These two examples are the
Fauvet--Foissy--Manchon (comodule) bialgebra of finite
topologies~\cite{AIF_2017__67_3_911_0} and the
Calaque--Ebrahimi-Fard--Manchon (comodule) bialgebra of
trees~\cite{CALAQUE2011282}.

The bialgebra of finite topologies is essentially the base case in our
setting, namely corresponding to the connected directed hereditary
species of posets. Modulo the difference between posets and preorders,
this turns out to be the same construction, in view of an equivalence we
establish between admissible maps into a poset $T$
(in the sense of \cite{AIF_2017__67_3_911_0}) and
contractions out of $T$ (\ref{subsec:admissible maps and contractions}).

The construction in Fauvet--Foissy--Manchon~\cite{AIF_2017__67_3_911_0} was
inspired by \'Ecalle's mould calculus in dynamical systems
(see~\cite{Cresson}, \cite{Ecalle:I}), and more precisely by the more
elaborate notion of {\em ormould}~\cite{Ecalle-Vallet}. The construction is subtle, and they write in
fact that it would have been difficult to guess the comultiplication
formula without the inspiration from \'Ecalle's work. From the present
viewpoint of connected directed hereditary species, it comes out very naturally from
general principles. (The slight difference between preorders and posets
might in fact be in the latter's favour: it seems that posets are
closer to the ormoulds of \'Ecalle~\cite{Ecalle-Vallet}.
Fauvet--Foissy--Manchon~\cite{AIF_2017__67_3_911_0} had to
introduce the notion of {\em quasi-ormould}.)

The Fauvet--Foissy--Manchon bialgebra constitutes a comodule
bialgebra in conjunction with the bialgebra of
finite topologies of Foissy--Malvenuto--Patras~\cite{Foissy-Malvenuto-Patras:1403.7488}.
Again, the subtle algebraic conditions to be verified are
now a direct consequence of the general theory.

The other important example, the Calaque--Ebrahimi-Fard-Manchon comodule
bialgebra~\cite{CALAQUE2011282}, is the
connected directed hereditary species of trees (\ref{subsec: runningexamplepartI}). It originates in numerical analysis:
its `restriction part' is the Butcher--Connes--Kreimer Hopf algebra
corresponding to composition of B-series~\cite{Butcher:1972}. The
`hereditary part' corresponds to the {\em substition}, a second operation
on B-series discovered by
Chartier--Hairer--Vilmart~\cite{Chartier-Hairer-Vilmart:FCM2010}. Again,
the comodule-bialgebra condition is now a formal consequence of the theory.
Previously, Kock~\cite{Kock_2021} had given a decomposition-space
interpretation (in fact in terms of operads) of the
Calaque--Ebrahimi-Fard-Manchon comodule bialgebra fitting it into the
general framework of the Baez--Dolan construction~\cite{Baez-Dolan:9702},
but his construction gives an operad version with operadic trees
rather than the combinatorial trees relevant in numerical
analysis. An ad hoc quotient construction was required to recover
those. The machinery of connected directed hereditary species delivers the
Calaque--Ebrahimi-Fard-Manchon comodule
bialgebra directly.


Finally, in the context of the theory of operadic categories, our proof that connected directed hereditary species induce operadic categories is very different from Carlier's proof in the discrete case.
Where he simply verified the $9$ axioms for operadic categories one by one
by hand, in the present work we exploit a recent conceptual simplicial
approach to operadic categories by Batanin, Kock, and Weber~\cite{BataninKockWeber}.
They reinterpret the operadic-category axioms in simplicial terms in
such a way that all 9 axioms end up as simplicial identities. In the end
the category of operadic categories can be described as a strict pullback,
involving small categories with chosen local terminal objects (in the
style of \cite{GKW2021}) and certain simplicial
groupoids. With this formalism in hand, we can establish the functor
from connected directed hereditary species to operadic categories by
exploiting the universal property of the pullback, without having to
check any axioms by hand.

\subsection*{Organisation of the paper}


After a brief review in Section \ref{section: preliminaries} of needed notions and some basic results on homotopy pullbacks of groupoids, decomposition spaces (\ref{subsection:decomposition spaces}), culf maps (\ref{subsec:culf}) and the decalage construction (\ref{subsec:decalage}); we introduce in Section \ref{section:decomposition K} the notion of contraction (\ref{definition:connectedcontraction}), and the connected directed hereditary species as a $\Grpd$-valued presheaf on the category of finite connected posets and partially defined contractions (\ref{definition: DCHS}). Examples of this notion are the Fauvet--Foissy--Manchon Hopf algebra of finite topologies, the Calaque--Ebrahimi-Fard--Manchon comodule bialgebra of rooted trees (\ref{subsec: runningexamplepartI}) and the Faà di Bruno comodule bialgebra of linear trees (\ref{exa:linear trees 1}). It should be noted that the old theory of hereditary species does not cover these examples.

In Section \ref{subsec:dec space K}, we define the monoidal decomposition space of contractions $\mathbf{K}$. In addition, we show that $\mathbf{K}$ is complete (\ref{proposition:Kiscomplete}), locally finite, locally discrete, and of locally finite length (\ref{proposition:Kislocally}), and monoidal (\ref{propo:k is monoidal}).

In Section \ref{section:decomposition AP}, we define the monoidal decomposition space of admissible maps of preorders $\mathbf{A}$. In \ref{subsec:comparation FFM with K}, we show that the incidence bialgebra of $\mathbf{A}$ corresponds to the Fauvet--Foissy--Manchon bialgebra of finite topologies. In Section \ref{subsec:admissible maps and contractions}, we relate the notions of admissible maps of preorders (due to \cite{AIF_2017__67_3_911_0}) and of contractions of posets through a culf map (\ref{proposition:AP to K is culf}). This map explains the connection between admissible maps of preorders and contractions of posets. In \ref{subsec: double category AP}, we define the double category $\mathbf{AdCon}$ of finite preorders, admissible maps as horizontal morphisms and contractions as vertical morphisms and show that the Waldhausen $S_{\bullet}$-construction in the sense of Bergner et al. \cite{BergnerSDoubleconstruction} of $\mathbf{AdCon}$ is equivalent to the decomposition space of admissible maps $\mathbf{A}$.

In Section \ref{subsec: DCH species as decomposition}, we show that every connected directed hereditary species has an associated a monoidal decomposition space (\ref{propo: H is a ds connected case}) which is locally finite, locally discrete, and of locally finite length (\ref{propo: H is locally etc}). 

Every hereditary species has an underlying {monoidal} restriction species; therefore, we have two bialgebra structures associated with a hereditary species. Carlier \cite{carlier2019hereditary} showed that these bialgebras are compatible in the sense that the incidence bialgebra associated with the restriction species is a left comodule bialgebra over the incidence bialgebra of the hereditary species. In Section \ref{section:comudule bialgebra}, we apply Carlier's ideas to the connected directed case. In \ref{subsec: runningexamplepartIII}, we show that the comodule bialgebra of the connected directed hereditary species of rooted trees is the Calaque--Ebrahimi-Fard--Manchon comodule bialgebra of rooted trees \cite{CALAQUE2011282}, and the comodule bialgebra of the connected directed hereditary species of linear trees is the Faá di Bruno comodule bialgebra of linear trees \cite{Kock_2021}.
  
In Section \ref{section:CDHS and operadics categories}, we describe a different construction on connected directed hereditary species, showing that we have a functor from the category of connected directed hereditary species to the category of operadic categories in the sense Batanin and Markl \cite{BM}, using a new approach by Batanin, Kock, and Weber~\cite{BataninKockWeber}. This result is interesting because it brings new examples to the theory of operadic categories.

On the other hand, Schmitt's hereditary species are not connected directed hereditary species, as the fibres along a surjection between discrete posets are not necessarily connected. To cover these examples, in Section \ref{section:decomposition D}, we introduce the notion of collapse map of posets (\ref{definition: contraction}) and define the decomposition space of collapse maps $\mathbf{D}$ (\ref{subsec: CP definition}). Similarly to the connected case, we prove that directed hereditary species induce monoidal decomposition spaces (\ref{subsec:direcherspecies as decomposition}), comodule bialgebras (\ref{subsec:dhs comodule bialgebra}), and operadic categories (\ref{subsec:dhs as operadic categories}).  

\subsection*{Acknowledgements}
The authors would like to thank Joachim Kock for his advice and support all along the project. The second author thanks Jesper Møller for hosting him at the Copenhagen Centre for Geometry and Topology, where this work was completed, and Andrew Tonks for his interest and feedback. The authors were supported by grant number MTM2016-80439-P. The second author was supported by grant number PID2020-116481GB-I00(AEI/FEDER, UE) of Spain.

\section{Preliminaries}\label{section: preliminaries}

This section establishes a few background facts and notation for the reader. These results are not new.




\begin{lem}\cite{LK}\label{lemma:pullbackfibres}
A square of groupoids 
\begin{equation}\label{dg:hopbk}
\begin{tikzcd}
P  \arrow[r, "u" ] \arrow[d, " "]& Y \arrow[d, ""] \\
X \arrow[r, "f"']& S
\end{tikzcd}
\end{equation}
is a homotopy pullback diagram if and only if for each $x \in X$ the induced comparison map $u_x: P_x \rightarrow Y_{fx}$ is an equivalence.
\end{lem}

In case $f$ is an isofibration in Lemma \ref{lemma:pullbackfibres}, the strict pullback of the square~(\ref{dg:hopbk}) is homotopy equivalent to the homotopy pullback  \cite[Theorem 1]{Joyal1993}. We usually omit the term homotopy and say `pullback',

\begin{blanko} 
{Fat nerve}\label{subsec:fatnerve}
\end{blanko}

The \emph{fat nerve} of a category $\mathbb{X}$ is the simplicial groupoid
\begin{align*}
\fatnerve\mathbb{X}: & \simplexcategory^{\op} \rightarrow  \Grpd      \\
& {[n]} \mapsto    {\Map([n], \mathbb{X})}.
\end{align*}
where $\Map([n], \mathbb{X})$ is the mapping space, defined as the maximal subgroupoid of the functor category $\Fun([n], \mathbb{X})$. 

\begin{blanko} 
{Symmetric monoidal category functor}\label{subsec:monoindalfunctor}
\end{blanko}

Let $\FinSet^{\operatorname{bij}}$ denote the groupoid of all finite sets and bijections.
The symmetric monoidal category functor $\mathsf{S} \colon \Grpd \rightarrow \Grpd$ \cite{KockWeber} is defined for each groupoid $X$ as the comma category $\FinSet^{\operatorname{bij}}_{/X}$. To simplify notation, when we refer to an object $\mathsf{S}X$ we will write ${\lbrace P_i \rbrace}_{i \in I}$ where $I$ is a finite set except in Section \ref{section:CDHS and operadics categories}, where we will use finite ordinals in order to have precise constructions in the context of operadic categories.

The algebras over $\mathsf{S}$ are symmetric monoidal categories. The unit sends an element $x$ to the list with one element $(x)$, and the multiplication takes disjoint union of index sets. Furthermore, $\mathsf{S}$ preserves homotopy pullbacks and fibrations \cite{Weber2015InternalAC}. 

\begin{rema}
Let $\mathbb{B}$ denotes the skeletal category of
finite ordinals and all bijections. The groupoid $\mathsf{S}X$ is normally defined
as the comma category (weak slice) $\mathbb{B}_{/X}$. We prefer to work with arbitrary finite sets to simplify many proofs.
\end{rema}

\begin{blanko} 
{Decomposition spaces}\label{subsection:decomposition spaces}
\end{blanko}

The \textit{simplex category} $\simplexcategory$ is the category whose objects are the nonempty finite ordinals and whose morphisms are the monotone maps. These are generated by \emph{coface} maps $d^i: [n-1] \rightarrow [n]$, which are the monotone injective functions for which $i \in [n]$ is not in the image, and \emph{codegeneracy} maps $s^i: [n+1] \rightarrow [n]$, which are monotone surjective functions for which $i \in [n]$ has a double preimage. We write $d^\bot := d^0$ and $d^\top := d^n$ for the outer coface maps. 

An arrow of  $\simplexcategory$  is termed \emph{active},
and written $g : [m] \actto [n]$, if it preserves end-points. An arrow
is termed \emph{inert}, and written $f : [m] \rightarrowtail [n]$, if it is distance preserving.

\begin{defi}\cite[Definition 3.1]{GTK1} \label{definitiondecompositionspace}
A \emph{decomposition space} is a simplicial groupoid $X \colon \simplexcategory^{\op} \rightarrow \Grpd$
such that the image of any active-inert pushout in $\simplexcategory$ is a pullback of groupoids. This is equivalent \cite[Proposition 3.5]{GTK1} to requiring that for each $n \geq 2$ the following diagrams are pullbacks for $0 < i < n$:
\begin{center}
\begin{tikzcd}
X_{n+1} \drpullback \arrow[r, "d_{i+1}"]\arrow[d, "d_\bot"']& X_n \arrow[d, "d_\bot"] \\
X_n \arrow[r, "d_i"']& X_{n-1}
\end{tikzcd} \; \; \;
\begin{tikzcd}
X_{n+1} \drpullback \arrow[r, "d_{i}"]\arrow[d, "d_\top"']& X_n \arrow[d, "d_\top"] \\
X_n \arrow[r, "d_i"']& X_{n-1}.
\end{tikzcd}
\end{center}
\end{defi}

The notion of decomposition space is a far-reaching generalization of posets (Rota~\cite{Rota1987}) and M\"obius categories (Leroux~\cite{Lero}) for the purpose of defining incidence algebras and M\"obius inversion. The comultiplication is given by (for $f \in X_1$)
$$\Delta_X(f) = \sum_{\binom{\sigma \in X_2}{d_1(\sigma) = f)}} d_2(\sigma) \otimes d_0(\sigma)$$
which means that we `sum over all 2-simplices with long edge $f$ and return the two short edges'.

\begin{defi}\cite[\S 2.9]{GTK1} \label{definitionsegal}
A simplicial space $X \colon \simplexcategory^{\op} \rightarrow \Grpd$ is called a \emph{Segal space} if for each $n > 0$  the following diagram is a pullback
\begin{center}
\begin{tikzcd}
X_{n+1} \drpullback \arrow[r, "d_\top"]\arrow[d, "d_\bot"']& X_n \arrow[d, "d_\bot"] \\
X_n \arrow[r, "d_\top"']& X_{n-1}.
\end{tikzcd}
\end{center}
\end{defi}

\begin{exa}\label{exa:fatnerve is Segal}
The fat nerve of a category is a Segal space \cite[\S 2.14]{GTK1}.
\end{exa}

\begin{propo}\cite[Proposition 3.7]{GTK1} \label{segaldecomp}
 Any Segal space is a decomposition space.
\end{propo}


\begin{blanko}
{Culf maps} \label{subsec:culf}
\end{blanko}
\noindent 
A simplicial map $F \colon X \rightarrow Y$ is \emph{culf} if $F$ is cartesian on each active map.

 
\begin{lem}\cite[Lemma 4.3]{GTK1} \label{lemma culfcondition for ds}
A simplicial map between decomposition spaces is culf if and only if it is cartesian on $d^1 \colon [1] \rightarrow [2]$.
\end{lem}

We will need the following lemmas in \ref{subsec:direcherspecies as decomposition}.

\begin{lem}\cite[Lemma 4.6]{GTK1} \label{proposition: ds+culf=ds}
If $X$ is a decomposition space and $f \colon Y \rightarrow X$ is a culf map, then also $Y$ is a decomposition space.
\end{lem}

In order to be able to take homotopy cardinality to get a coalgebra in vector spaces, $X$ must be \emph{locally finite}. This means that $X_1$ is a locally finite groupoid and that the maps $s_0 \colon X_0 \rightarrow X_1$ and $d_1 \colon X_1 \rightarrow X_0$ have finite fibres. We will also use the notion of \emph{locally discrete} decomposition space, which means that $s_0 \colon X_0 \rightarrow X_1 $ and $d_1 \colon X_1 \rightarrow X_0$ have discrete fibres. Moreover, to have a M\"{o}bius inversion formula, yet another condition is needed. A decomposition space is of \emph{locally finite length} if for each $a \in X_1$ there is an upper bound on the $n$ for which the map $X_n \actto X_1$ has non-degenerate elements in the fibre. Since the culfness condition refers to active maps, just as the finiteness conditions stated above, we have the following result.

\begin{lem}\cite[Lemma 1.12]{GKT:restr} \label{proposition: ds+locallydiscrete = locally discrete}
If $X$ is a locally discrete (resp.~locally finite length) decomposition space and $f \colon Y \rightarrow X$ is a culf map, also $Y$ is locally discrete (resp. locally finite length) decomposition space. This also the case for locally finite, but we must to check that $Y_1$ is locally finite.
\end{lem}

\begin{blanko} 
{Decalage}\label{subsec:decalage}
\end{blanko}

\noindent Given a simplicial space $X$, the \emph{lower dec} $\Dec_{\perp} X$ is a new simplicial space obtained by deleting $X_0$ and shifting everything one position down, also deleting all $d_0$ face maps and all $s_0$ degeneracy maps. It comes equipped with a simplicial map, called the lower dec map, $d_\perp \colon \Dec_{\perp} X \rightarrow X$ given by the original $d_0$. 
 Similarly, the \emph{upper dec} $\Dec_{\top} X$ is obtained by instead deleting, in each degree, the top face map $d_\top$ and the top degeneracy map $s_\top$. The deleted top face maps become the upper dec map $d_\top \colon \Dec_{\top} X \rightarrow X$. The decomposition property can be characterized in terms of decalage.

\begin{teo}\cite{DK, unital, GTK1} \label{theorem: decalage and segal condition}
A simplicial space $X \colon \simplexcategory^{\op} \rightarrow \Grpd$ is a decomposition space if and only if
both $\Dec_{\perp} X$ and $\Dec_{\top} X$ are Segal spaces.
\end{teo}

\begin{blanko} 
{Monoidal decomposition spaces}\label{subsec:monoidaldspace}
\end{blanko}

There is a natural notion of monoidal decomposition space \cite[\S9]{GTK1}; which leads to bialgebras. It is a decomposition space $X$ with two functors $\eta \colon 1 \rightarrow X$ and $\otimes \colon X \times X \rightarrow X$ required to be a monoidal structure and culf. In examples coming from combinatorics, the monoidal structure will typically be given by categorical sum.

\section{Connected directed hereditary species}
\label{section:decomposition K}

In this section, we will introduce the concept of connected directed hereditary species, but first, we will provide a series of tools necessary to define this notion.

\begin{blanko} 
{Contractions}\label{sec:decspace}
\end{blanko}

A map of posets $f \colon P \rightarrow Q$ is \emph{convex} if for all $x,y \in P$ and $f(x) \leq w \leq f(y)$ in $Q$ there is a unique $p \in P$ with $x \leq p \leq y$ and $f(p) = w$.

\begin{lem}\label{lemma:convezmapstablepullback}
In the category of posets, convex maps are stable under pullback.
\end{lem}

In a poset $P$, we say that $p'$ covers $p$, written $p \lessdot p'$, if $p < p'$ and there is no element $x$ such that $p < x < p'$.  

\begin{defi}\label{definition:connectedcontraction}
A map of posets $f \colon P \rightarrow Q$ is a \emph{contraction} if:
\begin{enumerate}
    \item $f$ is a monotone surjection;
    \item for each $q \in Q$, the fibre $P_q$ is a connected convex subposet of $P$;
    \item for any cover $q \lessdot q'$ in $Q$, there exists a cover $p \lessdot p'$ in $P$ such that $f(p) = q$ and $f(p') = q'$.
\end{enumerate}
\end{defi}

\begin{rema}
The following picture gives an illustration of a contraction
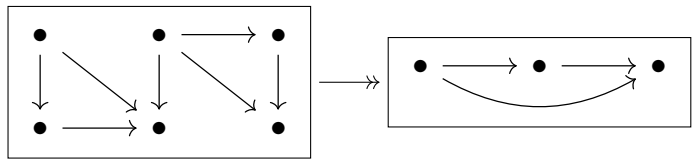
\begin{center}
\begin{tikzpicture}[yscale=1,xscale=1]
\node[draw] (box2) at (-5,0){%
\begin{tikzcd}
	\bullet & \bullet & \bullet \\
	\bullet & \bullet & \bullet
	\arrow[from=2-1, to=2-2]
	\arrow[from=1-1, to=2-1]
	\arrow[from=1-1, to=2-2]
	\arrow[from=1-2, to=2-2]
	\arrow[from=1-2, to=2-3]
	\arrow[from=1-2, to=1-3]
	\arrow[from=1-3, to=2-3]
\end{tikzcd}
};
\node[draw] (box1) at (0,0){%
\begin{tikzcd}
	\bullet & \bullet & \bullet
	\arrow[from=1-1, to=1-2]
	\arrow[from=1-2, to=1-3]
	\arrow[bend right = 30, from=1-1, to=1-3]
\end{tikzcd}
};
\draw[->>, shorten <=3pt, shorten >=3pt] (box2) to (box1);
\end{tikzpicture}
\end{center}
Note the importance of demand only to lift covers $\lessdot$; if we had demanded to be able to lift $<$, then the above map would not be an example of contraction.
\end{rema}

\begin{rema}
The following picture gives an illustration of a monotone surjection that satisfies condition $(2)$ in Definition \ref{definition:connectedcontraction} but does not lift covers
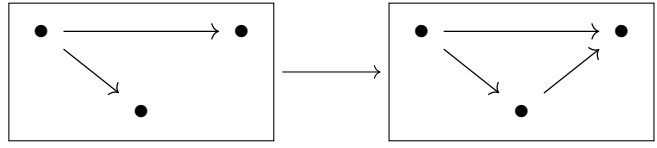
\begin{center}
\begin{tikzpicture}[yscale=1,xscale=1]
\node[draw] (box2) at (-5,0){%
\begin{tikzcd}[sep=scriptsize]
\bullet && \bullet \\
	& \bullet
	\arrow[from=1-1, to=1-3]
	\arrow[from=1-1, to=2-2]
\end{tikzcd}
};
\node[draw] (box1) at (0,0){%
\begin{tikzcd}[sep=scriptsize]
	\bullet && \bullet \\
	& \bullet
	\arrow[from=1-1, to=1-3]
	\arrow[from=1-1, to=2-2]
	\arrow[from=2-2, to=1-3]
\end{tikzcd}
};
\draw[->, shorten <=3pt, shorten >=3pt] (box2) to (box1);
\end{tikzpicture}
\end{center}
If we allow this map to be a contraction, Lemma \ref{lemma:surjectivetoprove K is DS} would be false. This lemma is a key ingredient for developing the connection between directed hereditary species and decomposition spaces, so cover lifting is necessary.
\end{rema}

\begin{lem}\label{lemma:stableconnectedpullback}
In the category of posets, contractions are stable under pullback along convex maps.
\end{lem}

\begin{proof}
Let $P, Q$ and $V$ be posets. Let $f \colon P \twoheadrightarrow V$ be a contraction. Let $g \colon Q \rightarrow V$ be a convex map, and let
\begin{center}
\begin{tikzcd}
P \times_V Q \arrow[d, "\pi_P"'] \arrow[r, "\pi_Q"] \drpullback \arrow[white]{dr}[black, description]{(1)} & Q \arrow[d, "g"] \\
P \arrow[r, "f"', two heads]                                   & V                          
\end{tikzcd}
\end{center}
be a pullback diagram. Since monotone surjections are stable under pullback, we have that $\pi_Q$ is a monotone surjection. Let us see that for each $q \in Q$, the fibre $(P \times_{V} Q)_q$ is a connected convex subposet. Since the diagram $(1)$ is a pullback and $g$ is a convex, for each $q \in Q$, the map $\pi_P: (P \times_V Q)_q \rightarrow P_{g(v)}$ is a convex map. Furthermore, $P_{g(v)}$ is a connected convex subposet of $P$ since $f$ is a contraction. Combining this with the convex property of $\pi_P$, it follows that $(P \times_V Q)_q$ is a connected convex subposet of $P \times_V Q$. 
It only remains to prove that for any cover $q \lessdot q'$ in $Q$, there exists a cover $(p, q) \lessdot (p', q')$ in $P \times_{V} Q$. Let $q \lessdot q'$ be a cover in $Q$. Since $g$ is convex, we have that $g(q) \lessdot g(q')$. Since $f$ is a contraction, there exists $p$ and $p'$ in $P$ such that $p \lessdot p'$, and $f(p)= g(q)$ and $f(p') = g(q')$. This means that $(p,q) \lessdot (p', q')$. 
\end{proof}


\begin{defi}
A \emph{partially defined contraction} $P \rightarrow Q$ consists of a convex map $\imath \colon P' \to P$ and a contraction $f \colon P' \twoheadrightarrow Q$ represented as the span:
$\begin{tikzcd}[column sep=scriptsize]
	P & {P'} & Q.
	\arrow["f", two heads, from=1-2, to=1-3]
	\arrow["\imath"', tail, from=1-2, to=1-1]
\end{tikzcd}$
\end{defi}
Let $\mathbb{K}_p$ denote the category of finite connected non-empty posets, and whose morphisms are partially defined contractions. Partially defined contractions are composed by pullback composition of spans in the category of posets by the stability property of convex maps (\ref{lemma:convezmapstablepullback}) and contractions under pullbacks (\ref{lemma:stableconnectedpullback}).

\begin{blanko} 
{Species}\label{section:species}\label{subsec:directedhereditaryspeices}
\end{blanko}

The theory of species was introduced by Joyal~\cite{Joyal:1981} as a combinatorial theory of formal power series. A species is a functor  from the category of finite sets and bijections to the category of sets. Usual types of species include graphs, trees, endomorphisms, permutations, etc. Schmitt \cite{schmitt_1993} extended the notion of species in such a way that the combinatorial structures are classified according to their functorial properties. Let $\mathbb{S}_p$ be the category of finite sets and partial surjections: a \emph{hereditary species}~\cite{schmitt_1993} is a functor $H \colon \mathbb{S}_p \rightarrow \Set$. If partial surjections are portrayed as spans
\begin{tikzcd}[column sep=small]
	P & {P'} & Q,
	\arrow["\iota"', tail, from=1-2, to=1-1]
	\arrow["f", two heads, from=1-2, to=1-3]
\end{tikzcd}
we can say that $H$ is contravariant in injections and covariant in surjections. 


We introduce the new notion of \emph{connected directed hereditary species} to cover examples related to the Fauvet--Foissy--Manchon comodule bialgebra of finite topologies and admissible maps; and the Calaque--Ebrahimi-Fard--Manchon comodule bialgebra of rooted trees. But it does not cover Schmitt hereditary species. For such species, we need the non-connected case, which will be defined in Section \ref{section:decomposition D}.

\begin{defi}\label{definition: DCHS}
A \emph{connected directed hereditary species} is a functor $H \colon \mathbb{K}_p \rightarrow \Grpd$. 
\end{defi}
Note that $H$ is covariant in contractions and contravariant in convex maps.
An element of $H[P]$ is called a $H$-structure on the finite poset $P$. 

\begin{blanko}
{Coalgebras from directed connected hereditary species}\label{subsec:coalgebraofcdhs}
\end{blanko}

As in the case of Schmitt hereditary species, connected directed hereditary species give rise to coalgebras. Let $P$ be a poset and $X$ an $H$-structure on $P$. 
The comultiplication of $X$ is given by 
$$\Delta(X) = \sum_{f \colon P \twoheadrightarrow Q} \lbrace X \vert P_q \rbrace_{q \in Q}  \otimes X_! Q,$$
where the sum ranges over isomorphism classes of contractions $P\twoheadrightarrow Q$. Here $X \vert P_q$ is the restriction of $X$ to the fibre $P_q$ for every $q\in Q$, and $X_! Q$ is the $H$-structure $H(f)(X)$ in $H(Q)$.

\begin{blanko}
{Running example, part I: Calaque--Ebrahimi-Fard--Manchon comodule bialgebra of rooted trees} \label{subsec: runningexamplepartI}
\end{blanko}

 We will see that the coalgebra of rooted trees studied by Calaque, Ebrahimi-Fard, and Manchon \cite{CALAQUE2011282} comes from a connected directed hereditary species. For this, we need some preliminaries.
 
Let $P$ be a poset. For $p, p'$ in $P$, an interval $[p,p']$ is \emph{linear} if for each $z,w \in [p,p']$, we have that $z \leq w$ or $w \leq z$.
 
\begin{defi}
 A \emph{tree} $T$ is a poset with a terminal object where every interval is linear.
\end{defi}
 
A \emph{forest} is a disjoint union of trees. We need the following results to obtain a connected directed hereditary species with connected posets and contractions. Let $P$ be a poset. For each $x \in P$, we define $P_{/x} := \lbrace p \in P \vert p \leq x \rbrace$.
 
 \begin{lem}\label{lemma:slicesoftreesaretrees} 
 Let $P$ be a tree. For each $x \in P$, the poset $P_{/x}$ is a tree with root $x$.
 \end{lem}
 
 \begin{lem}\label{lemma:slicessetoftreesareforests}
 Let $P$ be a tree. Let $S$ be a convex subposet of $P$, put $P_{/S} = \sum_{x \in S} P_{/x}$. Then $P_{/S}$ is a forest. 
 \end{lem}
 
 \begin{proof}
 Since every interval in $P$ is linear, it is straightforward to see that $P_{/S}$ is a disjoint union of posets. Moreover, for each $x \in S$, the poset $P_{/x}$ is a tree by Lemma \ref{lemma:slicesoftreesaretrees}. Hence, $P_{/S}$ is a forest.
 \end{proof}
 
 \begin{lem}\label{lemma:slicesofforestsareforests}
 Let $T$ be a forest. Let $S$ be a convex subposet of $T$. Then $S$ is a forest.
 \end{lem}
 
 \begin{lem}
 Let $f \colon P \twoheadrightarrow Q$ be a contraction between posets. If $P$ is a tree, then $Q$ is a tree.
 \end{lem}
 
\begin{proof}
Let $\top_P$ denote the terminal object of $P$. The object $f(\top_P)$ is terminal in $Q$. Indeed, let $q$ be an object in $Q$. Since $f$ is a surjection, there exists $p \in P$ such that $f(p) = q$. Furthermore, $p < \top_P$ since $\top_P$ is terminal, and therefore $f(p) < f(\top_P)$ by the monotonicity condition of $f$. 
Now we will prove that every interval in $Q$ is linear. Let $[q_1, q_2]$ be an interval in $Q$, and let $\imath \colon [q_1, q_2] \to Q$ denote the canonical convex map. Consider the pullback diagram
\[\begin{tikzcd}
	{f^{-1}([q_1,q_2])} & P \\
	{[q_1,q_2]} & Q.
	\arrow["f", two heads, from=1-2, to=2-2]
	\arrow[tail, from=1-1, to=1-2]
	\arrow["{f'}"', two heads, from=1-1, to=2-1]
	\arrow["\imath"', tail, from=2-1, to=2-2]
	\arrow["\lrcorner"{anchor=center, pos=0.125}, draw=none, from=1-1, to=2-2]
\end{tikzcd}\]
Here $f'$ is a contraction since contractions are stable under pullback along convex maps (\ref{lemma:stableconnectedpullback}). Since $f$ is a contraction and $\imath$ is a convex map, we have that $f^{-1}([q_1, q_2])$ is a connected convex subposet of $P$. For each $z$ and $w$ in $[q_1, q_2]$, there exists $a$ and $b$ in $f^{-1}([q_1, q_2])$ such that $f'(a) = z$ and $f'(b) = w$ since $f'$ is surjective. By the connectivity property of $f^{-1}([q_1, q_2])$, we have that $a$ and $b$ are connected by a ziz-zag. The contraction property of $f'$ forces that the ziz-zag is of the form: $a < \cdots < b$ or $b < \cdots < a$. By the monotonicity property of $f'$ follows that $z < w$ or $w < z$. Hence, $[q_1,q_2]$ is linear.
\end{proof}

 \begin{lem}\label{lemma:contractionsofforestsareforests}
Let $f \colon P \twoheadrightarrow Q$ be a contraction between posets. Suppose, moreover, that $P$ is a forest, then $Q$ is a forest.
 \end{lem}
 
The functor $H_{\CEM}: \mathbb{K}_p \rightarrow \Grpd$ is defined as follows: if $P$ is a tree $H_{\CEM}(P)$ has one object, which is $P$ itself, and if $P$ is not a tree $H_{\CEM}(P)$ is empty. By Lemma \ref{lemma:slicesofforestsareforests}, we have that $H_{\CEM}$ is contravariant in convex maps. By Lemma \ref{lemma:contractionsofforestsareforests}, the functor $H_{\CEM}$ is covariant in contractions. Therefore, $H_{\CEM}$ is a connected directed hereditary species.

\begin{rema}\label{remark:incidencecoalgebraCEM}
The comultiplication $\Delta_{H_{\CEM}}$ of the incidence coalgebra of $H_{\CEM}$ is given by:
$$\Delta_{H_{\CEM}}(T) = \sum_{f: T \twoheadrightarrow Q} \lbrace T_q \rbrace_{q \in Q} \otimes Q.$$
This coalgebra is the Calaque--Ebrahimi-Fard--Manchon coalgebra of rooted trees \cite{CALAQUE2011282}. 
\end{rema}

\begin{exa}[\textbf{Faà di Bruno comodule bialgebra of linear trees, part I}]\label{exa:linear trees 1}
A \emph{linear tree} is one in which every node has precisely one input edge. The functor $H_{\FB}: \mathbb{K}_p \rightarrow \Grpd$ is defined as follows: if $P$ is a linear tree $H_{\FB}(P)$ has one object, which is $P$ itself, and if $P$ is not a linear tree $H_{\FB}(P)$ is empty. By Lemmas \ref{lemma:slicesofforestsareforests} and  \ref{lemma:contractionsofforestsareforests}, we have  that  $H_{\FB}$ is contravariant in convex maps and covariant in contractions. Therefore, $H_{\FB}$ is a connected directed hereditary species. 
The bialgebra associated to ${H}_{\CEM}$ is the Faà di Bruno bialgebra of linear trees since there is a one-to-one correspondence
between comultiplying contractions of linear
trees and comultiplying monotone surjections $[n] \twoheadrightarrow 1$. The way it appears
here is like in algebraic topology where it
is the (dual) Landweber--Novikov bialgebra
(see \cite[\S 3]{Morava:Oaxtepec}), whereas the usual presentation of the Fa\`a di Bruno
bialgebra is with (non monotone) surjections (see for example \cite{KockWeber} and \cite{Cebrian:2008.09798}) or with partitions, as in \cite{Figueroa-GraciaBondia:0408145}. Over the rational numbers the two are isomorphic (see for example \cite{EbrahimiFard-Lundervold-Manchon:1402.4761}).
\end{exa}


\section{The decomposition space of contractions}\label{subsec:dec space K}

In this section we will introduce the decomposition space of contractions $\mathbf{K}$, but first we need some preliminaries.

Let $\mathbf{Cat}_{\lt}$ denote the category of categories with chosen local
terminals, or equivalently upper-dec coalgebras~\cite{GKW2021}. Let $\mathbb{K}$ denote the category of finite connected non-empty posets and contractions. 

\begin{exa}
In the category $\mathbb{K}$ of connected finite non-empty posets and contractions a chosen terminal object is a poset with one element.
\end{exa}

The \textit{t-simplex category} $\simplexcategory^{\ter}$ is the category whose objects are the nonempty finite ordinals and whose morphisms are the monotone maps that preserve the top element. 

\begin{defi}\label{defi: t-nerve}
For $\mathcal{C}$ a category with chosen local terminals, its \emph{fat $\lt$-nerve} $\fatnerve^{\lt}(\mathcal{C})$ is the $\Grpd$-valued $\simplexcategory^{\ter}$-presheaf describe as follows: for $n \geq 1$, the groupoid $\fatnerve^{\lt}(\mathcal{C})_n$ is the same as the groupoid $\fatnerve(\mathcal{C})_{n-1}$. The groupoid $\fatnerve^{\lt}(\mathcal{C})_{0}$ is the groupoid of chosen local terminal objects in $\mathcal{C}$. The face and degeneracy maps act as the usual fat nerve construction except in $d_\perp \colon \fatnerve^{\lt}(\mathcal{C})_{1} \to \fatnerve^{\lt}(\mathcal{C})_{0}$ that sends each object in $\mathcal{C}$ to its corresponding chosen local terminal object. The degeneracy map $s_0 \colon \fatnerve^{\lt}(\mathcal{C})_{0} \to \fatnerve^{\lt}(\mathcal{C})_{1}$ is the inclusion.
\end{defi}

To simplify the notation, we put $\mathbf{K}^{\circ} := \fatnerve^{\lt}(\mathbb{K})$. Thus, $\mathbf{K}^{\circ}_2$ is the groupoid of contractions, and $\mathbf{K}^{\circ}_1$ is the groupoid of finite connected non-empty posets and monotone bijections. $\mathbf{K}^{\circ}_0$ is the terminal groupoid. To obtain the top face maps, it is necessary to introduce families through the symmetric monoidal category functor. We define $\mathbf{K} := \mathsf{S}\mathbf{K}^{\circ}$ to be the symmetric monoidal category functor $\mathsf{S}$ applied to $\mathbf{K}^{\circ}$. All the face maps (except the missing top ones) and degeneracy maps are $\mathsf{S}$ applied to the face and degeneracy maps of $\fatnerve^{\lt}(\mathbb{K})$. The top face map is given by: 

\begin{itemize}
\item For $n \geq 2$, the top face map $d_\top \colon \mathbf{K}_n \rightarrow \mathbf{K}_{n-1}$ is defined as follows: given a $(n-1)$-chain of contractions 
\begin{center}
\begin{tikzcd}
P_0 \arrow[r, "f_0", two heads] & P_1 \arrow[r, "f_1", two heads] & P_2 \arrow[r, two heads] & \cdots \arrow[r, two heads] & P_{n-2} \arrow[r, "f_{n-2}", two heads] & Q,
\end{tikzcd}
\end{center}
for each element $q \in Q$, we can form the fibres over $q$. We end up with a family
\begin{center}
\begin{tikzcd}
\lbrace (P_0)_{q} \arrow[r, "(f_0)_q", two heads] & (P_1)_q \arrow[r, "(f_1)_q", two heads] & (P_2)_{q} \arrow[r, two heads] & \cdots \arrow[r, two heads] & (P_{n-3})_q \arrow[r, "(f_{n-2})_q", two heads] & (P_{n-2})_q \rbrace_{q \in Q}.
\end{tikzcd}
\end{center}
For each $0 \leq i \leq  n-2$, the map $(f_i)_q$ is a contraction by Lemma \ref{lemma:stableconnectedpullback}.
  
\item $d_\top \colon \mathbf{K}_1 \rightarrow \mathbf{K}_{0}$ sends a family of finite posets $\lbrace Q_i \rbrace_{i \in I}$ to the family whose components are the terminal poset $1$ indexed by the disjoint union $\sum_{i \in I} Q_i$.  
\end{itemize}

Note that the fibres are convex non-empty subposets since we only consider contractions, not arbitrary maps. The simplicial identities will be verified in \ref{proposition:Ksimplicial}. To simplify the proof that $\mathbf{K}$ is a decomposition space, we need some preliminaries.  


\begin{propo}\label{proposition:dectopKisSegal}
We have an equality $\Dec_\top \mathbf{K} = \mathsf{S}\fatnerve\mathbb{K}$.
\end{propo}

\begin{proof}
The proof follows from combining that $\mathbf{K} = \mathsf{S}\fatnerve^{\lt}(\mathbb{K})$, the definition of the fat $\lt$-nerve, and that taking upper decalage is deleting, in each degree, the top face map and the top degeneracy map.
\end{proof}

\begin{rema}\label{rema:decttopKexplication}
Since $\mathbb{K}$ is a category, its fat nerve $\fatnerve\mathbb{K}$ is a Segal space \ref{exa:fatnerve is Segal}, and therefore $\mathsf{S} \fatnerve \mathbb{K}$ is a Segal space, as $\mathsf{S}$ preserves pullbacks and hence Segal objects. By Proposition \ref{proposition:dectopKisSegal}, we have that $\Dec_\top \mathbf{K} = \mathsf{S}\fatnerve\mathbb{K}$. Combining everything, we have that $\Dec_\top \mathbf{K}$ is a Segal space
and therefore
for each $n \geq 2$ the following diagram is a pullback for $0 < i < n$:
\begin{center}
\begin{tikzcd}
\mathbf{K}_{n+1}   \arrow[r, "d_{i+1}"]\arrow[d, "d_\bot"']& \mathbf{K}_n \arrow[d, "d_\bot"] \\
\mathbf{K}_n \arrow[r, "d_i"']& \mathbf{K}_{n-1}.
\end{tikzcd}
\end{center}
\end{rema}

\begin{propo}\label{proposition:Ksimplicial}
The groupoids $\mathbf{K}_n$ and the degeneracy and face maps given above form a pseudosimplicial groupoid $\mathbf{K}$.
\end{propo}

\begin{proof}
The only pseudosimplicial identity is $d_\top d_{\top} \simeq d_\top d_{\top - 1}$. The other simplicial identities are strict and follows from equality $\Dec_\top \mathbf{K} = \mathsf{S}\nerve\mathbb{K}$ of Proposition \ref{proposition:dectopKisSegal}. Let us prove $d_\top d_{\top}(\mathbf{K}_n) \simeq d_\top d_{\top - 1}(\mathbf{K}_n)$ for $n = 2$. For greater $n$ the proof is completely analogous. Given an object 
$\begin{tikzcd}[column sep=scriptsize]
	P & Q & V
	\arrow["f", two heads, from=1-1, to=1-2]
	\arrow["g", two heads, from=1-2, to=1-3]
\end{tikzcd}$
in $\mathbf{K}_3$, we consider the following commutative diagram for each $v \in V$:
\[\begin{tikzcd}
	{(P_v)_q} & {P_v} & P \\
	1 & {Q_v} & Q \\
	& 1 & V.
	\arrow["{\ulcorner v \urcorner}"', from=3-2, to=3-3]
	\arrow[tail, from=1-2, to=1-3]
	\arrow["{f_v}", two heads, from=1-2, to=2-2]
	\arrow[from=2-2, to=3-2]
	\arrow["f", two heads, from=1-3, to=2-3]
	\arrow["g", from=2-3, to=3-3]
	\arrow[tail, from=2-2, to=2-3]
	\arrow[tail, from=1-1, to=1-2]
	\arrow["{\ulcorner q \urcorner}"', from=2-1, to=2-2]
	\arrow[from=1-1, to=2-1]
	\arrow["\lrcorner"{anchor=center, pos=0.125}, draw=none, from=2-2, to=3-3]
	\arrow["\lrcorner"{anchor=center, pos=0.125}, draw=none, from=1-2, to=2-3]
	\arrow["\lrcorner"{anchor=center, pos=0.125}, draw=none, from=1-1, to=2-2]
\end{tikzcd}\]
It is easy to see that $d_\top d_{\top}(\mathbf{K}_3) = \lbrace (P_v)_q \rbrace_{q \in Q}$. Furthermore, note that the top horizontal rectangle is isomorphic to $P_q$ for each $q \in Q$. This implies that $d_\top d_{\top}(\mathbf{K}_3) \simeq d_\top d_{\top - 1}(\mathbf{K}_3)$ since $d_\top d_{\top - 1}(\mathbf{K}_3) = \lbrace P_q \rbrace_{q \in Q}$.
\end{proof}

\begin{lem}\label{lemma:surjectivetoprove K is DS}
Suppose we have a contraction $f \colon P \twoheadrightarrow Q$ between connected posets, and a family of contractions $\lbrace h_q \colon P_q\twoheadrightarrow W_q \rbrace_{q \in {Q}}$ where each $P_q$ and $W_q$ are connected posets. Then there exists a unique connected poset $W$ and contractions $h$ and $g$ such that the diagram 
\[\begin{tikzcd}
	P & W & Q \\
	{P_q} & {W_q}
	\arrow["{h_q}"', two heads, from=2-1, to=2-2]
	\arrow[tail, from=2-1, to=1-1]
	\arrow[dotted, tail, from=2-2, to=1-2]
	\arrow["h"', dotted, two heads, from=1-1, to=1-2]
	\arrow["g"', dotted, two heads, from=1-2, to=1-3]
	\arrow["f", bend left = 20, two heads, from=1-1, to=1-3]
\end{tikzcd}\]
commutes. Here the vertical arrows are convex inclusions.
\end{lem}

\begin{proof}
We will do the proof in three steps: in the first place, we will construct the underlying set of the poset $W$ and the functions $h \colon P \rightarrow W$ and $g \colon W \rightarrow V$. After that, we will construct a partial order $<_W$ on $W$ forced by the requirement that $h$ and $g$ are monotone maps. To conclude, we will prove that $f$ and $g$ are contractions.

\begin{enumerate}
\item Put $W := \sum_{q \in Q} W_q$ and $h := \sum_{q \in Q} h_q$. The map $g \colon W \to Q$ is defined as $g(w) = q$ for $w \in W_q$. Furthermore, the diagram
\[\begin{tikzcd}
	P & W & Q \\
	{P_q} & {W_q} & 1
	\arrow["{h_q}"', two heads, from=2-1, to=2-2]
	\arrow[tail, from=2-1, to=1-1]
	\arrow[dotted, tail, from=2-2, to=1-2]
	\arrow["h"', dotted, two heads, from=1-1, to=1-2]
	\arrow["g"', dotted, two heads, from=1-2, to=1-3]
	\arrow["f", bend left = 20, two heads, from=1-1, to=1-3]
	\arrow[from=2-2, to=2-3]
	\arrow["{\ulcorner q\urcorner }"', from=2-3, to=1-3]
\end{tikzcd}\]
commutes at the level of sets by the way $h$ and $g$ were defined.

\item The partial order $<_W$ on $W$ is defined by taking transitive closure in the following relation: for $w, w' \in W$, we declare that $w <_W w'$ if one of the following conditions is satisfied:
\begin{enumerate}
\item In case $w, w' \in W_q$ and  $w <_{W_{q}} w'$;
\item There exist $p <_P p'$ in $P$ such that $h(p) = w$ and $h(p') = w'$.
\end{enumerate}
The condition $(b)$ is necessary for $h$ and $g$ to be monotone maps.

\item We will prove that $h$ is a contraction. Let $w \in W$, by the way $W$ was defined, we have that $w \in W_q$ for some $q \in Q$. Recall that contractions and convex maps are stable under pullback by Lemmas \ref{lemma:convezmapstablepullback} and \ref{lemma:stableconnectedpullback}. These imply that for each $w \in W$ in the commutative diagram
\[\begin{tikzcd}
	{P_w} & {P_q} & P \\
	1 & {W_q} & W,
	\arrow["h", two heads, from=1-3, to=2-3]
	\arrow[tail, from=1-2, to=1-3]
	\arrow[tail, from=1-1, to=1-2]
	\arrow[from=1-1, to=2-1]
	\arrow[tail, from=2-2, to=2-3]
	\arrow["{\ulcorner w \urcorner }"', from=2-1, to=2-2]
	\arrow["{h_q}", two heads, from=1-2, to=2-2]
	\arrow["\lrcorner"{anchor=center, pos=0.125}, draw=none, from=1-1, to=2-2]
\end{tikzcd}\]
the poset $P_w$ is connected and convex since $h_q$ is a contraction. Hence, $h$ is a contraction. By analogous arguments, we have that $g$ is a contraction.
\end{enumerate}
The poset $W$ is connected since $g$ is a contraction and $Q$ is connected.
\end{proof}

\begin{lem}\label{lemma:discretedibrationK}
For each $0 < i < n$, the map $d_i \colon \mathbf{K}_n \rightarrow \mathbf{K}_{n-1}$ is a fibration.
\end{lem}

\begin{proof}
Since $\mathsf{S}$ preserves fibrations and $d_i \colon \mathbf{K}_n \rightarrow \mathbf{K}_{n-1}$ is equal to $\mathsf{S}(d_i) \colon \mathsf{S}(\mathbf{K}^{\circ}_n) \rightarrow \mathsf{S}(\mathbf{K}^{\circ}_{n-1})$, it is enough to prove that $d_i \colon \mathbf{K}^{\circ}_n \rightarrow \mathbf{K}^{\circ}_{n-1}$ is a fibration. We will only check that $d_1: \mathbf{K}^{\circ}_2 \rightarrow \mathbf{K}^{\circ}_{1}$ is a fibration, the other cases use similar arguments. Let $f \colon P \twoheadrightarrow Q$ be an object in $\mathbf{K}^{\circ}_2$. Let $u: P' \rightarrow P$ be a morphism in $\mathbf{K}^{\circ}_1$. It is straightforward to check that the morphism $\sigma \colon f \circ u \rightarrow f$, pictured in the diagram
\begin{center}
\begin{tikzcd}
P' \arrow[d, "f\circ u"', two heads] \arrow[r, "u"] \arrow[white]{dr}[black, description]{\sigma} & P \arrow[d, "f", two heads] \\
Q \arrow[r, "\identity_Q"']                        & Q,                          
\end{tikzcd}
\end{center}
is a lift of the morphism $u$ in $\mathbf{K}^{\circ}_2$ such that $d_1(\sigma) = u$. 
\end{proof}

\begin{propo}\label{proposition:Kisdecompositionspace}
The pseudosimplicial groupoid $\mathbf{K}$ is a decomposition space.
\end{propo}

\begin{proof}
We will prove that the following diagrams are pullbacks for $0 < i < n$:
\begin{center}
\begin{tikzcd}
\mathbf{K}_{n+1} \arrow[white]{dr}[black, description]{(1)}  \arrow[r, "d_{i+1}"]\arrow[d, "d_\bot"']& \mathbf{K}_n \arrow[d, "d_\bot"] \\
\mathbf{K}_n \arrow[r, "d_i"']& \mathbf{K}_{n-1}
\end{tikzcd} \; \; \;
\begin{tikzcd}
\mathbf{K}_{n+1} \arrow[white]{dr}[black, description]{(2)}  \arrow[r, "d_{i}"]\arrow[d, "d_\top"']& \mathbf{K}_n \arrow[d, "d_\top"] \\
\mathbf{K}_n \arrow[r, "d_i"']& \mathbf{K}_{n-1}.
\end{tikzcd}
\end{center}
Let us prove it for $n=2$. For greater $n$ the proof is completely analogous. The square $(1)$ is a pullback by Remark (\ref{rema:decttopKexplication}). To prove that the square $(2)$ is a pullback we will use Lemma \ref{lemma:pullbackfibres}. This means that $(2)$ is a pullback if and only if for each object
$f \colon P \twoheadrightarrow Q$ in $\mathbf{K}_2$, the map $d_\top \colon \Fib_{f} d_1 \rightarrow \Fib_{d_\top f } d_1 $ is an equivalence of groupoids. For an object $\lbrace h_q \colon P_q \twoheadrightarrow W_q \rbrace_{q \in Q}$ in $\Fib_{d_\top f} d_1$, Lemma \ref{lemma:surjectivetoprove K is DS} gives contractions $h$ and $g$ such that the diagram commutes 
\[\begin{tikzcd}
	P & W & Q \\
	{P_q} & {W_q}.
	\arrow["{h_q}"', two heads, from=2-1, to=2-2]
	\arrow[tail, from=2-1, to=1-1]
	\arrow[tail, from=2-2, to=1-2]
	\arrow["h"', two heads, from=1-1, to=1-2]
	\arrow["g"', two heads, from=1-2, to=1-3]
	\arrow["f", bend left = 20, two heads, from=1-1, to=1-3]
\end{tikzcd}\] 
The commutativity of the diagram implies that \begin{tikzcd}[column sep=small,row sep=scriptsize]
P \arrow[r, "f", two heads] & W \arrow[r, "g", two heads] & Q
\end{tikzcd} is an object in $\Fib_{f} d_1$. Therefore, $d_\top$ is surjective in objects. The map
$d_\top$ is full. Indeed, for any morphism $u_qr_q \colon f_q \to f'_q$ in $\Fib_{d_\top f} d_1$, put $u = \sum_{q \in Q} u_q$ and $r = \sum_{q \in Q} r_q$. The map $ur\identity_Q$ satisfies that $d_\top (ur\identity_Q) = u_qr_q$. Furthermore, the diagram
\[\begin{tikzcd}[column sep=scriptsize,row sep=small]
	& P & {} & {W'} \\
	{P_q} && {W'_q} && Q \\
	& P && W \\
	{P_q} && {W_q}
	\arrow["{h'}", two heads, from=1-2, to=1-4]
	\arrow[tail, from=2-3, to=1-4]
	\arrow[tail, from=2-1, to=1-2]
	\arrow[tail, from=4-1, to=3-2]
	\arrow["{u_q}", from=4-1, to=2-1]
	\arrow["u"{pos=0.2}, shift left=1, from=3-2, to=1-2]
	\arrow["{r_q}"{description, pos=0.3}, from=4-3, to=2-3]
	\arrow["{h_q}"',  two heads, from=4-1, to=4-3]
	\arrow["{h'_q}"{description, pos=0.7},  two heads, from=2-1, to=2-3]
	\arrow["r", from=3-4, to=1-4]
	\arrow[tail, from=4-3, to=3-4]
	\arrow["h"{pos=0.8}, two heads, from=3-2, to=3-4]
	\arrow["{g'}", two heads, from=1-4, to=2-5]
	\arrow["g"', two heads, from=3-4, to=2-5]
\end{tikzcd}\]
commutes by Lemma \ref{lemma:surjectivetoprove K is DS} and the definitions of $u$ and $r$. This implies that $ur\identity_Q$ is a morphism in $\Fib_{f} d_1$ and hence $d_\top$ is full. To prove that $d_\top$ is faithful, let $uq\identity_Q$ and $u'q'\identity_Q$ be morphisms in $\Fib_{f} d_1$ such that $d_\top uq\identity_Q = d_\top u'q' \identity_Q$ in $\Fib_{d_\top f} d_1$. This means that for each $q \in Q$, we have $u_q = u'_q$ and $r_q = r'_q$, but $ur$ and $u'r'$ are determined by $u_q r_q$ and ${u'}_q {r'}_q$, hence $u = u'$ and $r = r'$.
\end{proof}

A decomposition space $X$ is \emph{complete} when $s_0 \colon X_0 \rightarrow X_1$ is a monomorphism (i.e. is ($-1$)-truncated). It follows from the decomposition space axiom that, in this case, all degeneracy maps are monomorphisms \cite[Lemma 2.5]{GTK2}.

\begin{propo}\label{proposition:Kiscomplete}
The decomposition space $\mathbf{K}$ is complete.
\end{propo}

\begin{proof}
Note that $s_0 \colon \mathbf{K}_0 \rightarrow \mathbf{K}_1$ is actually $\mathsf{S}$ of the map $s_0 \colon \mathbf{K}^{\circ}_0 \rightarrow \mathbf{K}^{\circ}_1$, and $\mathsf{S}$ preserves pullbacks and hence monomorphism. This means that $s_0 \colon \mathbf{K}_0 \rightarrow \mathbf{K}_1$ is a monomorphism if $s_0 \colon \mathbf{K}^{\circ}_0 \rightarrow \mathbf{K}^{\circ}_1$ is a monomorphism, but this is clear since $\mathbf{K}^{\circ}_0$ is the terminal groupoid consisting of only the poset with one element and $s_0$ sends the terminal groupoid to the poset with one element. 

\end{proof}

\begin{propo}\label{proposition:Kislocally}
The decomposition space $\mathbf{K}$ is locally finite, locally discrete and of locally finite length.
\end{propo}

\begin{proof}
Since $\mathsf{S}$ respects finite maps and $\mathbf{K} = \mathsf{S} \mathbf{K}^{\circ}$, we will prove that $\mathbf{K}^{\circ}$ is locally finite, locally discrete and of locally finite length:
\begin{itemize}
\item Note that a connected finite non-empty poset has only a finite number of automorphisms. This means that each object in $\mathbf{K}^{\circ}_1$ has a finite number of automorphisms, and therefore, $\mathbf{K}^{\circ}_1$ is locally finite.
\item In the proof of Proposition \ref{proposition:Kiscomplete}, we see that $s_0 \colon \mathbf{K}^{\circ}_0 \to \mathbf{K}^{\circ}_1$ is finite and discrete. Since $\mathbf{K}^{\circ}_1$ is locally finite and $s_0$ is finite, to prove that $\mathbf{K}^{\circ}$ is locally discrete and locally finite, we have to check that $d_1 \colon \mathbf{K}^{\circ}_2 \to \mathbf{K}^{\circ}_1$ is discrete and finite. In the proof of Lemma \ref{lemma:discretedibrationK}, we showed that $d_1$ is a fibration so we will use strict fibres. In the strict pullback diagram 
\begin{center}
\begin{tikzcd}
\Fib_{d_1}(P) \arrow[d] \arrow[r] \drpullback & 1 \arrow[d, "\ulcorner P \urcorner"] \\
\mathbf{K}^{\circ}_2 \arrow[r, "d_1"'] & \mathbf{K}^{\circ}_1,                                            
\end{tikzcd}
\end{center}
let $f \colon P \twoheadrightarrow Q$ and $f' \colon P \twoheadrightarrow Q'$ be objects in $\Fib_{d_1}(P)$. A morphism  $u \colon f \to f'$ in $\Fib_{d_1}(P)$, it is in fact a monotone bijection $u \colon Q \to Q'$ such that $u \circ f = f'$. This equality with the monotone surjection condition of $f$ and $f'$ force that $u$ is unique. Therefore, $\Fib_{d_1}(P)$ is discrete and $\mathbf{K}^{\circ}$ is locally discrete. Furthermore, the discrete groupoid $\Fib_{d_1}(P)$ is finite since we have a finite number of contractions whose source is $P$ by the finite condition of $P$. Therefore, $\mathbf{K}^{\circ}$ is locally finite.

\item $\mathbf{K}^{\circ}$ is of locally finite length. Indeed, the fibre of $P$ along $\mathbf{K}^{\circ}_n \actto \mathbf{K}^{\circ}_1$ has no degenerate simplices for $n$ greater than the cardinality of $P$.
\end{itemize}
\end{proof}

The decomposition space $\mathbf{K}$ has a monoidal structure given by disjoint union.
Recall $\mathbf{K}_n$ is the groupoid of families of $(n-1)$-chains of contractions. The disjoint union of two such families is just the family whose components are the objects of the families index by the disjoint union of the two index sets. This clearly defines a simplicial map $+_{\mathbf{K}} \colon \mathbf{K} \times \mathbf{K} \to \mathbf{K}$. So $\mathbf{K}$ is a monoidal decomposition space if the map $+_{\mathbf{K}}$ is culf \cite[\S 9]{GTK1}.

\begin{propo}\label{propo:k is monoidal}
The map $+_{\mathbf{K}} \colon \mathbf{K} \times \mathbf{K} \to \mathbf{K}$ is culf.
\end{propo}

\begin{proof}
By Lemma \ref{lemma culfcondition for ds}, the map $+_{\mathbf{K}}$ is culf, if the diagram
\[\begin{tikzcd}
	{\mathbf{K}_2 \times \mathbf{K}_2} & {\mathbf{K}_1 \times \mathbf{K}_1} \\
	{\mathbf{K}_2} & {\mathbf{K}_1}
	\arrow["{+_\mathbf{K}}"', from=1-1, to=2-1]
	\arrow["{d_1}"', from=2-1, to=2-2]
	\arrow["{d_1}", from=1-1, to=1-2]
	\arrow["{+_{\mathbf{K}}}", from=1-2, to=2-2]
\end{tikzcd}\]
is a pullback since $\mathbf{K}$ is a decomposition space. But this is clear: a pair of families of contractions (an object in ${\mathbf{K}_2 \times \mathbf{K}_2}$) can be uniquely reconstructed if we know what the two source families of posets are (an object in ${\mathbf{K}_1 \times \mathbf{K}_1}$) and we know how the disjoint union is contract (an object in ${\mathbf{K}_2}$). This is subject to identifying the disjoint union of those two source families of posets with the source of the disjoint union of the two families of contractions (which is to say that the data agree down in $\mathbf{K}_1$).
\end{proof}

Since $\mathbf{K}$ is a monoidal decomposition space, it follows that the resulting incidence coalgebra is also a bialgebra \cite[\S 9]{GTK1}.

\section{The decomposition space of admissible maps}\label{section:decomposition AP}

In this section we will introduce the decomposition space of admissible maps $\mathbf{A}$.

\begin{defi}\cite[\S 2.2]{AIF_2017__67_3_911_0} \label{definition:admissible maps}
Let $T'\rightarrow T$ be an identity-on-objects monotone  map between two preorders. We define $T/T'$ to be the preorder with the same objects as $T$ and $T'$, and the preorder is defined by closing the relation $\mathcal{R}$ by transitivity:
$$x\mathcal{R}y\Longleftrightarrow (x\le_T y \;\text{or}\; y\le_{T'}x).$$
\end{defi}
In other words, we obtain $T/T'$ by inverting those arrows of $T$ that also belong to $T'$, and then closing by transitivity. Later in this section, we will use the following more categorical characterization of quotients.

The underlying set of a preorder $T$ will be denoted as $\underline{T}$, and $\overline{T}$ denotes the discrete preorder of connected components of $T$. We have a natural map $\comp \colon T \to \overline{T}$ that sends each object to its corresponding connected component.

\begin{lem}\label{lemma: definition T/T pushout}
Let $T' \rightarrow T$ be an identity-on-objects monotone  map between two preorders. Then $T/T'$ is the pushout in category of preorders
\begin{center}
\begin{tikzcd}
T'\arrow[d] \arrow[r] \drpushout& T \arrow[d] \\
T'/T' \arrow[r]& T/T'.                 
\end{tikzcd}
\end{center}
\end{lem}

\begin{proof}
Since the maps are identity-on-objects then $\underline{T/T'} = \underline{T}$. This means that at the level of sets the square is a pushout. Consider the preorder given by closing the relation $\mathcal{R}$ by transitivity:
$$x\mathcal{R}y\Longleftrightarrow (x\le_T y \;\text{or}\; y\le_{T'}x).$$
It is straightforward to see that the diagram commutes at the level of preorders. In case that we have another preorder $<'$ in $T/T'$ that made the square commutes, the monotonicity of the maps force that $<'$ is equal to the transitive closure of $\mathcal{R}$.
\end{proof}

\begin{defi}\cite[Definition 2.2.]{AIF_2017__67_3_911_0}\label{defi:admissible} A map of preorders $T'\rightarrow T$ is \emph{admissible} if it satisfies the following:
\begin{enumerate}
    \item it is the identity-on-objects,
    \item for every connected sub-preorder $Y$ in $T'$, we have that $T_{|\underlyingset{Y}}= Y$,
    \item $x\sim_{T/T'}y$ if and only if $x\sim_{T'/T'}y$.
\end{enumerate}
\end{defi}

\noindent The following picture gives an illustration of an admissible map of preorders:

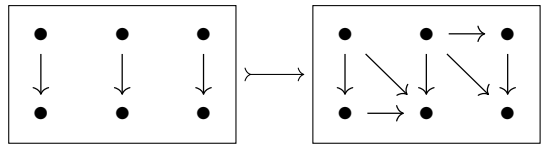
\begin{center}
\begin{tikzpicture}[yscale=1,xscale=1]
\node[draw] (box1) at (0,0){%
  \begin{tikzcd}[column sep=small,row sep=scriptsize]
  	\bullet & \bullet & \bullet \\
	\bullet & \bullet & \bullet
	\arrow[from=1-1, to=2-1]
	\arrow[from=1-2, to=2-2]
	\arrow[from=1-3, to=2-3]
  \end{tikzcd}
};
\node[draw] (box2) at (4,0){%
\begin{tikzcd}[column sep=small,row sep=scriptsize]
	\bullet & \bullet & \bullet \\
	\bullet & \bullet & \bullet
	\arrow[from=2-1, to=2-2]
	\arrow[from=1-1, to=2-1]
	\arrow[from=1-1, to=2-2]
	\arrow[from=1-2, to=2-2]
	\arrow[from=1-2, to=2-3]
	\arrow[from=1-2, to=1-3]
	\arrow[from=1-3, to=2-3]
\end{tikzcd}
};

\draw[>->, shorten <=3pt, shorten >=3pt] (box1) -- (box2);
\end{tikzpicture}
\end{center}

\begin{lem}\label{lemma: admissble map stable pullback}
In the category of preorders, admissible maps are stable under pullback along identity-on-objects monotone maps.
\end{lem}

\begin{proof}
Let $\beta \colon V' \rightarrowtail V$ be an admissible map and let $f \colon T \to V$ be an identity-on-objects monotone map. Consider the pullback diagram
\[\begin{tikzcd}
	{V' \times_V T} & T \\
	{V'} & V.
	\arrow["\beta"', tail, from=2-1, to=2-2]
	\arrow["f", from=1-2, to=2-2]
	\arrow["{\pi_{V'}}"', from=1-1, to=2-1]
	\arrow["{\pi_T}", from=1-1, to=1-2]
	\arrow["\lrcorner"{anchor=center, pos=0.125}, draw=none, from=1-1, to=2-2]
\end{tikzcd}\]
It is easy to see that $\pi_T$ is an identity-on-objects monotone map. For a connected sub-preorder $Y$ of ${V' \times_V T}$, we have that $\pi_{V'}(Y)$ is a connected sub-preorder of $V'$ since $\pi_{V'}$ is a projection map. By the admissible condition of $\beta$, it follows that $\beta(\pi_{V'})$ is a connected sub-preorder of $V$ and $V\vert_{Y} = Y$. Combining this with the commutative of the square and the identity-on-objects monotone condition of the arrows it follows that $\pi_T(Y)$ is a connected sub-preorder of $T$ and $T\vert_{Y} = Y$. To prove that $\pi_{T}$ is admissible, all that remains is to check $x\sim_{T/(V' \times_V T)}y$ if and only if $x\sim_{(V' \times_V T)/(V' \times_V T)}y$, but the commutative condition of the square implies that $x\sim_{T/(V' \times_V T)}y$ if $x\sim_{V/V'}y$ and $x\sim_{(V' \times_V T)/(V' \times_V T)}y$ if  $x\sim_{V'/V'}y$. In other words $x\sim_{V/V'}y$ if and only if $x\sim_{V'/V'}y$ but this is true as consequence of the admissible condition of $\beta$. Hence $\pi_T$ is admissible. 
\end{proof}

The posetification of a preorder $T$ will be denoted as $\posetify{T}$. Let $\collapse{P}$ denote the discrete poset of connected components of a poset $P$, which is the same as $\posetify{P/P}$.

\begin{lem}\label{lem:admissible pullback preorders}
An identity-on-objects monotone map of preorders $T'\rightarrow T$ is admissible if and only if the pushout square
\begin{center}
\begin{tikzcd}
T'\arrow[d] \arrow[r] \drpushout& T \arrow[d] \\
\collapse{V'} \arrow[r]& V,                 
\end{tikzcd}
\end{center}
satisfies the following properties:
\begin{enumerate}
    \item $\underlyingset{V}=\collapse{V'}$
    \item it is also a pullback square.
\end{enumerate}
\end{lem}

\begin{proof}
First, note that condition 1 is equivalent to condition 3 of Definition~\ref{defi:admissible}. Indeed, if $x\sim_{T/T'}y\Leftrightarrow x\sim_{T'/T'}y$ then necessarily $V=\posetify{T/T'}$ and $\collapse{V'}=\posetify{T'/T'}$ have the same objects, and vice-versa, because the objects of $V$ are given by the equivalence classes of objects of $T/T'$, and the objects of $\collapse{T'}$ are given by the equivalence classes of objects of $T'/T'$. On the other hand, assuming these two conditions hold, condition 2 is equivalent to condition 2 of Definition~\ref{defi:admissible}. Indeed, the connected components of $T'$ are the preimages of $\underlyingset{V}$, and the fact that the square is a pullback is equivalent to the fact that the map $T'\rightarrow T$ is full on each connected component of $T'$, which is precisely condition 2 of Definition~\ref{defi:admissible}. 
\end{proof}

\begin{blanko}
{Fauvet--Foissy--Manchon Hopf algebra of finite topologies and admissible maps}\label{subsec:comparation FFM with K}
\end{blanko}

Fauvet, Foissy, and Manchon (\cite{AIF_2017__67_3_911_0}, \S 3.1) introduced the notion of quotient of a topology $\mathcal{T}$ on a finite set $X$ by another topology $\mathcal{T}'$ finer than $\mathcal{T}$. The quotient topology $\mathcal{T}/\mathcal{T}'$ thus obtained lives on the same set. Furthermore, they introduced the relation ${\stackMath\mathbin{\stackinset{c}{0ex}{c}{0ex}{<}{\bigcirc}}}$ on the topologies on $X$ defined by $\mathcal{T}' {\stackMath\mathbin{\stackinset{c}{0ex}{c}{0ex}{<}{\bigcirc}}} \mathcal{T}$ if and only if $\mathcal{T}'$ is finer than $\mathcal{T}$ and fulfills the technical condition of $\mathcal{T}$-admissibility. This enabled them to give the internal coproduct 
$$\Delta(\mathcal{T}) = \sum_{\mathcal{T}' {\stackMath\mathbin{\stackinset{c}{0ex}{c}{0ex}{<}{\bigcirc}}} \mathcal{T}} \mathcal{T}' \otimes \mathcal{T}/\mathcal{T}'$$
Since there is a natural bijection between topologies and quasi-orders on a finite set $X$, Fauvet, Foissy, and Manchon also expressed the $\mathcal{T}$-admissibility in the context of preorders. This corresponds to the one we use in this section (\ref{definition:admissible maps}). The above coproduct is just rewritten in the context of preorders as:  
$$\Delta(T) = \sum_{T' \rightarrowtail T} T' \otimes T/T'.$$ 
We shall see in \ref{lem:incidence coalgebra of A} that this coproduct corresponds to the coproduct $\Delta_{\mathbf{A}}$ of the incidence coalgebra of the decomposition space of admissible maps $\mathbf{A}$.

\begin{blanko}
{The decomposition space $\mathbf{A}$}\label{subsec:definition A}
\end{blanko}

A \emph{groupoid preorder} is a preorder where all its morphisms are invertible. Given a finite preorder $T$, we denote by $T^{\invp}$ the groupoid preorder that contains the same objects that $T$ but only the invertible morphisms of $T$.

We describe a pseudo simplicial groupoid (\ref{proposition: APr simplicial}) of admissible maps which we call $\mathbf{A}$. Let $\mathbf{A}_n$ denote the groupoid of $(n-1)$-chains of admissible maps between non-empty finite preorders. Thus, $\mathbf{A}_2$ is the groupoid of admissible maps and $\mathbf{A}_1$ is the groupoid of finite preorders whose underlying sets are ordinals and monotone bijections. $\mathbf{A}_0$ is the groupoid whose objects are the groupoid preorders. Face maps are given by:

\begin{itemize}
\item For $n \geq 2$, the bottom face map $d_\perp \colon \mathbf{A}_n \rightarrow \mathbf{A}_{n-1}$ is defined as follows: given $(n-1)$-chain of admissible maps
\begin{center}
\begin{tikzcd}
T_0 \arrow[r, tail] & T_1 \arrow[r, tail] & T_2 \arrow[r,tail] & \cdots \arrow[r,tail] & T_{n-2} \arrow[r, tail] & T_{n-1},
\end{tikzcd}
\end{center}
for each poset in the chain, we can form the quotient $T_i/T_0$ by Lemma \ref{lemma: definition T/T pushout}. We end up with a $(n-2)$-chain of admissible maps
\begin{center}
\begin{tikzcd}
T_1/T_0 \arrow[r, tail] & T_2/T_0 \arrow[r, tail] & T_3/ T_0 \arrow[r,tail] & \cdots \arrow[r,tail] & T_{n-2}/ T_0 \arrow[r, tail] & T_{n-1} / T_0;
\end{tikzcd}
\end{center}
 
\item $d_\perp \colon \mathbf{A}_1 \rightarrow \mathbf{A}_{0}$ sends a preorder $T$ to $T^{\invp}$.

\item $d_1$ forgets the first preorder in the chain;

\item $d_i \colon \mathbf{A}_n \rightarrow \mathbf{A}_{n-1}$ composes the $i$th and $(i+1)$th admissible map, for $1 < i < n-1$; 

\item $d_{\top}$ forgets the last preorder in the chain.
\end{itemize}
 
Degeneracy maps are given by:

\begin{itemize}
\item $s_\perp \colon \mathbf{A}_n \rightarrow \mathbf{A}_{n+1}$ is given by appending with the map whose source is the underlying groupoid preorder of the first preorder of the chain.

\item $s_i \colon \mathbf{A}_n \rightarrow \mathbf{A}_{n+1}$ inserts an identity arrow at object number $i$, for $0 < i \leq n$.
\end{itemize}

The simplicial identities will be verified in \ref{proposition: APr simplicial}. Let $\mathbb{A}$ denote the category of finite preorders and admissible maps. 

\begin{propo}\label{proposition:dec perp AP is Segal}
We have an equality $\Dec_\perp \mathbf{A} = \fatnerve \mathbb{A}$.
\end{propo}
  
Proposition \ref{proposition:dec perp AP is Segal} implies the compatibility of the face and degeneracy maps in $\mathbf{A}$ except the face maps $d_\perp$. We need the following result to verify the simplicial identities for $d_\perp$.

\begin{lem}\label{lemma: pushouts for 2-admissible maps}
Let $T_0 \rightarrowtail T_1$ and $T_1 \rightarrowtail T_2$ be admissible maps. The following diagram
\[\begin{tikzcd}
	{} & {T_0} & {T_1} & {T_2} \\
	& {T_0/T_0} & {T_1/T_0} & {T_2/T_0} \\
	&& {T_1/T_1} & {T_2/ T_1}
	\arrow[tail, from=1-2, to=1-3]
	\arrow[tail, from=1-3, to=1-4]
	\arrow[from=1-3, to=2-3]
	\arrow[from=1-4, to=2-4]
	\arrow[from=2-4, to=3-4]
	\arrow[tail, from=2-3, to=2-4]
	\arrow[from=2-3, to=3-3]
	\arrow[from=1-2, to=2-2]
	\arrow[tail, from=2-2, to=2-3]
	\arrow[tail, from=3-3, to=3-4]
\end{tikzcd}\]
commutes and the squares are pushouts. 
\end{lem}

\begin{proof}
By the prism Lemma, the outer diagram
\[\begin{tikzcd}
	{T_1} & {T_2} \\
	{T_1/T_0} & {T_2/T_0} \\
	{\frac{(T_1/T_0)}{(T_1/T_0)}} & {\frac{(T_2/T_0)}{(T_1/T_0)}}
	\arrow[tail, from=1-1, to=1-2]
	\arrow[from=1-1, to=2-1]
	\arrow[from=1-2, to=2-2]
	\arrow[from=2-1, to=3-1]
	\arrow[from=2-2, to=3-2]
	\arrow[tail, from=2-1, to=2-2]
	\arrow[tail, from=3-1, to=3-2]
	\arrow["\lrcorner"{anchor=center, pos=0.125, rotate=180}, draw=none, from=2-2, to=1-1]
	\arrow["\lrcorner"{anchor=center, pos=0.125, rotate=180}, draw=none, from=3-2, to=2-1]
\end{tikzcd}\]
is a pushout. This combining with the fact that ${\frac{(T_1/T_0)}{(T_1/T_0)}} = T_1/T_1$ implies that 
${\frac{(T_2/T_0)}{(T_1/T_0)}}$ is the pushout of $T_1 \rightarrowtail T_2$ along $T_1 \to T_1/T_1$ but $T_2/T_1$ is also a pushout over the same diagram. Therefore, ${\frac{(T_2/T_0)}{(T_1/T_0)}} \cong T_2/T_1$.
\end{proof}

\begin{propo}\label{proposition: APr simplicial}
The groupoids $\mathbf{A}_n$ and the degeneracy and face maps given above form a pseudosimplicial groupoid $\mathbf{A}$.
\end{propo}

\begin{proof}
We only need to verify the simplicial identities that involve $d_\perp$. The others follow from the fact that $\Dec_\perp \mathbf{A} = \fatnerve \mathbb{A}$ (\ref{proposition:dec perp AP is Segal}). Let us prove that the following diagram
\[\begin{tikzcd}
	{\mathbf{A}_3} & {\mathbf{A}_2} \\
	{\mathbf{A}_2} & {\mathbf{A}_1}
	\arrow["{d_\perp}", from=1-1, to=1-2]
	\arrow["{d_\perp}", from=1-2, to=2-2]
	\arrow["{d_1}"', from=1-1, to=2-1]
	\arrow["{d_\perp}"', from=2-1, to=2-2]
\end{tikzcd}\]
commutes, for the other cases the proof follows the same arguments but the notation becomes much heavier. Let $T_0 \rightarrowtail T_1 \rightarrowtail T_2$ be an object in $\mathbf{A}_3$. It easy to check that $d_\perp d_\perp (T_0 \rightarrowtail T_1 \rightarrowtail T_2) = \frac{T_2/T_0}{T_1/T_0}$ and 
$d_\perp d_1 (T_0 \rightarrowtail T_1 \rightarrowtail T_2) = T_2/T_1$. So the square commutes when $\frac{T_2/T_0}{T_1/T_0} \cong T_2/T_1$, but in the proof of Lemma \ref{lemma: pushouts for 2-admissible maps} we have this isomorphism. 
\end{proof}


\begin{lem}\label{lemma: surjectivity of d0 for AP}
Suppose we have admissible maps $T' \rightarrowtail T$ and $Q \rightarrowtail T/T'$, there exists a unique preorder $P$ and admissible maps $T' \rightarrowtail P$ and $P \rightarrowtail T$ such that the diagram
\[\begin{tikzcd}
	{T'} & P & T \\
	& Q & {T/T'}
	\arrow[dotted, tail, from=1-1, to=1-2]
	\arrow[dotted, tail, from=1-2, to=1-3]
	\arrow[tail, from=2-2, to=2-3]
	\arrow[dotted, from=1-2, to=2-2]
	\arrow[from=1-3, to=2-3]
	\arrow[bend left = 30, tail, from=1-1, to=1-3]
\end{tikzcd}\]
commutes.
\end{lem}

\begin{proof}
Put $P = Q \times_{T/T'} T$. Since admissible maps are stable under pullbacks, the map $P \to T$ is admissible. Furthermore, we have a natural map from $T'/T' \to Q$ which is admissible since $T'/T' \rightarrowtail T/T'$ and $Q \rightarrowtail T/T'$ are admissible. Therefore, the outer diagram  
\[\begin{tikzcd}
	{T'} & P & T \\
	{T'/T'} & Q & {T/T'}
	\arrow[from=1-3, to=2-3]
	\arrow[from=1-2, to=2-2]
	\arrow[tail, from=2-2, to=2-3]
	\arrow[tail, from=1-2, to=1-3]
	\arrow[from=1-1, to=2-1]
	\arrow[bend right = 25, tail, from=2-1, to=2-3]
	\arrow[bend left = 25, tail, from=1-1, to=1-3]
	\arrow[tail, from=2-1, to=2-2]
	\arrow[dotted, from=1-1, to=1-2]
	\arrow["\lrcorner"{anchor=center, pos=0.125}, draw=none, from=1-2, to=2-3]
\end{tikzcd}\]
commutes. The dotted arrow then exists by the pullback property of $P$. By the prism Lemma, the left square is a pullback since the outer diagram is a pullback by Lemma \ref{lem:admissible pullback preorders}. The map $T' \to P$ is admissible since admissible maps are stable under pullbacks and 
$T'/T' \rightarrowtail Q$ is admissible.
\end{proof}

\begin{propo}\label{proposition: AP is ds}
The pseudosimplicial groupoid $\mathbf{A}$ is a decomposition space.
\end{propo}

\begin{proof}
We will prove that the following diagrams are pullbacks for $0 < i < n$:
\begin{center}
\begin{tikzcd}
\mathbf{A}_{n+1} \arrow[white]{dr}[black, description]{(1)}  \arrow[r, "d_{i+1}"]\arrow[d, "d_\bot"']& \mathbf{A}_n \arrow[d, "d_\bot"] \\
\mathbf{A}_n \arrow[r, "d_i"']& \mathbf{A}_{n-1}
\end{tikzcd} \; \; \;
\begin{tikzcd}
\mathbf{A}_{n+1} \arrow[white]{dr}[black, description]{(2)}  \arrow[r, "d_{i}"]\arrow[d, "d_\top"']& \mathbf{A}_n \arrow[d, "d_\top"] \\
\mathbf{A}_n \arrow[r, "d_i"']& \mathbf{A}_{n-1}.
\end{tikzcd}
\end{center}
Let us prove it for $n=2$. For greater $n$ the proof is completely analogous. The square $(2)$ is a pullback as a consequence of Lemma \ref{proposition:dec perp AP is Segal}. To prove that the square $(1)$ is a pullback we will use Lemma \ref{lemma:pullbackfibres}. This means that $(1)$ is a pullback if and only if for each object
$\alpha \colon T' \rightarrowtail T$ in $\mathbf{A}_2$, the map $d_\perp \colon \Fib_{\alpha} d_2 \rightarrow \Fib_{d_\perp \alpha } d_1$ is an equivalence of groupoids. For an object $V \rightarrowtail T/T'$ in $\Fib_{d_\perp \alpha } d_1$, Lemma \ref{lemma: surjectivity of d0 for AP} gives the following commutative diagram
\[\begin{tikzcd}
	{T'} & V & T \\
	{T'/T'} & {V'} & {T/T'}
	\arrow[from=1-3, to=2-3]
	\arrow[from=1-2, to=2-2]
	\arrow[tail, from=2-2, to=2-3]
	\arrow[tail, from=1-2, to=1-3]
	\arrow[from=1-1, to=2-1]
	\arrow[bend right = 20, tail, from=2-1, to=2-3]
	\arrow["\alpha", bend left = 20, tail, from=1-1, to=1-3]
	\arrow[tail, from=2-1, to=2-2]
	\arrow[from=1-1, to=1-2]
	\arrow["\lrcorner"{anchor=center, pos=0.125}, draw=none, from=1-2, to=2-3]
\end{tikzcd}\]
where the horizontal arrows are admissible maps. The commutativity of the diagram implies that $T' \rightarrowtail V \rightarrowtail T$ is an object in  $\Fib_{\alpha} d_2$ and hence $d_\perp$ is surjective on objects. For the full condition of $d_\perp$, let $p' \colon V' \rightarrowtail R'$ be admissible map such that the lower square
\[\begin{tikzcd}[column sep=small,row sep=small]
	T && V && {T'} \\
	& T && R && {T'} \\
	&& {V'} && {T/T'} \\
	&&& {R'} && {T/T'}
	\arrow["\identity"', from=1-1, to=2-2]
	\arrow[tail, from=1-1, to=1-3]
	\arrow[tail, from=1-3, to=1-5]
	\arrow["\alpha", bend left = 20, from=1-1, to=1-5]
	\arrow[from=1-3, to=3-3]
	\arrow[from=2-4, to=4-4]
	\arrow["{p'}"', from=3-3, to=4-4]
	\arrow[tail, from=4-4, to=4-6]
	\arrow["\identity", from=1-5, to=2-6]
	\arrow["p"{description}, dotted, from=1-3, to=2-4]
	\arrow[tail, from=2-2, to=2-4]
	\arrow[tail, from=2-4, to=2-6]
	\arrow[from=1-5, to=3-5]
	\arrow[from=2-6, to=4-6]
	\arrow[tail, from=3-3, to=3-5]
	\arrow["\identity", from=3-5, to=4-6]
\end{tikzcd}\]
commutes. This means that $\lbrace p', \identity_{T/T'}\rbrace$ is a morphism in $\Fib_{d_\perp \alpha} d_1$. Lemma \ref{lemma: surjectivity of d0 for AP} applied to $V'$ and $R'$ gives the top admissible maps. The pullback property of $R$ gives the dotted arrow $p \colon V \to R$. The commutativity of the diagram follows from the pullback condition of $V$ and $R$, and therefore $\lbrace\identity_T,  p,  \identity_{T'}\rbrace$ is a morphism in $\Fib_{\alpha} d_2$. This implies that $d_\perp$ is full. The faithful condition of $d_\perp$ is straightforward to check using Lemma \ref{lemma: surjectivity of d0 for AP}. Indeed, let $\lbrace\identity_T,  p,  \identity_{T'}\rbrace$ and $\lbrace\identity_T,  q,  \identity_{T'}\rbrace$ be morphisms in $\Fib_{\alpha} d_2$ such that $d_\perp \lbrace\identity_T,  p,  \identity_{T'}\rbrace = d_\perp \lbrace\identity_T,  q,  \identity_{T'}\rbrace$. This means that the following diagram commutes
\[\begin{tikzcd}[column sep=small,row sep=small]
	T && V && {T'} \\
	& T && R && {T'} \\
	&& {V'} && {T/T'} \\
	&&& {R'} && {T/T'}
	\arrow["\identity"', from=1-1, to=2-2]
	\arrow[tail, from=1-1, to=1-3]
	\arrow[tail, from=1-3, to=1-5]
	\arrow["\alpha", bend left= 20, from=1-1, to=1-5]
	\arrow[from=1-3, to=3-3]
	\arrow[from=2-4, to=4-4]
	\arrow["{p'}"', from=3-3, to=4-4]
	\arrow[tail, from=4-4, to=4-6]
	\arrow["\identity", from=1-5, to=2-6]
	\arrow["p"', shift right=1, from=1-3, to=2-4]
	\arrow[tail, from=2-2, to=2-4]
	\arrow[tail, from=2-4, to=2-6]
	\arrow[from=1-5, to=3-5]
	\arrow[from=2-6, to=4-6]
	\arrow[tail, from=3-3, to=3-5]
	\arrow["\identity"', from=3-5, to=4-6]
	\arrow["q", shift left=1, from=1-3, to=2-4]
\end{tikzcd}\]
By the pullback property of $R$, it follows that $p = q$.  
\end{proof}

\begin{propo}\label{proposition: AP is complete}
The decomposition space $\mathbf{A}$ is complete.
\end{propo}

\begin{proof}
Let $T$ be an object in $\mathbf{A}_1$. Consider the pullback diagram
\begin{center}
\begin{tikzcd}
\Fib_{s_0}  T  \arrow[d] \arrow[r] \drpullback & 1 \arrow[d, "\ulcorner T \urcorner"] \\
\mathbf{A}_0 \arrow[r, "s_0"']                               & \mathbf{A}_1.                                                    \end{tikzcd}
\end{center}
We have that $\Fib_{s_0} T$ is empty if $T$ is not a groupoid preorder. In case $T$ is a groupoid preorder, we have that $\Fib_{s_0} T \simeq T^{\invp}$. Therefore, $s_0 \colon \mathbf{A}_0 \rightarrow \mathbf{A}_1$ is a monomorphism.
\end{proof}


Recall that the comultiplication of the incidence coalgebra of $\mathbf{A}$ is given by the formula:
$$\Delta_{\mathbf{A}}(\alpha) = \sum_{\binom{\alpha \in \mathbf{A}_2}{d_1(\alpha) = T)}} d_2(\alpha) \otimes d_0(\alpha)$$
which means that we sum over all admissible maps $\alpha \colon T' \rightarrowtail T$ that have target $T$ and return $T'$ and $T/T'$. But this is precisely the comultiplication of the Fauvet--Foissy--Manchon coalgebra of finite topological spaces after using the bijection between finite topological spaces and finite preorders (see \S \ref{subsec:comparation FFM with K}). Hence, we have the following result.

\begin{lem}\label{lem:incidence coalgebra of A}
The incidence coalgebra $\Delta_{\mathbf{A}}$ is the Fauvet--Foissy--Manchon coalgebra of finite topological spaces. 
\end{lem}

\begin{rema}
There is a natural bijection between finite $T_0$-topological spaces and finite posets. It would be interesting to explore the constructions given above from a topological point of view. The starting point would be to translate the notion of contraction between posets into a contraction between finite $T_0$-spaces. We will decide to leave this point of view for future work.  
\end{rema}

\section{Admissible maps and contractions}\label{subsec:admissible maps and contractions}

In this section, we will relate the notions of admissible maps of preorders (due to \cite{AIF_2017__67_3_911_0}) and of contractions for posets through a culf map (see \ref{proposition:AP to K is culf}). This provides a deeper explanation of both classes of maps and simultaneously allows us to relate the Fauvet--Foissy--Manchon bialgebra of finite topological spaces with the Calaque--Ebrahimi--Fard--Manchon bialgebra of rooted trees (see \S \ref{subsec:comparation FFM with K}). To relate admissible maps between preorders and contractions of posets, it is necessary to introduce some results.

\begin{defi}\label{definition: contraction preorders}
A monotone map of preorders $f \colon T \rightarrow V$ is a \emph{contraction} if:
\begin{enumerate}
    \item its is identity in objects;
    \item for each $x \in V$, the fibre $f^{-1}(x^{\simeq})$ is a connected sub-preorder of $T$, where $x^{\simeq}$ denotes the subpreorder of objects equivalent to $x$ in $V$;
    \item for any cover $a \lessdot a'$ in $V$, there exists a cover $t \lessdot t'$ in $T$ such that $f(t) = a$ and $f(t') = a'$.
\end{enumerate}
\end{defi}

Given a preorder $T$ there is a canonical contraction $T \to T/T$. We will show that for any admissible map $T' \rightarrowtail T$, we can construct a contraction $T \twoheadrightarrow T/T'$. An illustration of the canonical contraction from an admissible map is given by the following pushout diagram

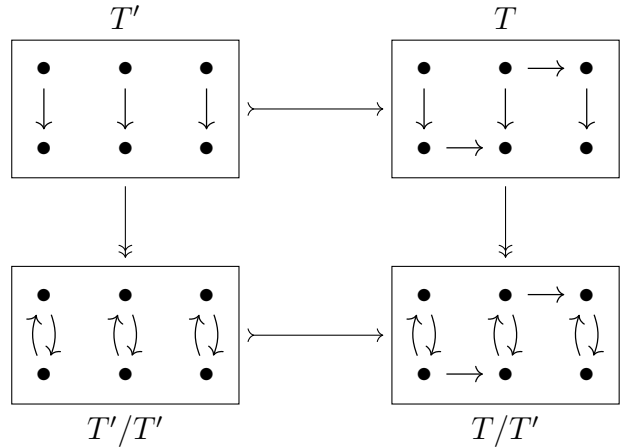
\begin{center}
\begin{tikzpicture}[yscale=1,xscale=1]
\node[draw, label={$T'$}] (box1) at (0,3){%
  \begin{tikzcd}[column sep=small,row sep=scriptsize]
  	\bullet & \bullet & \bullet \\
	\bullet & \bullet & \bullet
	\arrow[from=1-1, to=2-1]
	\arrow[from=1-2, to=2-2]
	\arrow[from=1-3, to=2-3]
  \end{tikzcd}
};
\node[draw, label={$T$}] (box2) at (5,3){%
\begin{tikzcd}[column sep=small,row sep=scriptsize]
	\bullet & \bullet & \bullet \\
	\bullet & \bullet & \bullet
	\arrow[from=1-1, to=2-1]
	\arrow[from=1-2, to=2-2]
	\arrow[from=1-3, to=2-3]
	\arrow[from=2-1, to=2-2]
	\arrow[from=1-2, to=1-3]
\end{tikzcd}
};

\node[draw, label=below:{$T'/T'$}] (box3) at (0,0){
\begin{tikzcd}[column sep=small,row sep=scriptsize]
	\bullet & \bullet & \bullet \\
	\bullet & \bullet & \bullet
	\arrow[bend left = 20, from=1-1, to=2-1]
	\arrow[bend left = 20, from=1-2, to=2-2]
	\arrow[bend left = 20, from=1-3, to=2-3]
	\arrow[bend left = 20, from=2-1, to=1-1]
	\arrow[bend left = 20, from=2-2, to=1-2]
	\arrow[bend left = 20, from=2-3, to=1-3]
\end{tikzcd}
};

\node[draw, label=below:{$T/T'$}] (box4) at (5,0){
\begin{tikzcd}[column sep=small,row sep=scriptsize]
	\bullet & \bullet & \bullet \\
	\bullet & \bullet & \bullet
	\arrow[bend left = 20, from=1-1, to=2-1]
	\arrow[bend left = 20, from=1-2, to=2-2]
	\arrow[bend left = 20, from=1-3, to=2-3]
	\arrow[bend left = 20, from=2-1, to=1-1]
	\arrow[bend left = 20, from=2-2, to=1-2]
	\arrow[bend left = 20, from=2-3, to=1-3]
	\arrow[from=2-1, to=2-2]
	\arrow[from=1-2, to=1-3]
\end{tikzcd}
};

\draw[>->, shorten <=3pt, shorten >=3pt] (box1) -- (box2);
\draw[->>, shorten <=3pt, shorten >=3pt] (box1) to (box3);
\draw[->>, shorten <=3pt, shorten >=3pt] (box2) -- (box4);
\draw[>->, shorten <=3pt, shorten >=3pt] (box3) to (box4);
\end{tikzpicture}
\end{center}

Recall that $\posetify{T}$ is the posetification of a preorder $T$ and let $\posy \colon T \to \posetify{T}$ denote the posetification monotone map.

\begin{lem} 
Let $T$ be a preorder. There is a bijection between the set of admissible maps onto $T$ and the set of admissible maps onto $\posetify{T}$.
\end{lem}

\begin{proof} The conditions of Definition~\ref{defi:admissible} imply that if $x\sim_T y$ then $x\sim_{T'}y$, so that contract $T$ to $\posetify{T}$ does not have any effect in the set of admissible maps.
\end{proof}

\begin{lem}\label{lemma:posetification of contraction is a contraction}
Let $f \colon T \to V$ be a contraction between preorders. Then $\posetify{f} \colon \posetify{T} \to \posetify{V}$ is a contraction of posets.
\end{lem}

\begin{proof}
The map $\posetify{f} \colon \posetify{T} \to \posetify{V}$ is monotone since for each $[t] \in \posetify{T}$, we have that $\posetify{f}([t]) = [f(t)]$. Furthermore, the diagram
\[\begin{tikzcd}
	T & \posetify{T} \\
	V & \posetify{V}
	\arrow["f"', two heads, from=1-1, to=2-1]
	\arrow["{\posy_T}", from=1-1, to=1-2]
	\arrow["{\posy_V}"', from=2-1, to=2-2]
	\arrow["\posetify{f}", from=1-2, to=2-2]
	\arrow["{(1)}"{description}, draw=none, from=1-1, to=2-2]
\end{tikzcd}\]
commutes. The commutativity of the square together with the monotone surjection condition of $f$ and ${\posy_T}$ and ${\posy_V}$ implies that $\posetify{f}$ is a monotone surjection. Let's prove that for each $[v] \in \posetify{V}$, the fibre $\posetify{f}_{[v]}$ is a connected convex subposet of $\posetify{T}$. Let $[a]$ and $[b]$ be objects in $\posetify{f}_{[v]}$. This implies that 
$$ f(a) \sim_V v \sim_V f(b),$$
and therefore $a$ and $b$ are objects in $f^{-1}(v^{\simeq})$. Furthermore, $a$ and $b$ are connected since $f^{-1}(v^{\simeq})$ is a connected preorder in $T$ by the contraction condition of $f$. Therefore, $[a]$ and $[b]$ are connected.
Let's prove that $\posetify{f}$ lifts covers: let $[v] \lessdot [v']$ be a cover in $\posetify{V}$. Since the posetification map respects covers, we have that $v \lessdot v'$ in $V$. Recall that $f$ lifts covers since it is a contraction. This means that there exists a cover $t \lessdot t'$ in $T$ such that $f(t) = v$ and $f(t') = v$. Applying the posetification map to $t \lessdot t'$, we have that $[t] \lessdot [t']$ in $\posetify{T}$, and by the commutative of the square $(1)$, we get that $\posetify{f}([t]) = [v]$ and $\posetify{f}([t']) = [v]$.
\end{proof}

\begin{lem}\label{lemma:the posetification T to T/T' is connected}
For an admissible map of preorders $\alpha \colon T' \rightarrowtail T$, the map ${T} \to {T/T'}$ is a contraction.
\end{lem}

\begin{proof}
The map $f \colon T \to T/T'$ is given by the pushout in the category of preorders 
\[\begin{tikzcd}
	{T'} & T \\
	{T'/T'} & {T/T'}.
	\arrow["\alpha", tail, from=1-1, to=1-2]
	\arrow[from=1-1, to=2-1]
	\arrow[tail, from=2-1, to=2-2]
	\arrow["f", from=1-2, to=2-2]
	\arrow["\lrcorner"{anchor=center, pos=0.125, rotate=180}, draw=none, from=2-2, to=1-1]
\end{tikzcd}\]
The map $f$ is clearly identity-on-objects. For each $x \in T/T'$, we have that $f^{-1}(x^{\simeq})$ is connected. Indeed, for each $y \simeq x$ in $T/T'$, we have three possibilities: $y \simeq_T x$ in $T$, or $x <_{T'} y$ or $y <_{T'} x$. By the admissible property of $\alpha$, we have that $x <_T y$ or $y <_T x $ if $x < y$ or $y < x$ in $T'$ since $\alpha$ respects the connected subpreorders from $T'$ to $T$. 
Furthermore, $f$ lifts covers. Indeed, if $x \lessdot x'$ in $T/T'$ then $x < x'$ in $T$ since $f$ only inverts the relations that are in $T'$ and preserves all other relations in $T$. Hence, $f$ is a contraction of preorders. 
\end{proof}

\begin{lem}\label{lemma:bijection adm and contr pre}
Let $T$ be a finite preorder. There is a bijection between the set of admissible maps with target $T$ and the set of contractions with source $T$.
\end{lem}

\begin{proof}
Given an admissible map $T' \rightarrowtail T$, the map $T \twoheadrightarrow T/T'$ is a contraction by Lemma \ref{lemma:the posetification T to T/T' is connected}. For a contraction $T \twoheadrightarrow V$, consider the pullback diagram:
 \[\begin{tikzcd}
	{T \times_V V^{\invp}} & T \\
	{V^{\invp}} & V.
	\arrow[from=2-1, to=2-2]
	\arrow[two heads, from=1-2, to=2-2]
	\arrow[from=1-1, to=2-1]
	\arrow[from=1-1, to=1-2]
	\arrow["\lrcorner"{anchor=center, pos=0.125}, draw=none, from=1-1, to=2-2]
\end{tikzcd}\]
The map $V^{\invp} \to V$ is an identity-on-objects monotone map. Furthermore, this map is canonical admissible since $V^{\invp}$ only contains the invertible morphisms in $V$, which is equivalent to condition $(3)$ of the requirement to be admissible (\ref{defi:admissible}). The condition $(2)$ to be admissible is automatic since we only consider invertible morphism in $V^{\invp}$. Therefore, $V^{\invp} \to V$ is admissible.
Since admissible maps are stable under pullback along identity-on-objects monotone maps (\ref{lemma: admissble map stable pullback}), the map ${T \times_V V^{\invp}} \to T$ is admissible.
\end{proof}



\begin{cons}\label{construction: AP to K}
We will construct a simplicial map $\RAdCon$ from the decomposition space of admissible maps $\mathbf{A}$ to the decomposition space of contractions $\mathbf{K}$. \\
The functor $\RAdCon_1 \colon \mathbf{A}_1 \to \mathbf{K}_1$ sends a preorder $T$ to the family of the connected components $\lbrace \posetify{T}_i \rbrace_{i \in \posetify{{T}/{T}}}$ of the posetification of $T$. 
For $n \geq 2$, let 
\begin{center}
\begin{tikzcd}
T_0 \arrow[r, tail] & T_1 \arrow[r, tail] & T_2 \arrow[r, tail]& \cdots \arrow[r, tail] & T_{n-2} \arrow[r, tail] & T_{n-1}
\end{tikzcd}
\end{center}
be an object in $\mathbf{A}_n$. Consider the following diagram consisting of pushout squares
\begin{equation*}\tag{1}\label{diagram:construction AP 1}
\begin{tikzcd}[column sep=small,row sep=small]
    & {T_0} & {T_1} & \cdots & {T_{n-2}} & {T_{n-1}} \\
	& {T_0/T_0} & {T_1/T_0} & \cdots & {T_{n-2}/T_0} & {T_{n-1}/T_0} \\
	&& {T_1/T_1} & \cdots & \vdots & \vdots \\
	&&&& {T_{n-2}/T_{n-2}} & {T_{n-1}/T_{n-2}} \\
	&&&&& {T_{n-1}/T_{n-1}}.
	\arrow[tail, from=1-2, to=1-3]
	\arrow[tail, from=1-3, to=1-4]
	\arrow[tail, from=1-4, to=1-5]
	\arrow[tail, from=1-5, to=1-6]
	\arrow[tail, from=2-4, to=2-5]
	\arrow[tail, from=3-4, to=3-5]
	\arrow[two heads, from=1-2, to=2-2]
	\arrow[two heads, from=1-3, to=2-3]
	\arrow[two heads, from=1-5, to=2-5]
	\arrow[two heads, from=1-6, to=2-6]
	\arrow[two heads, from=2-5, to=3-5]
	\arrow[two heads, from=2-6, to=3-6]
	\arrow[tail, from=2-3, to=2-4]
	\arrow[tail, from=3-3, to=3-4]
	\arrow[tail, from=4-5, to=4-6]
	\arrow[two heads, from=4-6, to=5-6]
	\arrow[tail, from=2-2, to=2-3]
	\arrow[two heads, from=2-3, to=3-3]
	\arrow[two heads, from=3-5, to=4-5]
	\arrow[two heads, from=3-6, to=4-6]
	\arrow[tail, from=3-5, to=3-6]
	\arrow[tail, from=2-5, to=2-6]
	\arrow[two heads, from=1-4, to=2-4]
	\arrow[two heads, from=2-4, to=3-4]
\end{tikzcd}
\end{equation*}
Since the top horizontal arrows are admissible maps and the canonical map $T_0 \rightarrowtail T_0/T_0$ is a contraction, we have that the horizontal maps are admissible and the verticals are contractions by the stability property of admissible maps and contractions under pushout (\ref{lemma: admissble map stable pullback} and \ref{lemma:the posetification T to T/T' is connected}). 
Taking the posetification of the last column of the diagram (\ref{diagram:construction AP 1}), we obtain the chain 
\begin{equation*}\tag{2}\label{diagram:construction AP 2}
\begin{tikzcd}
	\posetify{{T_{n-1}}} & \posetify{T_{n-1}/T_{0}} &  \posetify{T_{n-1}/T_{1}} & \cdots & \posetify{T_{n-1}/T_{n-2}} & \posetify{T_{n-1}/T_{n-1}}
	\arrow[two heads, from=1-2, to=1-3]
	\arrow[two heads, from=1-1, to=1-2]
	\arrow[two heads, from=1-3, to=1-4]
	\arrow[two heads, from=1-4, to=1-5]
	\arrow[two heads, from=1-5, to=1-6]
\end{tikzcd}
\end{equation*}
of contractions between posets. Since the posets in the chain are not necessarily connected, for each $i \in \posetify{T_{n-1}/T_{n-1}}$, we take the pullback of (\ref{diagram:construction AP 2}) along $\ulcorner i \urcorner \colon 1 \to \posetify{T_{n-1}/T_{n-1}}$ to obtain the diagram
\begin{equation*}\tag{3}\label{diagram:construction AP 3}
\begin{tikzcd}
	\posetify{T_{n-1}} & \posetify{T_{n-1}/T_{0}} & \cdots & \posetify{T_{n-1}/T_{n-2}} & \posetify{T_{n-1}/T_{n-1}} \\
	{(\posetify{T_{n-1}})_i} & {(\posetify{T_{n-1}/T_{0}})_i} & \cdots & {(\posetify{T_{n-1}/T_{n-2}})_i} & 1
	\arrow[two heads, from=1-2, to=1-3]
	\arrow[two heads, from=1-1, to=1-2]
	\arrow[two heads, from=1-3, to=1-4]
	\arrow[two heads, from=1-4, to=1-5]
	\arrow["{\ulcorner i \urcorner }"', from=2-5, to=1-5]
	\arrow[from=2-4, to=2-5]
	\arrow[two heads, from=2-3, to=2-4]
	\arrow[two heads, from=2-2, to=2-3]
	\arrow[two heads, from=2-1, to=2-2]
	\arrow[from=2-4, to=1-4]
	\arrow[from=2-3, to=1-3]
	\arrow[from=2-2, to=1-2]
	\arrow[from=2-1, to=1-1]
	\arrow["\lrcorner"{anchor=center, pos=0.125, rotate=90}, draw=none, from=2-4, to=1-5]
	\arrow["\lrcorner"{anchor=center, pos=0.125, rotate=90}, draw=none, from=2-3, to=1-4]
	\arrow["\lrcorner"{anchor=center, pos=0.125, rotate=90}, draw=none, from=2-1, to=1-2]
\end{tikzcd}
\end{equation*}
where the lower part is a chain of contractions between connected posets since the top part of the diagram is given by  a chain of contractions. 

For $n \geq 2$, the functor $\RAdCon_n \colon \mathbf{A}_n \to \mathbf{K}_n$ sends an $(n-1)$-chain of admissible maps to the family of $(n-1)$-chains of contractions
\begin{equation*}\tag{4}\label{diagram:construction AP 4}
\begin{tikzcd}
	\lbrace (\posetify{{T_{n-1}}})_i & (\posetify{T_{n-1}/T_{0}})_i & (\posetify{T_{n-1}/T_{1}})_i & \cdots & (\posetify{T_{n-1}/T_{n-2}})_i \rbrace_{i \in \posetify{T_{n-1}/T_{n-1}}}
	\arrow[two heads, from=1-2, to=1-3]
	\arrow[two heads, from=1-1, to=1-2]
	\arrow[two heads, from=1-3, to=1-4]
	\arrow[two heads, from=1-4, to=1-5]
\end{tikzcd}
\end{equation*}
\end{cons}

To prove that $\RAdCon$ is a simplicial map, we need the following result:

\begin{lem}\label{lemma:dtop and R compatible}
For each admissible map $T' \rightarrowtail T$, we have that
$$ \lbrace \posetify{T}_i \rbrace_{i \in \posetify{T/T'}} = \lbrace \posetify{T'}_i \rbrace_{i \in \posetify{T'/T'}} . $$
\end{lem}

\begin{proof}
Since $T' \rightarrowtail T$ is admissible, we have that $\underline{T'/T'} = \underline{T/T'}$ and the right square 
\[\begin{tikzcd}
	{\posetify{T}_i} \\
	& {\posetify{T'}_i} & {\posetify{T'}} & {\posetify{T}} \\
	& 1 & {\posetify{T'/T'}} & {\posetify{T/T'}}
	\arrow[tail, from=2-2, to=2-3]
	\arrow[tail, from=2-3, to=2-4]
	\arrow[tail, from=3-3, to=3-4]
	\arrow[two heads, from=2-4, to=3-4]
	\arrow[two heads, from=2-3, to=3-3]
	\arrow[from=2-2, to=3-2]
	\arrow["{\ulcorner i\urcorner }"', from=3-2, to=3-3]
	\arrow[bend right = 20, from=1-1, to=3-2]
	\arrow[bend left = 20, from=1-1, to=2-4]
	\arrow["\lrcorner"{anchor=center, pos=0.125}, draw=none, from=2-2, to=3-3]
	\arrow[dotted, from=1-1, to=2-2]
\end{tikzcd}\]
is a pullback as a consequence of Lemma \ref{lem:admissible pullback preorders}. By the prism Lemma the outer rectangle is a pullback since the left square is a pullback by definition. So by the pullback property of $\posetify{T'}_i$, the dotted arrow exists and is an equality since the outer diagram is a pullback.
\end{proof}

To see that $\RAdCon$ is a simplicial map, we have to verify that for $0 \leq i < n$, the diagram 
\[\begin{tikzcd}
	{\mathbf{A}_n} & {\mathbf{A}_{n-1}} \\
	{\mathbf{K}_n} & {\mathbf{K}_{n-1}}
	\arrow["{\RAdCon_n}"', from=1-1, to=2-1]
	\arrow["{d_i}"', from=2-1, to=2-2]
	\arrow["{\RAdCon_{n-1}}", from=1-2, to=2-2]
	\arrow["{d_i}", from=1-1, to=1-2]
\end{tikzcd}\]
commutes, but this follows from Construction \ref{construction: AP to K}. For the case that $i = n$, we have to use Lemma \ref{lemma:dtop and R compatible}. Let us prove it for $n = 2$. For greater $n$ the proof is completely analogous, but the notation becomes much heavier. To prove that the diagram
\[\begin{tikzcd}
	{\mathbf{A}_2} & {\mathbf{A}_{1}} \\
	{\mathbf{K}_2} & {\mathbf{K}_{n-1}}
	\arrow["{\RAdCon_2}"', from=1-1, to=2-1]
	\arrow["{d_\top}"', from=2-1, to=2-2]
	\arrow["{\RAdCon_{1}}", from=1-2, to=2-2]
	\arrow["{d_\top}", from=1-1, to=1-2]
\end{tikzcd}\]
commutes. It is enough to verify that for each admissible map $T' \rightarrowtail T$, we have that 
$$ \lbrace \posetify{T}_i \rbrace_{i \in \posetify{T/T'}} = \lbrace \posetify{T'}_i \rbrace_{i \in \posetify{T'/T'}}  $$
since $d_\top (T' \rightarrowtail T) = T'$ and $d_\top (\posetify{T} \twoheadrightarrow \posetify{T/T'}) = \lbrace \posetify{T}_i \rbrace_{i \in \posetify{T/T'}}$. But this follows from Lemma \ref{lemma:dtop and R compatible}.
  
\setcounter{equation}{0}

\begin{propo}\label{proposition:AP to K is culf}
The map $\RAdCon \colon \mathbf{A} \to \mathbf{K}$ is culf.
\end{propo}

\begin{proof}
Since $\mathbf{A}$ and $\mathbf{K}$ are decomposition spaces, to prove that $\RAdCon$ is culf, it is enough to 
check that the diagram
\[\begin{tikzcd}
	{\mathbf{A}_2} & {\mathbf{A}_1} \\
	{\mathbf{K}_2} & {\mathbf{K}_1}
	\arrow["{\RAdCon_2}"', from=1-1, to=2-1]
	\arrow["{d_1}"', from=2-1, to=2-2]
	\arrow["{\RAdCon_1}", from=1-2, to=2-2]
	\arrow["{d_1}", from=1-1, to=1-2]
\end{tikzcd}\]
is a pullback  by Lemma \ref{lemma culfcondition for ds}. To prove that the square is a pullback, we will use Lemma \ref{lemma:pullbackfibres}. This means that for each object $T \in \mathbf{A}_1$, we have to show that the map $\RAdCon_2 \colon \Fib_T (d_1) \to \Fib_{\RAdCon_1T}(d_1)$ is an equivalence. We will divide the proof into three steps: first we will prove that is $\RAdCon_2 \colon \Fib_T (d_1) \to \Fib_{\RAdCon_1T}(d_1)$ is essentially surjective on objects. After we will show that $\RAdCon_2$ is full and finally that it is faithful.
\begin{itemize}
\item Let $f \colon \posetify{T} \twoheadrightarrow P$ be an object in $\Fib_{\RAdCon_1T}(d_1)$. Consider the following pullback diagram:
\[\begin{tikzcd}
	{\underline{P} \times_P T } & T \\
	{\underline{P} \times_P \posetify{T} } & \posetify{T} \\
	{\underline{P}} & P.
	\arrow["\posy", two heads, from=1-2, to=2-2]
	\arrow["f", two heads, from=2-2, to=3-2]
	\arrow[tail, from=3-1, to=3-2]
	\arrow["\alpha", from=1-1, to=1-2]
	\arrow[from=1-1, to=2-1]
	\arrow[from=2-1, to=3-1]
	\arrow[from=2-1, to=2-2]
	\arrow["\lrcorner"{anchor=center, pos=0.125}, draw=none, from=1-1, to=2-2]
	\arrow["\lrcorner"{anchor=center, pos=0.125}, draw=none, from=2-1, to=3-2]
\end{tikzcd}\]
Since $\underline{P} \to P$ is admissible, the map $\alpha$ is admissible by the stability property of admissible maps under pullback (Lemma \ref{lemma: admissble map stable pullback}). It is easy to check that $\alpha$ is an object in $\Fib_T(d_1)$. To verify that $\RAdCon_2(\alpha) \cong f$, consider the commutative diagram:
\[\begin{tikzcd}
	{\underline{P} \times_P T } & T \\
	{(\underline{P} \times_P T)/\underline{P} \times_P T } & {T/(\underline{P} \times_P T )} & \posetify{T} \\
	& {\underline{P}} & P.
	\arrow["\alpha", tail, from=1-1, to=1-2]
	\arrow[two heads, from=1-1, to=2-1]
	\arrow["f", two heads, from=2-3, to=3-3]
	\arrow["\posy", two heads, from=1-2, to=2-3]
	\arrow[two heads, from=1-2, to=2-2]
	\arrow[from=2-1, to=2-2]
	\arrow["\lrcorner"{anchor=center, pos=0.125, rotate=180}, draw=none, from=2-2, to=1-1]
	\arrow[from=2-1, to=3-2]
	\arrow[tail, from=3-2, to=3-3]
	\arrow["u"', dashed, from=2-2, to=3-3]
	\arrow["{(1)}"{description}, draw=none, from=1-1, to=2-2]
	\arrow["{(2)}"{description, pos=0.6}, draw=none, from=1-2, to=3-3]
\end{tikzcd}\]
The square $(1)$ is a pushout by definition of ${T/(\underline{P} \times_P T )}$. Since the outer diagram commutes, the pushout property of ${T/(\underline{P} \times_P T )}$ gives the dotted arrow $u$. Applying the posetification functor to the square $(2)$, we obtain the commutative diagram:
\[\begin{tikzcd}
	{\posetify{T}} & {\posetify{T}} \\
	{\posetify{T/(\underline{P} \times_P T)}} & P.
	\arrow["f", two heads, from=1-2, to=2-2]
	\arrow["\identity", from=1-1, to=1-2]
	\arrow["{\RAdCon_2(\alpha)}"', two heads, from=1-1, to=2-1]
	\arrow["{u'}"', from=2-1, to=2-2]
	\arrow["{(3)}"{description}, draw=none, from=1-1, to=2-2]
\end{tikzcd}\]
Furthermore, $u'$ is a bijection. Indeed, $u'$ is a monotone surjection since $\identity$ and ${\RAdCon_2(\alpha)}$ and $f$ are monotone surjections, and the diagram $(3)$ commutes. The injectivity condition of $u'$ follows from the fact that the diagrams $(1)$ and $(2)$ commute and $u'$ is obtained from the posetification of $u$. This implies that ${\RAdCon_2(\alpha)} \simeq f$, and therefore $\RAdCon_2$ is essentially surjective on objects.

\item Let's prove that $\RAdCon_2$ is full. Let $u$ be a morphism in $\Fib_{\RAdCon_1T}(d_1)$. The morphism $u$ can be illustrated by the following diagram
\[\begin{tikzcd}[column sep=scriptsize]
	& {\posetify{T}} \\
	P && {P'}.
	\arrow["f"', two heads, from=1-2, to=2-1]
	\arrow["{f'}",two heads, from=1-2, to=2-3]
	\arrow["u"', from=2-1, to=2-3]
\end{tikzcd}\]
Since $\RAdCon_2$ is essentially surjective on objects, for $f$ and $f'$ in $\Fib_{\RAdCon_1T}(d_1)$, we have objects $\alpha$ and $\alpha'$ in $\Fib_{T}(d_1)$ such that $\RAdCon_2(\alpha) = f$, and $\RAdCon_2(\alpha') = f'$, and the following diagram commutes 
\[\begin{tikzcd}[sep=small]
	W && T \\
	& {W'} && T \\
	{\posetify{W}} && {\posetify{T}} \\
	& {\posetify{W}'} && {\posetify{T}} \\
	{\underline{P}} && P \\
	& {\underline{P}'} && {P'}.
	\arrow[from=6-2, to=6-4]
	\arrow["{f'}", two heads, from=4-4, to=6-4]
	\arrow[two heads, from=2-4, to=4-4]
	\arrow["{\alpha'}"'{pos=0.2}, tail, from=2-2, to=2-4]
	\arrow[two heads, from=2-2, to=4-2]
	\arrow[from=4-2, to=6-2]
	\arrow[tail, from=4-2, to=4-4]
	\arrow["u"', from=5-1, to=6-2]
	\arrow[from=5-1, to=5-3]
	\arrow["u", from=5-3, to=6-4]
	\arrow["\identity", from=1-3, to=2-4]
	\arrow[two heads, from=1-3, to=3-3]
	\arrow["f"'{pos=0.8}, two heads, from=3-3, to=5-3]
	\arrow["\identity", from=3-3, to=4-4]
	\arrow["\alpha", tail, from=1-1, to=1-3]
	\arrow[from=3-1, to=5-1]
	\arrow[two heads, from=1-1, to=3-1]
	\arrow["{\overline{u}}"', dotted, from=1-1, to=2-2]
	\arrow[tail, from=3-1, to=3-3]
\end{tikzcd}\]
The pullback property of $W'$ gives the dotted arrow $\overline{u}$, and the top square commutes. This implies that $\overline{u} \colon \alpha \to \alpha'$ is a morphism in $\Fib_T(d_1)$ and $\RAdCon_2(\overline{u}) = u$. 

\item It only remains to prove that $\RAdCon_2$ is faithful. Let $v$ and $v'$ be morphisms in $\Fib_{T}(d_1)$ such that $\RAdCon_2(v) = \RAdCon_2(v')$. The objects $v$ and $v'$ can be illustrated by the following diagram
\begin{equation}\tag{3}\label{equation: T to P diagram 4}
\begin{tikzcd}
	W && {W'} \\
	& T.
	\arrow["\alpha"', tail, from=1-1, to=2-2]
	\arrow["{\alpha'}", tail, from=1-3, to=2-2]
	\arrow["v"', shift right=1, from=1-1, to=1-3]
	\arrow["{v'}", shift left=2, from=1-1, to=1-3]
\end{tikzcd}
\end{equation}
Applying $\RAdCon_2$ to (\ref{equation: T to P diagram 4}), we have the following commutative diagram 
\[\begin{tikzcd}[sep=small]
	W && T \\
	& {W'} && T \\
	{\posetify{W}} && {\posetify{T}} \\
	& {\posetify{W}'} && {\posetify{T}} \\
	{\underline{P}} && P \\
	& {\underline{P}'} && {P'}.
	\arrow[from=6-2, to=6-4]
	\arrow["{\RAdCon_2(\alpha')}", two heads, from=4-4, to=6-4]
	\arrow[two heads, from=2-4, to=4-4]
	\arrow["{\alpha'}"'{pos=0.2}, tail, from=2-2, to=2-4]
	\arrow[two heads, from=2-2, to=4-2]
	\arrow[from=4-2, to=6-2]
	\arrow[tail, from=4-2, to=4-4]
	\arrow["{\RAdCon_2(v)}"', from=5-1, to=6-2]
	\arrow[from=5-1, to=5-3]
	\arrow["{\RAdCon_2(v)}"', from=5-3, to=6-4]
	\arrow["\identity", from=1-3, to=2-4]
	\arrow[two heads, from=1-3, to=3-3]
	\arrow["{\RAdCon_2(\alpha)}"'{pos=0.8}, two heads, from=3-3, to=5-3]
	\arrow["\identity", from=3-3, to=4-4]
	\arrow["\alpha", tail, from=1-1, to=1-3]
	\arrow[from=3-1, to=5-1]
	\arrow[two heads, from=1-1, to=3-1]
	\arrow[tail, from=3-1, to=3-3]
	\arrow["{v'}"', shift right=1, from=1-1, to=2-2]
	\arrow["v", shift left=1, from=1-1, to=2-2]
\end{tikzcd}\]
By the pullback property of $W'$, there exists a unique map from $W$ to $W'$ such that the above diagram commutes. This forces $v = v'$. Hence, $\RAdCon_2$ is faithful.
\end{itemize}
\end{proof}

\begin{rema}
The culf map $\RAdCon \colon \mathbf{A} \to \mathbf{K}$ induces a homomorphism from the Fauvet--Foissy--Manchon bialgebra of admissible maps $\Delta_{\mathbf{A}}$ to the incidence bialgebra of contractions $\Delta_\mathbf{K}$. 
\end{rema}

\begin{blanko} 
{Admissible maps and the Waldhausen construction}
\label{subsec: double category AP}
\end{blanko}

Bergner, Osorno, Ozornova, Rovelli, and Scheimbauer \cite{BergnerSDoubleconstruction} showed an equivalence between decomposition objects and augmented stable double Segal objects, which is given by an $S_{\bullet}$-construction. In this section, we will construct a stable augmented double category using admissible maps and contractions between preorders.


A \emph{double category} is a category internal to categories. More informally,
it consists of the following data, subject to some axioms: a set of objects;
two different classes of morphisms (horizontal and vertical); and squares, which are connected by various source, target, identity, and composition maps.

\begin{exa}
Let $\mathbf{{AdCon}}$ denote the double category of finite preorders, admissible maps as horizontal morphisms, and contractions as vertical morphisms. The squares are pushout diagrams of admissible maps along contractions of preorders.
\end{exa}

An \emph{augmentation} of a double category $\mathbf{C}$ consists of a set of objects $C$ satisfying the condition that for every object $d$ of $\mathbf{C}$, there are a unique horizontal morphism $c \rightarrowtail d$ and a unique vertical morphism $d \twoheadrightarrow c'$ such that $c$ and $c'$ are in $C$ \cite{BergnerSDoubleconstruction}. We need some results to show an augmentation in $\mathbf{{AdCon}}$. 

\begin{lem}\label{lemma:chosen iniital admissibles gr. pre.}
The category of finite preorders and admissible maps has chosen initial objects given by the groupoid finite preorders.
\end{lem}

\begin{proof}
Given a finite preorder $T$, we have a canonical groupoid finite preorder $T^{\invp}$ and a canonical admissible map $T^{\invp} \to T$. 
\end{proof}

\begin{lem}\label{lemma:chosen terminal contractions gr. pre.}
The category of finite preorders and contractions of preorders has chosen terminal objects given by the groupoid finite preorders.
\end{lem}

\begin{proof}
Given a finite preorder $T$, we have a groupoid finite preorder $T/T$ and a canonical identity-on-objects monotone map $T \to T/T$, which is a contraction since we invert all the morphism in $T$. Indeed, the lift of covers is automatic. The connection of the fibres is given by the fact that $T/T$ is the preorder of connected components of $T$.
\end{proof}

\begin{propo}
The double category $\mathbf{{AdCon}}$ has an augmentation given by the set of groupoid finite preorders
\end{propo}

\begin{proof}
For each preorder $T$, we have a canonical admissible map $T^{\invp} \rightarrowtail T$ (\ref{lemma:chosen iniital admissibles gr. pre.}) and a canonical contraction $T \twoheadrightarrow T/T$ (\ref{lemma:chosen terminal contractions gr. pre.}). It is easy to see that $T^{\invp}$ and $T/T$ are groupoid preorders, and hence the set of groupoid finite preorders is an augmentation of $\mathbf{AdCon}$.
\end{proof}

A double category is \emph{stable} if every square is uniquely determined by its span of source morphisms and, independently, by its cospan of target morphisms \cite{BergnerSDoubleconstruction}.

\begin{propo}
The double category $\mathbf{{AdCon}}$ is stable.
\end{propo}

\begin{proof}
By definition of $\mathbf{{AdCon}}$ the squares are pushout diagrams as follows:
\[\begin{tikzcd}
	{\cdot} & {\cdot} \\
	{\cdot} & {\cdot}
	\arrow[tail, from=1-1, to=1-2]
	\arrow[tail, from=2-1, to=2-2]
	\arrow[two heads, from=1-1, to=2-1]
	\arrow[two heads, from=1-2, to=2-2]
	\arrow["\lrcorner"{anchor=center, pos=0.125, rotate=180}, draw=none, from=2-2, to=1-1]
\end{tikzcd}\]
where the horizontal morphisms are admissible maps and the vertical are contractions of preorders. Since the top horizontal arrow is admissible, the square is also a pullback as a consequence of Lemma \ref{lem:admissible pullback preorders}. This means that all the squares in $\mathbf{{AdCon}}$ are bipullbacks, and therefore $\mathbf{{AdCon}}$ is stable.
\end{proof}

In an augmented stable double category, there is a bijection between the set of horizontal morphisms and the set of vertical arrows \cite{BergnerSDoubleconstruction}. This implies that in $\mathbf{{AdCon}}$, we have a bijection between admissible maps and contractions of preorders. This bijection was also proved in Lemma \ref{lemma:bijection adm and contr pre}.

Bergner, Osorno, Ozornova, Rovelli, and Scheimbauer \cite{BergnerSDoubleconstruction} expressed the Waldhausen $S_{\bullet}$-construction in terms of a  
functor $S_{\bullet}$ from the category of augmented double categories to the category of simplicial groupoids. If we only consider stable augmented double categories, they showed that the functor $S_{\bullet}$ sends stable augmented double categories to decomposition spaces. 

\begin{teo}\cite[Theorem 4.8]{BergnerSDoubleconstruction}\label{teorema:Wal construction}
The Waldhausen $S_{\bullet}$-construction restricts to the functor
$$S_{\bullet} \colon \operatorname{DCat}_{\operatorname{aug}}^{\operatorname{st}} \to \Dcmp,$$
where $\operatorname{DCat}_{\operatorname{aug}}^{\operatorname{st}}$ denotes the category of augmented stable double categories. 
\end{teo}

The construction of $S_{\bullet}$ and the proof of Theorem \ref{teorema:Wal construction} require some preliminaries that are beyond the scope of this paper but which can be found in detail in \cite{BergnerSDoubleconstruction}. So we will choose to give a detailed description of the example we are interested in, the decomposition space $S_{\bullet}(\mathbf{{AdCon}})$. The objects of the groupoid $(S_{\bullet}\mathbf{{AdCon}})_0$ are finite preorders and the morphisms monotone bijections. The objects of the groupoid $(S_{\bullet}\mathbf{{AdCon}})_n$ are diagrams of the form:

\begin{equation*}\tag{1}\label{diagram: Sn AP}
\begin{tikzcd}[sep=scriptsize]
	{V_{00}} & {T_{01}} & {T_{02}} & \cdots & {T_{0{(n-1})}} & {T_{0n}} \\
	& {V_{11}} & {T_{12}} & \cdots & {T_{1{(n-1})}} & {T_{1n}} \\
	&& {V_{22}} & \cdots & \vdots & \vdots \\
	&&&& {V_{(n-1)(n-1)}} & {T_{(n-1)n}} \\
	&&&&& {V_{nn}}
	\arrow[tail, from=1-1, to=1-2]
	\arrow[tail, from=1-2, to=1-3]
	\arrow[tail, from=1-3, to=1-4]
	\arrow[two heads, from=1-2, to=2-2]
	\arrow[two heads, from=1-3, to=2-3]
	\arrow[tail, from=2-2, to=2-3]
	\arrow[tail, from=2-3, to=2-4]
	\arrow[from=2-3, to=3-3]
	\arrow[tail, from=3-3, to=3-4]
	\arrow[two heads, from=1-4, to=2-4]
	\arrow[two heads, from=2-4, to=3-4]
	\arrow[tail, from=1-4, to=1-5]
	\arrow[tail, from=1-5, to=1-6]
	\arrow[tail, from=2-4, to=2-5]
	\arrow[two heads, from=1-5, to=2-5]
	\arrow[two heads, from=1-6, to=2-6]
	\arrow[two heads, from=2-6, to=3-6]
	\arrow[two heads, from=2-5, to=3-5]
	\arrow[tail, from=3-4, to=3-5]
	\arrow[tail, from=2-5, to=2-6]
	\arrow[from=3-5, to=3-6]
	\arrow[two heads, from=3-5, to=4-5]
	\arrow[two heads, from=3-6, to=4-6]
	\arrow[tail, from=4-5, to=4-6]
	\arrow["\square"{description}, draw=none, from=1-2, to=2-3]
	\arrow["\square"{description}, draw=none, from=1-3, to=2-4]
	\arrow["\square"{description}, draw=none, from=2-3, to=3-4]
	\arrow["\square"{description}, draw=none, from=1-4, to=2-5]
	\arrow["\square"{description}, draw=none, from=2-4, to=3-5]
	\arrow["\square"{description}, draw=none, from=1-5, to=2-6]
	\arrow[draw=none, from=2-5, to=3-6]
	\arrow["\square"{description}, draw=none, from=3-5, to=4-6]
	\arrow[two heads, from=4-6, to=5-6]
	\arrow["\square"{description}, draw=none, from=2-5, to=3-6]
\end{tikzcd}
\end{equation*}
where each $V_{ii}$ is a groupoid finite preorder and the squares are bipullbacks. The morphisms in $(S_{\bullet}\mathbf{{AdCon}})_n$ are families of monotone bijections between the finite preorders of the diagrams. 
The face maps $d_i \colon (S_{\bullet}\mathbf{{AdCon}})_n \to (S_{\bullet}\mathbf{{AdCon}})_{n-1}$ acts by `erasing' all the finite preorders that containing an index $i$.  The degeneracy maps $s_i \colon (S_{\bullet}\mathbf{{AdCon}})_{n} \to (S_{\bullet}\mathbf{{AdCon}})_{n+1}$ `repeat' the $i$th row and the $i$th column. 

We have a canonical map from $\frow \colon S_\bullet(\mathbf{{AdCon}}) \to \mathbf{A}$, that forgets the extra information and only conserves the first row of the diagram (\ref{diagram: Sn AP}) deleting the groupoid finite preorder $V_{00}$ in the chain. In other words, its sends the diagram (\ref{diagram: Sn AP}) to the chain 
\[\begin{tikzcd}
	{T_{01}} & {T_{02}} & \cdots & {T_{0(n-1)}} & {T_{0n}}.
	\arrow[tail, from=1-1, to=1-2]
	\arrow[tail, from=1-2, to=1-3]
	\arrow[tail, from=1-3, to=1-4]
	\arrow[tail, from=1-4, to=1-5]
\end{tikzcd}\]  

\begin{propo}\label{propo:stretficationA}
The map $\frow \colon S_\bullet(\mathbf{{AdCon}}) \to \mathbf{A}$ is an equivalence of simplicial spaces.
\end{propo}

\begin{proof}
The map $\frow$ is essentially surjective on objects. Let us prove it for an $3$-simplex, for a $n$-simplex the proof is completely analogous. 
For a $3$-simplex 
\[\begin{tikzcd}
	{T_0} & {T_1} & {T_2}
	\arrow[tail, from=1-1, to=1-2]
	\arrow[tail, from=1-2, to=1-3]
\end{tikzcd}\] 
in $\mathbf{A}$, we construct the following diagram
\[\begin{tikzcd}
	{T_{0}} & {T_1} & {T_2} \\
	{T_{0}/T_{0}} & {T_1/T_0} & {T_2/T_0} \\
	& {T_1/T_1} & {T_2/T_1}
	\arrow[tail, from=1-1, to=1-2]
	\arrow[two heads, from=1-1, to=2-1]
	\arrow[from=2-1, to=2-2]
	\arrow[two heads, from=1-2, to=2-2]
	\arrow["\lrcorner"{anchor=center, pos=0.125, rotate=180}, draw=none, from=2-2, to=1-1]
	\arrow[tail, from=1-2, to=1-3]
	\arrow[tail, from=2-2, to=2-3]
	\arrow[two heads, from=1-3, to=2-3]
	\arrow[two heads, from=2-2, to=3-2]
	\arrow[tail, from=3-2, to=3-3]
	\arrow[two heads, from=2-3, to=3-3]
	\arrow["\lrcorner"{anchor=center, pos=0.125, rotate=180}, draw=none, from=2-3, to=1-2]
	\arrow["\lrcorner"{anchor=center, pos=0.125, rotate=180}, draw=none, from=3-3, to=2-2]
\end{tikzcd}\]
where each square is a pushout, and since the top arrows are admissible it follows that the squares are bipullbacks by Lemma \ref{lem:admissible pullback preorders}. Furthermore, adding to the top row the canonical admissible map $T_0^{\invp} \rightarrowtail T_0$ and the end of the last column the contraction $T_2/T_1 \twoheadrightarrow (T_2/T_1)/(T_2/T_1)$, we obtain a $3$-simplex in $S_\bullet(\mathbf{{AdCon}})$. The map $\frow$ is full and faithful since any morphism between $n$-simpleces in $S_\bullet(\mathbf{{AdCon}})$ is unique determined by families of monotone bijections between the preorders of the top row of the $n$-simpleces in $S_\bullet(\mathbf{{AdCon}})$. Since $\frow$ is full and faithful, and essentially surjective on objects, then it is an equivalence of decomposition spaces.
\end{proof}

Note that $S_\bullet(\mathbf{{AdCon}})$ is a strict simplicial space while $\mathbf{A}$ is a pseudo simplicial space. So the equivalence $\frow \colon S_\bullet(\mathbf{{AdCon}}) \to \mathbf{A}$ allows us to see $S_\bullet(\mathbf{{AdCon}})$ as the strictification of $\mathbf{A}$ in the sense of Gambino \cite[\S 6.4]{gambino_2008}.


\section{Connected directed hereditary species as decomposition spaces}\label{subsec: DCH species as decomposition}

In this section, we will show how to obtain a decomposition space from a connected directed hereditary species.

Let $H \colon\mathbb{K}_p \rightarrow \Grpd$ be a directed hereditary species. The left fibration $\int H \to \mathbb{K}_p$ denotes the Grothendieck construction of $H$. The objects of $\int H$ are pairs $(P,x)$ where $P$ is a finite poset and $x \in H[P]$. The morphisms are pairs $(p, \alpha) \colon (P, x) \to (Q, y)$ where $p \colon P \to Q$ is a partially defined contraction and $\alpha \colon H(p)(x) \to y$ is a morphism in $H[Q]$.  We are interested in the category, with chosen local terminal objects, of connected finite non-empty posets and contractions $\mathbb{K}$, not all partially defined contractions. For that, we consider $\mathbb{H}$ as the pullback
\[\begin{tikzcd}
	{\mathbb{H}} \drpullback & {\int H} \\
	{\mathbb{K}} & {\mathbb{K}_p}.
	\arrow[from=1-2, to=2-2]
	\arrow[hook, from=2-1, to=2-2]
	\arrow[from=1-1, to=2-1]
	\arrow[from=1-1, to=1-2]
\end{tikzcd}\]
Furthermore, we define $\mathbf{H} := \mathsf{S} \fatnerve^{\lt} \mathbb{H}$. Thus, an object in $\mathbf{H}_n$ is a family of $(n-1)$-chains of contractions

\begin{equation*}
\begin{tikzcd}[column sep=35pt]
P_0 \arrow[r, "f_0", two heads] & P_1 \arrow[r, two heads] &   \cdots\arrow[r, two heads]&P_{n-2} \arrow[r, "f_{n-2}",two heads]&P_{n-1}
\end{tikzcd}
\end{equation*}
with a $H$-structure on $P_0$ and $\mathbf{H}_0$ is the groupoid of families of $H$-structure over the poset with one element. We define the bottom face map as follows:  we remove the first element of the chain, $P_0$, and take the $H$-structure $H(f_0)$ on $P_1$. To define the top face map, we use (contravariant) functoriality on convex map: for each element $a \in P_{n-1}$, we can form the fibres over $a$. We end up with a family
\begin{center}
\begin{tikzcd}
\lbrace (P_0)_{a} \arrow[r, "(f_0)_a", two heads] & (P_1)_a \arrow[r, "(f_1)_a", two heads] & (P_2)_{a} \arrow[r, two heads] & \cdots \arrow[r, two heads] & (P_{n-3})_a \arrow[r, "(f_{n-3})_a", two heads] & (P_{n-2})_a \rbrace_{a \in P_{n-1}},
\end{tikzcd}
\end{center} 
and we restrict the $H$-structure on $P_0$ to the corresponding $H$-structure on the fibre $(P_0)_a$, in the same way as for $\mathbf{K}$. 

\begin{propo}\label{propo: H is a MDS}
The groupoids $\mathbf{H}_n$ form a simplicial groupoid $\mathbf{H}$.
\end{propo}

\begin{proof}
It is straightforward to see that all the simplicial identities involving inner face maps and degeneracy maps are satisfied since $\mathbf{H} = \mathsf{S} \fatnerve^{\lt} \mathbb{H}$. Let us see the identity $d_{\top}d_{\bot}\simeq d_{\bot}d_{\top}$: consider an element 
\begin{equation*}
\begin{tikzcd}[column sep=35pt]
P_0 \arrow[r, "f_0", two heads] & P_1 \arrow[r, two heads] &   \cdots\arrow[r, two heads]&P_{n-2} \arrow[r, "f_{n-2}",two heads]&P_{n-1}
\end{tikzcd}
\end{equation*}
of $\mathbf{H}_n$. Clearly the identity is satisfied at the level of posets, since it is satisfied in $\mathbf{K}$. Hence we only need to see that the final $H$-structures coincide, but this is just a consequence of the functoriality of $H$ applied to the square
 \begin{equation*}
\begin{tikzcd}[column sep=30pt, row sep=30]
P_0 \arrow[d,two heads] &(P_0)_p \arrow[d,two heads]\arrow[l,tail]\dlpullback\\
P_1&(P_1)_p\arrow[l,tail].
\end{tikzcd}
\end{equation*}
for every $p\in P_{n-1}$. Similarly, the identities $d_{\top}d_{\top}\simeq d_{\top}d_{\top-1}$ and $d_{\bot}d_{\bot}\simeq d_{\bot}d_{\bot+1}$ come from functoriality of $H$ in composition of convex maps and composition of contractions, respectively.
\end{proof}

Since the map $\mathbb{H} \to \mathbb{K}$ is a left fibration and $\mathsf{S}$ preserves pullbacks, we have a canonical culf map $\mathbf{H} \rightarrow \mathbf{K}$.

\begin{propo}\label{propo: H is a ds connected case}
The simplicial groupoid $\mathbf{H}$ is a monoidal decomposition space.
\end{propo}
\begin{proof}
Since there is a culf map $\mathbf{H} \rightarrow \mathbf{K}$ and $\mathbf{K}$ is a  decomposition space, so is $\mathbf{H}$ by Lemma \ref{proposition: ds+culf=ds}. The monoidal structure is again given by the disjoint union.
\end{proof}

\begin{propo} \label{propo: H is locally etc}
The decomposition space $\mathbf{H}$ is complete, locally finite, locally discrete and of locally finite length.
\end{propo}

\begin{proof}
Note that $\mathbf{H}_1$ is the groupoid of finite families of non-empty $H$-structures. The automorphisms of an object in $\mathbf{H}_1$ are given by permutations of the family and by the automorphisms of the underlying finite posets. Therefore each object can only have finitely many automorphisms. The rest follows from Propositions \ref{proposition: ds+locallydiscrete = locally discrete}, \ref{proposition:Kiscomplete}, and \ref{proposition:Kislocally}.
\end{proof}

Since $\mathbf{H}$ is locally finite by Lemma \ref{propo: H is locally etc}, the homotopy sum resulting from $\mathbf{H}$ is just an ordinary sum, as in \ref{subsec:coalgebraofcdhs}. Therefore, we have the following result:

\begin{propo} \label{propo: H hom cardinality}
The homotopy cardinality of the incidence bialgebra of $\mathbf{H}$ is isomorphic to the incidence bialgebra of $H$.
\end{propo}

\begin{blanko}
{Running example, part II: Calaque--Ebrahimi-Fard--Manchon comodule bialgebra of rooted trees} \label{subsec: runningexamplepartII}
\end{blanko}

In \ref{subsec: runningexamplepartI}, we constructed the connected directed hereditary species $H_{\CEM}$ of trees and contractions. Let $\mathbf{H}_{\CEM}$ denote the decomposition space 
associated with ${H}_{\CEM}$. Thus, $(\mathbf{H}_{\CEM})_n$ is the groupoid whose objects are families of $(n-1)$-chains of contractions between trees and whose morphisms are monotone bijections. The face and degeneracy maps of $\mathbf{H}_{\CEM}$ are defined as in $\mathbf{K}$. The incidence coalgebra of $\mathbf{H}_{\CEM}$ is described in Remark \ref{remark:incidencecoalgebraCEM}.

\begin{exa}[\textbf{Faà di Bruno comodule bialgebra of linear trees, part II}]\label{exa:linear trees 2}
Let $\mathbf{H}_{\FB}$ denote the decomposition space induced by the directed hereditary species of linear trees $H_{\FB}$ (\ref{exa:linear trees 1}). The description of $\mathbf{H}_{\FB}$ is similar to the description of $\mathbf{H}_{\CEM}$ but we consider linear trees instead of arbitrary trees. 
\end{exa}

\begin{blanko}
{Comodule structure}\label{subsec:comodule}
\end{blanko}

The theory of comodules in the context of decomposition spaces (2-Segal spaces) has been developed by Walde \cite{walde2017hall}, and independently by Young \cite{Young_2018}, both in the context of Hall algebras. Carlier \cite{carlier2019hereditary} gave a conceptual way to reformulate their definitions using linear functors. Given a map between two simplicial groupoids $F \colon C \rightarrow X$, the span
\begin{equation*}
\begin{tikzcd}
C_0 & C_1 \arrow[l, "d_\top"'] \arrow[r, "{(F_1,d_\perp)}"] & X_1 \times C_0
\end{tikzcd}
\end{equation*}
defines a linear functor 
$$\gamma \colon \Grpd_{/C_0} \rightarrow \Grpd_{/X_1} \otimes \Grpd_{/C_0}.$$

\begin{propo}\cite[Proposition 2.1.1]{carlier2019hereditary}\label{propo:leftcomodulegeneral}
Let $F \colon C \rightarrow X$ be a map between simplicial groupoids. Suppose moreover that $C$ is Segal, $X$ is a decomposition space, and the map $F \colon C \rightarrow X$ is culf, then the span 
$$
\begin{tikzcd}
C_0 & C_1 \arrow[l, "d_\top"'] \arrow[r, "{(F_1,d_\perp)}"] & X_1 \times C_0
\end{tikzcd}
$$
induces on the slice category $\Grpd_{/C_0}$ the structure of a left $\Grpd_{/X_1}$-comodule.
\end{propo}

Given a connected directed hereditary species $H$, the upper decalage of $\mathbf{H}$ induces a left comodule over $\Grpd_{/\mathbf{H}_1}$. 

\begin{lem} \label{lemma:M0isleftcomoduleoverH1}
The slice category $\Grpd_{/\mathbf{H}_1}$ is a left comodule over $\Grpd_{/\mathbf{H}_1}$.
\end{lem}

\begin{proof} 
Note that $\Dec_{\top} \mathbf{H}$ is a Segal space since $\mathbf{H}$ is a decomposition space. Furthermore, the decalage map $d_\top \colon \Dec_{\top} \mathbf{H} \rightarrow \mathbf{H}$ is culf. Since $\Dec_{\top} \mathbf{H}$ is Segal and $d_{\top}$ is culf, it follows that
\[\begin{tikzcd}
	{\mathbf{H}_1} & {\mathbf{H}_2} & {\mathbf{H}_1 \times \mathbf{H}_1}
	\arrow["{d_1}"', from=1-2, to=1-1]
	\arrow["{(d_\top,d_0)}", from=1-2, to=1-3]
\end{tikzcd}\]
induces a left comodule structure of $\Grpd_{/\mathbf{H}_1}$ over $\Grpd_{/\mathbf{H}_1}$ by Lemma \ref{propo:leftcomodulegeneral}.
\end{proof}

\section{The incidence comodule bialgebra of a connected directed hereditary species}\label{section:comudule bialgebra}

Let $\mathbb{I}$ be the category of finite sets and injections. A \emph{restriction species}~\cite{schmitt_1993} is a functor $R \colon \mathbb{I}^{\op} \rightarrow \Set$.
Every hereditary species is, in particular, a restriction species by precomposition with the inclusion $\mathbb{I}^{\op} \rightarrow \mathbb{S}_p$. Therefore we have two bialgebra structures associated to a hereditary species $H$: the incidence bialgebra $B$ of $H$ and the incidence bialgebra $A$ of the restriction species $R$ associated to $H$. Carlier \cite{carlier2019hereditary} showed that $A$ is a left comodule bialgebra over $B$. The main result of this section is to apply the Carlier ideas to the directed case.  

\begin{blanko}
{Directed restrictions species}\label{subsec:direcrestspe}
\end{blanko}

Let $\mathbb{C}$ denote the category of connected finite posets and convex maps. A directed restriction species in the sense of Gálvez--Kock--Tonks \cite{GKT:restr} is a functor $R \colon \mathbb{C}^{\op} \rightarrow \Grpd.$
As usual, the idea is that the value of a poset $S$ is the groupoid of all possible $R$-structures that have $S$ as the underlying poset.

The notion of directed restriction species needs to be modified to fit our theory of directed hereditary species since we only work with finite connected posets. Let $\mathbb{C}^{\circ}$ denote the category of connected finite posets and convex maps and $\mathbb{C}' := \mathsf{S}\mathbb{C}^{\circ}$. A \emph{directed restriction species} is a functor
$$R \colon (\mathbb{C}')^{\op} \to \Grpd.$$
Every hereditary species $H$ is, in particular, a directed restriction species. Indeed, let $R^{\circ}$ denote the precomposition of $H$ with the inclusion $(\mathbb{C}^{\circ})^{\op} \to \mathbb{K}_{p}$ which is identity on objects and sends a convex map $\imath \colon P' \rightarrow P$ to the span $\begin{tikzcd}[column sep=small]
	P & {P'} & {P'}.
	\arrow["\imath"', from=1-2, to=1-1]
	\arrow["\identity", from=1-2, to=1-3]
\end{tikzcd}$
The directed restriction species $R$ is the monoidal extension of $R^{\circ}$. In other words, $R$ is the unique functor that makes the following diagram commutes:
\[\begin{tikzcd}
	{(\mathbb{C}^{\circ})^{\op}} & {\mathbb{K}_p} \\
	{(\mathbb{C}')^{\op}}.
	\arrow["{R^{\circ}}", from=1-1, to=1-2]
	\arrow[from=1-1, to=2-1]
	\arrow["R"', dotted, from=2-1, to=1-2]
\end{tikzcd}\]

Every directed restriction species $R$ induces a decomposition space $\mathbf{R}$ where an $n$-simplex is a family of $n$-layered posets with an $R$-structure on the underlying posets. In other words, the objects of $\mathbf{R}_2$ are families of maps of posets $P \rightarrow 2$ with a $R$-structure on $P$, and $\mathbf{R}_n$ is the groupoid of families of $R$-structures $P \to n$. The construction of $\mathbf{R}$ follows from the theory developed by Gálvez, Kock, and Tonks in \cite[\S 7]{GKT:restr} considering only finite connected posets.

The comultiplication $\Delta_R: \Grpd_{/\mathbf{R}_1} \rightarrow \Grpd_{/\mathbf{R}_1} \otimes \Grpd_{/\mathbf{R}_1}$ is given by the span
$$
\begin{tikzcd}
\mathbf{R}_1 & \mathbf{R}_2 \arrow[l, "d_1"'] \arrow[r, "{(d_2,d_0)}"] & \mathbf{R}_1 \times \mathbf{R}_1
\end{tikzcd}
$$
where $d_1$ joins the two layeres of the $2$-simplex and $d_\top$ and $d_\perp$ return the first and second layers respectively.

\begin{exa}\label{exa:restriction BCK}
The Butcher--Connes--Kreimer Hopf algebra~\cite{Connes_1998, Dur} comes from the directed restriction species of trees $R_{\BCK}$:  a forest has an underlying poset, whose convex subposets inherit a tree structure (see Lemma \ref{lemma:slicesofforestsareforests}). The comultiplication $\Delta_{R_{\BCK}}$ is defined by summing over certain admissible cuts $c$:
$$\Delta_{R_{\BCK}}(T) = \sum_{c \in \mbox{\scriptsize{admi.cuts}}(T)} P_c \otimes R_c.$$
Note that $R_{\BCK}$ is the ordinary directed restriction species associated with the connected directed hereditary species $H_{\CEM}$.
\end{exa}

\begin{rema}
Recall that for a connected directed hereditary species $H$, we have that $\mathbf{R}_1 = \mathbf{H}_1$, so that by Lemma~\ref{lemma:M0isleftcomoduleoverH1}, the slice category $\Grpd_{/ \mathbf{R}_1}$ is a left $\Grpd_{/ \mathbf{H}_1}$-comodule. The left $\Grpd_{/ \mathbf{H}_1}$-coaction is the linear functor $\gamma_{\mathbf{R}_1} \colon \Grpd_{/ \mathbf{R}_1} \to \Grpd_{/ \mathbf{H}_1} \otimes \Grpd_{/ \mathbf{R}_1}$ given by the span 
\[\begin{tikzcd}
	{\mathbf{R}_1} & {\mathbf{H}_2} & {\mathbf{H}_1 \times \mathbf{R}_1}.
	\arrow["{d_1}"', from=1-2, to=1-1]
	\arrow["{(d_\top, d_0)}", from=1-2, to=1-3]
\end{tikzcd}\]
\end{rema}

The decomposition space $\mathbf{R}$ has a monoidal structure given by disjoint union.
Recall $\mathbf{R}_n$ is the groupoid of families of finite connected posets with $n-1$ compatible cuts. The disjoint union of two such structures is given by taking the disjoint union of the underlying posets, with the cuts concatenated. This defines a simplicial map $+_{\mathbf{R}} \colon \mathbf{R} \times \mathbf{R} \to \mathbf{R}$. So $\mathbf{R}$ is a monoidal decomposition space if the map $+_{\mathbf{R}}$ is culf \cite[\S 9]{GTK1}.

\begin{propo}\label{propo:R is monoidal}
The map $+_{\mathbf{R}} \colon \mathbf{R} \times \mathbf{R} \to \mathbf{R}$ is culf.
\end{propo}

\begin{proof}
The proof is analogous to that Proposition \ref{propo:k is monoidal}, but using cuts instead of contractions.
\end{proof}

Since $\mathbf{R}$ is a monoidal decomposition space, it follows that the resulting incidence coalgebra is also a bialgebra \cite[\S 9]{GTK1}.

\begin{blanko}
{Comodule bialgebra}\label{subsec:comodulebialgebra}
\end{blanko}

For a background on comodule bialgebras, see for example \cite[\S 3.2]{abe2004hopf}, \cite[\S 5]{carlier2019hereditary}, and \cite{Manchon:Abelsymposium}. We follow the terminology of Carlier \cite[\S 5]{carlier2019hereditary}. Let $B$ be a bialgebra. We can associate $B$ with a canonical braided monoidal category of left $B$-comodules. The categorical structure comes from the coalgebra structure of $B$, and the (braided) monoidal structure arises from the algebra structure of $B$. A \emph{comodule bialgebra} over
$B$ is a bialgebra object in the (braided) monoidal category of left $B$-comodules, where a \emph{bialgebra object} in the braided monoidal category of left $B$-comodules is a $B$-comodule $M$ together with structure maps
\begin{align*}
 \Delta_M \colon M \rightarrow M \otimes M    & & \epsilon_M \colon M \rightarrow \mathbb{Q}\\
 \mu_M \colon M \otimes M \rightarrow M  & &  \eta_M \colon \mathbb{Q} \rightarrow M 
\end{align*}
which are all required to be $B$-comodules maps and to satisfy the bialgebra axioms. We will deal, in particular, with the requirement that $\Delta_M$ and $\epsilon_M$ are compatible with the coaction $\gamma \colon M \to B \otimes M$. In Proposition \ref{proposition:Aleftcomodule}, the two other axioms will be automatically satisfied, because as comodule, $M$ coincides with $B$ itself and the algebra structure of $M$ coincides with that of $B$. 

Recall that every connected directed hereditary species is also a directed restriction species. Let $H$ be a connected directed hereditary species. We denote by $R$ the induced directed restriction species and, as usual, $\mathbf{H}$ and $\mathbf{R}$ correspond to decomposition spaces. Also, the incidence coalgebra of $\mathbf{R}$ is denoted $A$, and the incidence bialgebra of $\mathbf{H}$ is denoted $B$.

\begin{lem}\label{lem:comul A is Bcomod map}
The comultiplication structure of $A$ is a $B$-comodule map.
\end{lem}

\begin{proof}
We need to show that the following diagram commutes, and the squares $(1)$ and $(2)$ are pullbacks:
\begin{center}
\begin{tikzcd}
\mathbf{R}_1 &  & \mathbf{R}_2 \arrow[white]{rrd}[black, description]{(1)} \arrow[ll, "d_1"'] \arrow[rr, "{(d_2,d_0)}"]                                                                                                      &  & \mathbf{R}_1 \times \mathbf{R}_1 \\
\mathbf{H}_2 \arrow[u, "d_1"] \arrow[dd, "{(d_\top, d_0)}"']  \arrow[white]{rrdd}[black, description]{(2)} &  & Z \arrow[ll, "\overline{d}_2", dashed] \arrow[rr, "\overline{d}_3", dashed] \arrow[u, "\overline{d}_1", dashed] \arrow[dd, "{(g, \overline{d_0})}"', dashed] &  & \mathbf{H}_2 \times \mathbf{H}_2 \arrow[u, "d_1 \otimes d_1"'] \arrow[d, "{(d_\top,d_0) \otimes (d_\top,d_0)}"] \\
&  &                                                                                                                           &  & \mathbf{H}_1 \times \mathbf{R}_1 \times \mathbf{H}_1 \times \mathbf{R}_1 \arrow[d, "\omega"]                    \\
\mathbf{H}_1 \times \mathbf{R}_1                        &  & \mathbf{H}_1 \times \mathbf{R}_2  \arrow[ll, "\identity \otimes d_1"] \arrow[rr, "{\identity \otimes (d_2, d_0)}"']                                             &  & \mathbf{H}_1 \times \mathbf{R}_1 \times \mathbf{R}_1.                                                 
\end{tikzcd}
\end{center}
Here the map $\omega$ is given by first swapping the two middle tensor factors and then using the multiplication of $\mathbf{R}$. The objects of the groupoid $Z$ are families of pairs of $2$-chains of maps $P \twoheadrightarrow Q \rightarrow 2$, such that the first one is a contraction and the other is a cut, and with a $H$-structure on $P$. The map $\overline{d}_0$ forgets the contraction, the map $\overline{d}_1$ composes the contraction and the cut, the map $\overline{d}_2$ forgets the cut, the map $\overline{d}_3$ gives the pair of contractions $(P_1 \twoheadrightarrow Q_1, P_2 \twoheadrightarrow Q_2)$, and the map $g$ sends the $2$-chain to the family $\lbrace P_q \rbrace_{q \in Q}$ of the fibres of the contraction $P \twoheadrightarrow Q$.

By Lemma \ref{lemma:pullbackfibres}, the square $(1)$ is a pullback if for any pair of contractions $(f_1 \colon P_1 \twoheadrightarrow Q_1, f_2 \colon P_2 \twoheadrightarrow Q_2) \in \mathbf{H}_2 \times \mathbf{H}_2$ and cut $P \to 2 \in \mathbf{R}_2$ such that 
$$(d_2, d_0)(P \to 2) = (P_1, P_2) = (d_1 \otimes d_1)(P_1 \twoheadrightarrow Q_1, P_2 \twoheadrightarrow Q_2),$$
there exists a finite connected poset $Q$, a contraction $f \colon P \twoheadrightarrow Q$, and a cut $Q \to 2$ such that the diagram 
\[\begin{tikzcd}
	P & Q & 2 \\
	{P_i} & {Q_i}
	\arrow["f", two heads, from=1-1, to=1-2]
	\arrow[from=1-2, to=1-3]
	\arrow[tail, from=2-1, to=1-1]
	\arrow[tail, from=2-2, to=1-2]
	\arrow["{f_i}"', two heads, from=2-1, to=2-2]
\end{tikzcd}\]
commutes for $i = 1,2$. But this is easy to show if we put $Q := \sum_{i \in 2} Q_i$, and $f := \sum_{i \in 2} f_i$, and consider the partial order $<_Q$ on $Q$ given by taking transitive closure in the following relation: for $q, q' \in Q$, we declare that $q <_Q q'$ if $q, q' \in Q_i$ and  $q <_{Q_{i}} q'$ or there exists $p <_P p'$ in $P$ such that $f(p) = q$ and $f(p') = q'$.


We will prove that $(2)$ is a pullback. By the prism Lemma, $(2)$ is a pullback if the outer diagram 
\[\begin{tikzcd}[sep=2.25em]
	Z & {\mathbf{H}_1 \times \mathbf{R}_2 } & {\mathbf{R}_2} \\
	{\mathbf{H}_2} & {\mathbf{H}_1 \times \mathbf{R}_1 } & {\mathbf{R}_1}
	\arrow["{(d_\top, d_0)}", from=2-1, to=2-2]
	\arrow[from=2-2, to=2-3]
	\arrow["{d_0}"', bend right =20, from=2-1, to=2-3]
	\arrow["{(g, d_0)}"', from=1-1, to=1-2]
	\arrow[from=1-2, to=1-3]
	\arrow["{d_1}", from=1-3, to=2-3]
	\arrow["{\identity \otimes d_1}"{description}, from=1-2, to=2-2]
	\arrow["{\overline{d}_2}"', from=1-1, to=2-1]
	\arrow["{\overline{d}_0}", bend left =20, from=1-1, to=1-3]
	\arrow["{(3)}"{description}, draw=none, from=1-2, to=2-3]
\end{tikzcd}\]
and the square $(3)$ are pullbacks. The square $(3)$ is obtained after projecting away $\mathbf{H}_1$ so it is straightforward to see that it is a pullback. By Lemma \ref{lemma:pullbackfibres}, the outer diagram is a pullback since for any contraction $P \twoheadrightarrow Q \in \mathbf{H}_2$ and cut $Q \to 2 \in \mathbf{R}_2$, we can form the $2$-chain $P \twoheadrightarrow Q \to 2$, which is an object in $Z$, such that it makes the outer diagram commutes.
\end{proof}

\begin{lem}\label{lem: counit A is Bcom map}
The counit structure of $A$ is a $B$-comodule map.
\end{lem}

\begin{proof}
We must show that the following diagram commutes and the squares $(4)$ and $(5)$ are pullbacks:
\begin{center}
    \begin{tikzcd}
\mathbf{R}_1                                           &  & \mathbf{R_0} \arrow[white]{rd}[black, description]{(4)} \arrow[r] \arrow[ll, "s_0"']                                             & 1   \\
\mathbf{H}_2 \arrow[white]{rrd}[black, description]{(5)} \arrow[u, "d_1"] \arrow[d, "{(d_\top, d_0)}"'] &  & \mathbf{R}_0 \arrow[u, dotted] \arrow[d, dotted] \arrow[r, dotted] \arrow[ll, dotted] & 1 \arrow[d] \arrow[u]  \\
\mathbf{H}_1 \times \mathbf{R}_1                       &  & \mathbf{H}_1 \times \mathbf{R}_0 \arrow[r] \arrow[ll, " \identity \otimes s_0"]       & \mathbf{H}_1.
\end{tikzcd}
\end{center}
The pullback of $\identity \otimes s_0$ along $(d_\top, d_0)$ is the groupoid of families of contractions $P \twoheadrightarrow Q$ with a $H$-structure on $P$, such that the induced $H$-structure on $Q$ is an empty $H$-structure. This forces that both $P$ and $Q$ are empty, and therefore it is any $H$-structure on the empty poset, which is $\mathbf{R}_0$. This means that $(5)$ is a pullback.
\end{proof}

\begin{propo}\label{proposition:Aleftcomodule}
$A$ is naturally a left $B$-comodule bialgebra.
\end{propo}

\begin{proof}
Combining Lemmas \ref{lem:comul A is Bcomod map} and \ref{lem: counit A is Bcom map}, it follows that $A$ is naturally a left $B$-comodule coalgebra. The bialgebraic part follows from the monoidal property of $\mathbf{H}$ and $\mathbf{R}$, and that the multiplication of both is the same for how $\mathbf{R}$ was constructed.
\end{proof}

\begin{rema}
The arguments used in the proof that $A$ is naturally a left $B$-comodule coalgebra in Proposition \ref{proposition:Aleftcomodule} are similar to those given in the proof of Proposition 5.3 in \cite{carlier2019hereditary} considering contractions instead of monotone surjections and families of finite connected posets instead of sets. We prefer to add the proof to make the paper as self-contained as possible, but in any case, the ideas come from Carlier \cite{carlier2019hereditary}.
\end{rema}



 \begin{blanko}
{Running example, part III: Calaque--Ebrahimi-Fard--Manchon comodule bialgebra of rooted trees} \label{subsec: runningexamplepartIII}
\end{blanko}

The connected directed hereditary species $H_{\CEM}$ of trees, described in \S \ref{subsec: runningexamplepartI}, induces a comodule bialgebra. The comultiplication $\Delta_{H_{\CEM}}$ is explained in Remark \ref{remark:incidencecoalgebraCEM}. The second comultiplication is given by the directed restriction species of trees $R_{\BCK}$ (\S \ref{subsec:direcrestspe}). By Proposition \ref{proposition:Aleftcomodule}, this is a comodule bialgebra usually known as the Calaque--Ebrahimi-Fard--Manchon comodule bialgebra of rooted trees \cite{CALAQUE2011282}.

On the other hand, Kock \cite[\S 5.4]{Kock_2021} showed another presentation of the Calaque--Ebrahimi-Fard--Manchon comodule bialgebra of rooted trees related to the reduced Baez--Dolan construction of the terminal operad. In the operadic setting, operadic trees have an input edge (the leaf edge) and an output edge (the root edge), and there is a tree without nodes (where leaf=root). To relate operadic trees and combinatorial trees, we have to forget all decorations and shave off leaf edges and root edge. This construction is known as the \emph{core} of an operadic tree.    

Since the trees involved in the Calaque--Ebrahimi-Fard--Manchon comodule bialgebra are  combinatorial trees, it is not possible to realise this bialgebra as the incidence bialgebra of an operad, but this can be solved using the core construction. Kock \cite[Proposition 5.4.7]{Kock_2021} proved that taking core of the incidence comodule bialgebra of the reduced Baez--Dolan construction of the terminal operad, we obtain the Calaque--Ebrahimi-Fard--Manchon comodule bialgebra of rooted trees. Note that our approach to this comodule bialgebra is more straightforward since it follows from the connected directed hereditary species $H_{\CEM}$ of trees.

\begin{exa}[\textbf{Faà di Bruno comodule bialgebra of linear trees, part III}]\label{exa:linear trees 3}
The connected directed hereditary species $H_{\FB}$ of linear trees, described in Example \ref{exa:linear trees 1}, induces a comodule bialgebra: the comultiplication $\Delta_{H_{\FB}}$ is explained in Example \ref{exa:linear trees 1}. The second comultiplication is given by the directed restriction species of linear trees $R_{\FB}$, which is similar to $R_{\BCK}$ but we consider linear trees instead of arbitrary trees. By Proposition \ref{proposition:Aleftcomodule}, this is a comodule bialgebra.  

In fact, it is the Faà di Bruno comodule bialgebra of linear trees.  The form in which it arises here is very similar to that shown
in \cite[Section 5.2]{Kock_2021} from the reduced Baez--Dolan
construction on the identity monad. The difference between these two presentations of  the Faà di Bruno comodule bialgebra is analogous to the difference between the Calaque--Ebrahimi-Fard--Manchon bialgebra via the reduced Baez--Dolan construction~\cite{Kock_2021} and the construction from the directed hereditary species $H_{\CEM}$.


\end{exa}    

\section{Connected directed hereditary species and operadic categories}\label{section:CDHS and operadics categories}

The goal of this section is to construct a functor from the category of connected directed hereditary species to the category of operadic categories.

Operadic categories were introduced by Batanin and Markl~\cite{BM}
and used to prove the duoidal Deligne conjecture. An operadic
category is a kind of combinatorial structure whose `algebras'
are operads of various kinds depending on the operadic category.
For example, $\Delta$ is an operadic category and its operads
are nonsymmetric operads.

An operadic category $\mathcal{C}$ has chosen local terminal objects (that is,
in each connected component there is a chosen terminal object),
a cardinality functor $\vert - \vert : \mathcal{C} \to \mathbf{FinSet}$
to the standard skeleton of the category of
finite sets, and a notion of fibre: this is an assignment that
for each morphism $F \colon Y \to X$
in the operadic category, and each $i \in \vert {X} \vert$ gives a new
object denoted $f^{-1}(i)$, but this is abstract and does not
have to be a fibre in the usual sense of the word. For example,
it is not necessarily a subobject of $Y$. These data are subject to
many axioms, which can be formulated in various ways
\cite{BM}, \cite{lack2016operadic}, \cite{GKW2021}. The Carlier proof \cite{carlier2019hereditary} that hereditary
species induce operadic categories consisted in checking the whole
list of axioms.

Garner, Kock, and Weber~\cite{GKW2021} observed that the
chosen-local-terminals structure amounts precisely to be a coalgebra
for the upper decalage comonad, and went on to give a characterisation
of operadic categories in terms of a certain modified decalage
comononad. 

Recently Batanin, Kock, and Weber~\cite{BataninKockWeber} have
found a more conceptual characterisation, where all the axioms end up
formulated as simplicial identities. Their discovery is that just as
the chosen-local-terminals structure amounts to an extra top degeneracy
map, the fibre structures amount to an extra top face map, except
that this extra top face map lives in the Kleisli category for the
free-symmetric-monoidal-category monad. To understand this point of view, we will introduce a few concepts.

\begin{blanko}
{The lt-nerve} \label{subsec: lt nerve}
\end{blanko}

Recall that $\mathbf{Cat}_{\lt}$ is the category of categories with chosen local
terminals. Given a opfibration $p \colon \mathcal{E} \to \mathcal{B}$ and an object $x \in \mathcal{E}$, we will denote by $f_{!}(x)$ the opcartesian lift for a map $f \colon p(x) \to y$ in $\mathcal{B}$.

\begin{lem}\label{lemma: local terminals and discrete opfribration}

Let $\mathcal{B}$ be a category with chosen local terminal objects and let $p \colon \mathcal{E} \to \mathcal{B}$ be a discrete opfibration. Consider the pullback diagram
\[\begin{tikzcd}
	{\Fib_{{\mathcal{B_{\lt}}}}(p)} & {\mathcal{E}} \\
	\mathcal{B}_{\lt} & {\mathcal{B}}.
	\arrow[hook, from=2-1, to=2-2]
	\arrow["p", from=1-2, to=2-2]
	\arrow[from=1-1, to=2-1]
	\arrow[from=1-1, to=1-2]
	\arrow["\lrcorner"{anchor=center, pos=0.125}, draw=none, from=1-1, to=2-2]
\end{tikzcd}\]
The objects in $\Fib_{{\mathcal{B_{\lt}}}}(p)$ equip $\mathcal{E}$ with chosen local terminal objects.
\end{lem}

\begin{proof}
Let $x$ be an object in $\mathcal{E}$. Since $\mathcal{B}$ has chosen local terminal objects,  we have a unique map $t_{p(x)}: p(x) \to c_{p(x)}$ in $\mathcal{E}$, where $c_{p(x)}$ is a chosen local terminal object. Furthermore, we have a unique lift $(t_{p(x)})_{!}(x): x \to c_{x}$ for the map $t_{p(x)}$ in $\mathcal{E}$ since $p$ is a discrete opfibration. Note that $c_x \in \Fib_{\mathcal{B}_{\lt}}(p)$. To prove that $c_x$ is a chosen local terminal object in $\mathcal{E}$ it is enough to prove that any map $f \colon x \to y$ in $\mathcal{E}$ forces that $c_x = c_y$. Indeed, the objects $c_{p(x)}$ and $c_{p(y)}$ are connected by the zig-zag illustrated in the following diagram
\[\begin{tikzcd}
	{p(x)} & {p(y)} \\
	{c_{p(x)}} & {c_{p(y)}}.
	\arrow["{p(f)}", from=1-1, to=1-2]
	\arrow["{t_{p(x)}}"', from=1-1, to=2-1]
	\arrow["{t_{p(y)}}", from=1-2, to=2-2]
\end{tikzcd}\]
Since $c_{p(x)}$ and $c_{p(y)}$ are in the same connected component in $\mathcal{B}$, the chosen local terminal object property forces that $c_{p(x)} = c_{p(y)}$. This implies that the maps $(t_{p(y)})_{!}(y) \circ f \colon x \to c_y$ and $(t_{p(x)})_{!}(x) \colon x \to c_x$ are two lifts for the map $t_{p(x)}$. But the discrete opfibration property of $p$ forces that $(t_{p(y)})_{!}(y) \circ f = (t_{p(x)})_{!}(x)$, and hence $c_x = c_y$.
\end{proof}

We now work with posets whose underlying sets are ordinals and strict pullbacks. These assumptions are necessary to follow the work of Batanin and Markl~\cite[\S 1]{BM} on operadic categories. 

Moreover, we work only with $\mathbf{Set}$-valued species, as in the Schmitt theory of hereditary species. This is necessary to ensure that given a connected directed hereditary species $H \colon \mathbb{K}_p \to \mathbf{Set}$, its Grothendieck construction $\int H \to \mathbb{K}_p$ is a discrete opfibration.

For a connected directed hereditary species, the category $\mathbb{H}$ is the pullback of the Grothendieck construction $\int H \to \mathbb{K}_p$ along the inclusion $\mathbb{K} \to \mathbb{K}_p$.

\begin{lem}\label{lemma: GC of H is lt-nerve}
Let $H \colon \mathbb{K}_p \rightarrow \Set$ be a connected directed hereditary species. Then $\mathbb{H}$ is a category with chosen local terminal objects.
\end{lem}

\begin{proof}
Since $H$ is a presheaf with values in $\mathbf{Set}$, we have that $\int H \to \mathbb{K}_p$ is a discrete opfibration. This implies by the construction of $\mathbb{H}$ and the stability of discrete opfibrations under pullback that $\mathbb{H} \to \mathbb{K}$ is a discrete opfibration.
Combining this with the fact that $\mathbb{K}$ has a terminal object (the poset with one element), it follows that the set $H[1]$ of $H$-structures of the poset with one element is a set of chosen local terminal objects in $\mathbb{H}$ as a consequence of Lemma \ref{lemma: local terminals and discrete opfribration}. 
\end{proof}

Let $\mathbf{tsGrpd}$ denote the category of $\Grpd$-valued $\simplexcategory^{\ter}$-presheaves. 

\begin{defi}\label{defi: t-nerve}
For $\mathcal{C}$ a category with chosen local terminals, its $\lt$-nerve $\nerve^{\lt}(\mathcal{C})$ is the $\simplexcategory^{\ter}$-presheaf  
\[\begin{tikzcd}
	{{{\nerve^{\lt}}}(\mathcal{C}) \colon {(\simplexcategory^{\ter})}^{\op}} & {\Cat^{\op}} & \Set
	\arrow["{\Cat(-, \mathcal{C})}", from=1-2, to=1-3]
	\arrow[hook, from=1-1, to=1-2]
\end{tikzcd}\]
\end{defi}

Let us describe the $\lt$-nerve $\nerve^{\lt}(\mathcal{C})$: for $n \geq 0$, the set $\nerve^{\lt}(\mathcal{C})_n$ is the same as the set $\nerve(\mathcal{C})_{n}$. The set $\nerve^{\lt}(\mathcal{C})_{-1}$ is the set of chosen local terminal objects in $\mathcal{C}$. The face and degeneracy maps act as the usual nerve construction except in $d_\perp \colon \nerve^{\lt}(\mathcal{C})_{0} \to \nerve^{\lt}(\mathcal{C})_{-1}$ that sends each object in $\mathcal{C}$ to its corresponding chosen local terminal object. The degeneracy map $s_0 \colon \nerve^{\lt}(\mathcal{C})_{-1} \to \nerve^{\lt}(\mathcal{C})_{0}$ is the inclusion.

\begin{exa}\label{exa: ltnerve H}
Given $H \colon \mathbb{K}_p \rightarrow \Set$ a connected directed hereditary species, we have that $\mathbb{H}$ is a category with chosen local terminal objects by Lemma \ref{lemma: GC of H is lt-nerve}. The $\lt$-nerve of $\mathbb{H}$ is described as follows:
$(\nerve^{\lt} \mathbb{H})_{-1}$ is the set $H[1]$ of $H$-structures over the poset with one element.
$(\nerve^{\lt} \mathbb{H})_{0}$ is the set of finite posets with a $H$-structure. $(\nerve^{\lt} \mathbb{H})_{1}$ is the set of contractions with a $H$-structure on the first poset. For $n \geq 2$, the elements of the set $(\nerve^{\lt} \mathbb{H})_{n}$ are $(n-1)$-chains of contractions with a $H$-structure on the first poset in the chain. 
\end{exa}

\begin{blanko}
{Half decalage} \label{subsec: half decalge}
\end{blanko}

Let $\mathbf{sGrpd}^{\tps}$ denote the subcategory of pseudosimplicial groupoids that the unique pseudo-simplicial identities involve the top face maps. Given a pseudo simplicial space $X$ in $\mathbf{sGrpd}^{\tps}$, the half upper dec $\HDec_\top X$ is a $\simplexcategory^{\ter}$-presheaf obtained by deleting the top face maps and shifting everything one position down (for $n \geq -1$, we have that $(\HDec_\top X)_n = X_{n+1}$). Note that $\HDec_\top$ throws away the top face maps but keeps the top
degeneracy maps to get a $\simplexcategory^{\ter}$-presheaf (with values in groupoids). This gives a functor $\HDec_\top \colon \mathbf{sGrpd}^{\tps} \to \mathbf{tsGrpd}$.     

A $\simplexcategory^{\ter}$-presheaf $A$ is a \emph{$\simplexcategory^{\ter}$-Segal space} if the simplicial groupoid obtained after eliminating $A_{[-1]}$ is Segal. Since we are working in this section with $\mathbf{Set}$-valued species, we have to modify the definition of $\mathbf{H}$ to $\mathbf{H} = \mathsf{S} \nerve^{\lt} \mathbb{H}$.

\begin{lem}\label{lemma: H is a Dt segal}
Let $H \colon \mathbb{K}_p \rightarrow \Set$ be a connected directed hereditary species. Then $\HDec_\top \mathbf{H}$ is a $\simplexcategory^{t}$-Segal space.
\end{lem}

\begin{proof}
Since $\mathbb{H}$ is a category, its nerve $\nerve\mathbb{H}$ is a Segal space \ref{exa:fatnerve is Segal}, and therefore $\mathsf{S} \nerve \mathbb{H}$ is a Segal space, as $\mathsf{S}$ preserves pullbacks and hence Segal objects. This means that $\mathsf{S} \nerve^{\lt} \mathbb{H}$  is a {$\simplexcategory^{\ter}$-Segal space}.
\end{proof}

\begin{blanko}
{The category of connected directed hereditary species and $\mathbf{OpCat}$ \label{subsec: HSP to OpCat functor}}
\end{blanko}

For Batanin, Kock, and Weber~\cite{BataninKockWeber} an operadic category is a pseudosimplicial groupoid $X$ whose half upper dec $\HDec_\top X$ is equal to the symmetrical monoidal functor of the $ \lt$-nerve of some category with chosen local terminal objects. In short: it is a pair $(\mathcal{C},X)$
such that $\mathsf{S} \nerve^{\lt} \mathcal{C} =\HDec_\top X$, where $X \in \mathbf{sGrpd}^{\tps}$ and $\mathcal{C} \in \mathbf{Cat}_{\lt}$. To be more precise, they prove that the diagram
\[
\begin{tikzcd}
\mathbf{OpCat} \drpullback \ar[r] \ar[d] & \mathbf{sGrpd}^{\tps} \ar[d,
"\HDec_\top"]  \\
\mathbf{Cat}_{\lt} \ar[r, "\mathsf{S}\nerve^{\lt}"'] & \mathbf{tsGrpd}
\end{tikzcd}
\]
is a strict pullback of categories.

Let $\mathbf{ConDirHerSp}$ denote the category of connected directed hereditary species. For each connected directed hereditary species $H$, we have a pseudosimplicial groupoid $\mathbf{H}$. 
For a map $f \colon H' \to H$ between connected directed hereditary species, the Grothendieck construction of $f$ gives a map $\int f \colon \mathbb{H}' \to \mathbb{H}$. This map induces a map from $\mathbf{H}'$ to $\mathbf{H}$ that we denote as $\GD f \colon \mathbf{H}' \to \mathbf{H}$.
Since $f$ is simplicial map for each $P \in \mathbb{K}$, we have a functor from $f_P \colon H'[P] \to H[P]$. So the map $\GD f$ sends an $n$-simplex in $\mathbf{H}'$ which is a $n$-chain of contractions 
\[\begin{tikzcd}
	{P_0} & {P_1} & \dots & {P_{n-2}} & {P_{n-1}}
	\arrow[two heads, from=1-1, to=1-2]
	\arrow[two heads, from=1-2, to=1-3]
	\arrow[from=1-3, to=1-4]
	\arrow[two heads, from=1-4, to=1-5]
\end{tikzcd}\]
with an $H'$-structure $X$ in $P_0$ to the same $n$-chain but with the $H$-structure $f_{P_0}(X)$ in $P_0$ which is a $n$-simplex in $\mathbf{H}$.

We define the functor $\GD \colon \mathbf{ConDirHerSp} \to \mathbf{sGrpd}^{\tps}$ as follows: $\GD(H) = \mathbf{H}$, for each object $H \in \mathbf{ConDirHerSp}$ and for a morphism  
$f \colon H' \to H$, the functor sends it to $\GD f$. Furthermore, $\HDec_\top \circ \GD(H)$ is a $\simplexcategory^{\ter}$-Segal space by Lemma \ref{lemma: H is a Dt segal}. 

On the other hand, Lemma \ref{lemma: GC of H is lt-nerve} established that $\mathbb{H}$ is a category with chosen local terminal objects given by the set $H[1]$ of $H$-structures of the poset with one element. This gives a functor $\int^{\ast} \colon \mathbf{ConDirHerSp} \to \mathbf{Cat}_{\lt}$ that sends $H$ to the category $\mathbb{H}$.

\begin{lem}\label{lemma: CDHS commutes with GC and GD}
The diagram
\[\begin{tikzcd}
	{\mathbf{ConDirHerSp}} & {\mathbf{sGrpd}^{\tps}} \\
	{\mathbf{Cat}_{\lt}} & {\mathbf{tsGrpd}}
	\arrow["{\int^{\ast}}"', from=1-1, to=2-1]
	\arrow["{\mathsf{S} \nerve^{\lt}}"', from=2-1, to=2-2]
	\arrow["{\HDec_\top}", from=1-2, to=2-2]
	\arrow["\GD", from=1-1, to=1-2]
\end{tikzcd}\]
strictly commutes.
\end{lem}

\begin{proof}
Let $H$ be an object in ${\mathbf{ConDirHerSp}}$. The commutativity of the diagram follows from the fact that $\int^*(H) = \mathbb{H}$ and $\int^{\mathbf{K}}(H) = \mathbf{H} = \mathsf{S} \nerve^{\lt} (\mathbb{H})$.
\end{proof}

\begin{propo}\label{propo: CDHS to OpCat functor}
There exists a canonical functor from the category of connected directed hereditary species ${\mathbf{ConDirHerSp}}$ to the category of operadic categories $\mathbf{OpCat}$.
\end{propo}

\begin{proof}
Consider the diagram
\[\begin{tikzcd}
	{\mathbf{ConDirHerSp}} \\
	& {\mathbf{OpCat}} & {\mathbf{sGrpd^{\tps}}} \\
	& {\mathbf{Cat}_{\lt}} & {\mathbf{tsGrpd}}.
	\arrow[from=2-2, to=3-2]
	\arrow["{}", dotted, from=1-1, to=2-2]
	\arrow["{\HDec_\top}", from=2-3, to=3-3]
	\arrow["{\mathsf{S} \nerve^{\lt}}"', from=3-2, to=3-3]
	\arrow[from=2-2, to=2-3]
	\arrow["\lrcorner"{anchor=center, pos=0.125}, draw=none, from=2-2, to=3-3]
	\arrow["{\int^{\ast}}"', bend right=20, from=1-1, to=3-2]
	\arrow["\GD", bend left=20, from=1-1, to=2-3]
\end{tikzcd}\]
Batanin, Kock, and Weber~\cite{BataninKockWeber} proved that the square is a pullback. The outer diagram commutes by Lemma \ref{lemma: CDHS commutes with GC and GD}. The dotted arrow then exists by the pullback property of $\mathbf{OpCat}$. The functor $\mathbf{ConDirHerSp} \to \mathbf{OpCat}$ sends a connected directed hereditary species $H$ to the pair $(\mathbb{H}, \mathbf{H})$.
\end{proof}

Connected directed hereditary species constitute a new family of examples of operadic categories. In fact, the connected directed hereditary species associated to the Fauvet--Foissy--Manchon comodule bialgebra of finite topologies and admissible maps; and the connected directed hereditary species $H_{\CEM}$ associated to the Calaque--Ebrahimi-Fard--Manchon comodule bialgebra of rooted trees are now covered by the theory of operadic categories.

\section{Directed hereditary species as monoidal decomposition spaces, comodule bialgebras and operadic categories}\label{section:decomposition D}

Schmitt hereditary species are not connected directed hereditary species, as the fibres along a surjection between discrete posets are not necessarily connected. To cover these examples, in this section, we introduce the notion of collapse, which allows for non-connected fibres. This leads to the notion of (not-necessarily-connected) directed hereditary species. Furthermore, each directed hereditary species induces a decomposition space (\ref{subsec:direcherspecies as decomposition}), a comodule bialgebra (\ref{subsec:dhs comodule bialgebra}), and a operadic category (\ref{subsec:dhs as operadic categories}).


\begin{blanko} 
{Partially reflecting maps}\label{subsec:mapsposets}
\end{blanko}

\begin{defi}\label{definition: contraction}
A map of posets $f \colon P \rightarrow Q$ is \emph{partially reflecting} if $f(x) < f(y)$ in $Q$ implies that $x < y$ in $P$. Partially reflecting monotone surjections are called \emph{collapse maps}.
\end{defi}

\begin{lem}\label{lemma:stablereflectingunderpullback}
In the category of posets, collapse maps are stable under pullback.
\end{lem}

\begin{proof}
Let $P, Q$ and $V$ be posets. Let $f \colon P \twoheadrightarrow V$ be a collapse. Let $g \colon Q \rightarrow V$ be a monotone map and let
\begin{center}
\begin{tikzcd}
P \times_V Q \arrow[d, "\pi_P"'] \arrow[r, "\pi_Q"] \drpullback & Q \arrow[d, "g"] \\
P \arrow[r, "f"', two heads]                                   & V                          
\end{tikzcd}
\end{center}
be a pullback diagram. Since monotone surjections are stable under pullback, we have that $\pi_Q$ is a monotone surjection. It remains to prove that $\pi_Q$ is a partially reflecting map. Let $(p, q)$ and $(p', q')$ be objects in $P \times_V Q$. This means
\begin{equation}\label{spullback1}\tag{1}
f(p) = g(q) \; \mbox{and} \; \; f(p') = g(q') .
\end{equation}
Assuming that $\pi_Q(p, q) <_Q \pi_Q(p', q')$, we get that
\begin{equation}\label{spullback2}\tag{2}
q <_Q q'.
\end{equation}
Since $g$ is a monotone map, $g(q) <_V g(q')$. This together with Eq.~(\ref{spullback1}) implies that $f(p) <_V f(p')$. So $p <_P p'$ by the partially reflecting property of $f$. This combining with (\ref{spullback2}) implies $(p, q) < (p', q')$. Hence, $\pi_Q$ is partially reflecting. 
\end{proof}

\begin{blanko}
{Directed Hereditary Species}\label{subsec:DHS no connected case}
\end{blanko}

We can also define the notion of directed hereditary species for the non-connected case by substituting the category $\mathbb{K}_p$ of partially defined contractions by the category $\mathbb{D}_p$ of partially defined collapse maps in Definition \ref{definition: DCHS}. 

A \emph{partially defined collapse map} $P \rightarrow Q$ consists of a span $\begin{tikzcd}[column sep=scriptsize]
	P & {P'} & Q
	\arrow["f", two heads, from=1-2, to=1-3]
	\arrow["\imath"', tail, from=1-2, to=1-1]
\end{tikzcd}$
where $i$ is a convex map and $f$ is a collapse. Partially defined collapse maps are composed by pullback composition of spans in the category $\mathbb{D}$.

\begin{defi}\label{definition: CHS}
A \emph{directed hereditary species} is a functor $H \colon \mathbb{D}_p \rightarrow \Grpd$. 
\end{defi}

\begin{exa}
Any Schmitt hereditary species is a directed hereditary species since any set can be regarded as a discrete poset, and any partial surjection of sets is then a partially defined collapse map of discrete posets.
\end{exa}

\begin{blanko} 
{Pseudosimplicial groupoid of collapse maps}\label{subsec: CP definition}
\end{blanko}

In this subsection, the monoidal decomposition space $\mathbf{D}$ of finite non-empty posets and collapse maps is defined in analogy with $\mathbf{K}$. Let $\mathbb{D}$ denote the category of finite posets and collapse maps. We define $\mathbf{D} := \mathsf{S}\fatnerve^{\lt}(\mathbb{D})$ to be the symmetric monoidal category functor $\mathsf{S}$ applied to the fat $\lt$-nerve of $\mathbf{D}$. All the face maps (except the missing top ones) and degeneracy maps are $\mathsf{S}$ applied to the face and degeneracy maps of $\fatnerve^{\lt}(\mathbb{D})$. The top face map of a collapse map is the family of posets, obtained as the fibre over each element in the target poset.

\begin{propo}\label{proposition:Dsimplicial}
The groupoids $\mathbf{D}_n$ and the degeneracy and face maps given above form a pseudosimplicial groupoid $\mathbf{D}$.
\end{propo}

\begin{proof}
The proof is analogous to that of Proposition \ref{proposition:Ksimplicial}.
\end{proof}


\begin{propo}\label{proposition:dectopDisSegal}
We have an equality $\Dec_\top \mathbf{D} = \mathsf{S}\fatnerve \mathbb{D}$.
\end{propo}

\begin{proof}
The proof is analogous to that of Proposition \ref{proposition:dectopKisSegal} but using  collapse maps instead of contractions. 
\end{proof}

\begin{rema}\label{rema:decttopDexplication}
Since $\mathsf{S}$ preserves pullbacks and $\nerve \mathbb{D}$ is a Segal space,
Proposition \ref{proposition:dectopDisSegal} implies that $\Dec_\top \mathbf{D}$ is a Segal space. This is equivalent to saying that for each $n \geq 2$ the following diagram is a pullback for $0 < i < n$:
\begin{center}
\begin{tikzcd}
\mathbf{D}_{n+1}   \arrow[r, "d_{i+1}"]\arrow[d, "d_\bot"']& \mathbf{D}_n \arrow[d, "d_\bot"] \\
\mathbf{D}_n \arrow[r, "d_i"']& \mathbf{D}_{n-1}.
\end{tikzcd}
\end{center}
\end{rema}

\begin{lem}\label{lemma:surjectivetoprovedecbotDisSegal}
Suppose we have a collapse $f \colon P \twoheadrightarrow Q$ and a family of collapse maps $\lbrace h_q \colon P_q\twoheadrightarrow W_q \rbrace_{q \in {Q}}$. Then there exists a unique poset $W$ and collapse maps $h$ and $g$ such that the diagram 
\[\begin{tikzcd}
	P & W & Q \\
	{P_q} & {W_q}
	\arrow["{h_q}"', two heads, from=2-1, to=2-2]
	\arrow[tail, from=2-1, to=1-1]
	\arrow[dotted, tail, from=2-2, to=1-2]
	\arrow["h"', dotted, two heads, from=1-1, to=1-2]
	\arrow["g"', dotted, two heads, from=1-2, to=1-3]
	\arrow["f", bend left = 20, two heads, from=1-1, to=1-3]
\end{tikzcd}\]
commutes. Here the vertical arrows are convex inclusions.
\end{lem}

\begin{proof}
We will do the proof in two steps: in the first place, we will construct the underlying set of the poset $W$ and the functions $h \colon P \rightarrow W$ and $g \colon W \rightarrow Q$. After that, we will construct a partial order $<_W$ on $W$ forced by the requirement that $h$ and $g$ are collapse maps. 
\begin{itemize}
\item Put $W := \sum_{q \in Q} W_q$ and $h := \sum_{q \in Q} h_q$. The map $g \colon W \to Q$ is defined as $g(w) = q$ for $w \in W_q$. Furthermore, the diagram
\[\begin{tikzcd}
	P & W & Q \\
	{P_q} & {W_q} & 1
	\arrow["{h_q}"', two heads, from=2-1, to=2-2]
	\arrow[tail, from=2-1, to=1-1]
	\arrow[dotted, tail, from=2-2, to=1-2]
	\arrow["h"', dotted, two heads, from=1-1, to=1-2]
	\arrow["g"', dotted, two heads, from=1-2, to=1-3]
	\arrow["f", bend left = 20, two heads, from=1-1, to=1-3]
	\arrow[from=2-2, to=2-3]
	\arrow["{\ulcorner q\urcorner }"', from=2-3, to=1-3]
\end{tikzcd}\]
commutes at the level of sets by the way $h$ and $g$ were defined.

\item The partial order $<_W$ on $W$ is defined as follows: for $w, w' \in W$, we declare that $w <_W w'$ if one of the following conditions is satisfied:
\begin{enumerate}
\item In case $w, w' \in W_q$ and  $w <_{W_{q}} w'$;
\item In case $g(w) \neq g(w')$ and $g(w) <_Q g(w')$.
\end{enumerate}
The condition $(b)$ is required to make $g \colon W \rightarrow Q$ to be a collapse. Let $p$ and $p'$ be objects in $P_i$ such that $h(p) <_W h(p')$. Applying $g$, we have that $g(h(p)) <_{Q} g(h(p'))$. Since $f = g \circ h$, it follows that $f(p) <_Q f(p')$. By hypothesis, $f$ is partially reflecting. This implies that $p <_{P} p'$. Therefore, $h$ is a partially reflecting map. Furthermore, $h$ is a monotone surjection. Indeed, let $p$ and $p'$ be objects in $P$ such that $p <_{P} p'$. Applying $g \circ h$, we have that $g(h(p)) <_Q g(h(p'))$. Since $g$ is partially reflecting, it follows that $h(p) <_W h(p')$. 
\end{itemize}
\end{proof}

\begin{rema}
Lemma \ref{lemma:surjectivetoprovedecbotDisSegal} would not hold for general monotone surjections instead of collapse maps. Suppose we have monotone surjections illustrated in the following picture:
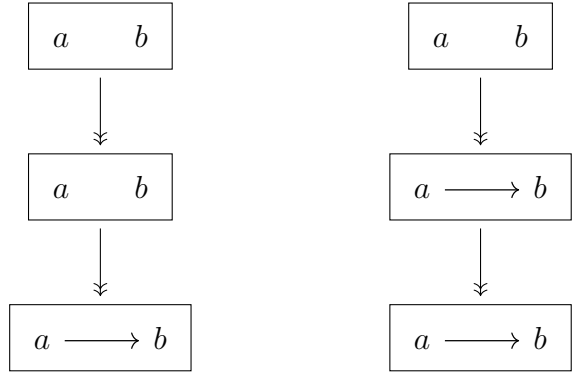
\begin{center}
\begin{tikzpicture}[yscale=1,xscale=1]
\node[draw] (box3) at (0,4){
\begin{tikzcd}[column sep=small,row sep=scriptsize]
	a & b
\end{tikzcd}
};

\node[draw] (box2) at (0,2){%
\begin{tikzcd}[column sep=small,row sep=scriptsize]
	a & b
\end{tikzcd}
};
\node[draw] (box1) at (0,0){%
\begin{tikzcd}
	a & b
	\arrow[from=1-1, to=1-2]
\end{tikzcd}
};
\draw[->>, shorten <=3pt, shorten >=3pt] (box2) to (box1);
\draw[->>, shorten <=3pt, shorten >=3pt] (box3) to (box2);

\node[draw] (box32) at (5,4){
\begin{tikzcd}[column sep=small,row sep=scriptsize]
	a & b
\end{tikzcd}
};

\node[draw] (box22) at (5,2){%
\begin{tikzcd}
	a & b
	\arrow[from=1-1, to=1-2]
\end{tikzcd}
};

\node[draw] (box12) at (5,0){%
\begin{tikzcd}
	a & b
	\arrow[from=1-1, to=1-2]
\end{tikzcd}
};
\draw[->>, shorten <=3pt, shorten >=3pt] (box22) to (box12);
\draw[->>, shorten <=3pt, shorten >=3pt] (box32) to (box22);
\end{tikzpicture}
\end{center}  

If the inclusion map from $\lbrace a, b \rbrace$ to $\lbrace a \rightarrow b \rbrace$ plays the role of $f$ in Lemma \ref{lemma:surjectivetoprovedecbotDisSegal}, the collapse maps illustrated above are two solutions to the problem described in Lemma \ref{lemma:surjectivetoprovedecbotDisSegal} and therefore the lemma would be false.
\end{rema}

\begin{lem}\label{lemma:discretedibrationDdi}
For each $0 < i < n$, the map $d_i \colon \mathbf{D}_n \rightarrow \mathbf{D}_{n-1}$ is a fibration.
\end{lem}

\begin{proof}
The proof is analogous to that of Proposition \ref{lemma:discretedibrationK}, but using the category $\mathbb{D}$ instead of $\mathbb{K}$. 
\end{proof}

\begin{propo}\label{proposition:Disdecompositionspace}
The pseudosimplicial groupoid $\mathbf{D}$ is a decomposition space.
\end{propo}

\begin{proof}
The proof is analogous to that of Proposition \ref{proposition:Kisdecompositionspace}, but applying Remark \ref{proposition:dectopDisSegal} instead of Remark \ref{rema:decttopKexplication} and Lemma \ref{lemma:surjectivetoprovedecbotDisSegal} instead of Lemma \ref{lemma:surjectivetoprove K is DS}.
\end{proof}

\begin{propo}\label{proposition:Discomplete}
The decomposition space $\mathbf{D}$ is complete.
\end{propo}

\begin{proof}
The proof is analogous to that of Proposition \ref{proposition:Kiscomplete}, but using $\mathbb{D}$ instead of $\mathbb{K}$.
\end{proof}

\begin{propo}\label{proposition:Dislocally}
The decomposition space $\mathbf{D}$ is locally finite, locally discrete, and of locally finite length.
\end{propo}

\begin{proof}
The proof is analogous to that of Proposition \ref{proposition:Kislocally}.
\end{proof}

Since the simplicial groupoid $\mathbf{D}$ is equal to $\mathsf{S}\fatnerve^{\lt}\mathbb{D}$, the decomposition space $\mathbf{D}$ is a monoidal decomposition space. The monoidal structure is obtained by categorical sum as in $\mathbf{K}$.

\begin{blanko}
{Directed hereditary species as decomposition spaces}\label{subsec:direcherspecies as decomposition}
\end{blanko}
\newcommand{\decH}{\mathbf{H}}

We can also construct a decomposition space $\mathbf{H}$ from a directed hereditary species $H \colon\mathbb{D}_p \rightarrow \Grpd$ similarly to the connected case. We define $\mathbf{H}_1$ as the groupoid of families of non-empty $H$-structures. An object of $\mathbf{H}_n$ is a family of chains of  collapses

\begin{equation*}
\begin{tikzcd}[column sep=35pt]
P_0 \arrow[r, "f_0", two heads] & P_1 \arrow[r, two heads] &   \cdots\arrow[r, two heads]&P_{n-2} \arrow[r, "f_{n-2}",two heads]&P_{n-1}
\end{tikzcd}
\end{equation*}
with an $H$-structure on each $P_0$. The inner face maps and degeneracy maps are induced by the maps in $\mathbf{D}$. For the bottom face map we use functoriality of $H$ along collapse maps. Similarly, we use (contravariant) functoriality in convex map to define the top face map. The groupoid $\mathbf{H}_0$ is defined as the groupoid of families of $H$-structures over the poset with one element.

\begin{propo}\label{propo: DH is ds} 
The groupoids $\mathbf{H}_n$ form a monoidal decomposition space $\mathbf{H}$.
\end{propo}

\begin{proof} 
The proof is analogous to that of Proposition \ref{propo: H is a ds connected case}.
\end{proof}

\begin{propo} \label{propo: DH is locally etc}
The decomposition space $\mathbf{H}$ is complete, locally finite, locally discrete, and of locally finite length.
\end{propo}

\begin{proof} 
The proof is analogous to that of Proposition \ref{propo: H is locally etc}.
\end{proof}

\begin{blanko}
{The incidence comodule bialgebra of non-connected directed hereditary species}\label{subsec:dhs comodule bialgebra}
\end{blanko}

Recall that every directed  hereditary species is also a directed restriction species, in the sense of Gálvez--Kock--Tonks \cite{GKT:restr}, through precomposition with the inclusion $\mathbb{C}^{\op} \rightarrow \mathbb{D}_p$. Let $H$ be a directed hereditary species. We denote by $R$ the induced directed restriction species and, as usual, $\mathbf{H}$ and $\mathbf{R}$ the corresponding to decomposition spaces. Also, the incidence coalgebra of $\mathbf{R}$ is denoted $A$, and the incidence bialgebra of $\mathbf{H}$ is denoted $B$. The following result is a consequence of the theory developed in Section \ref{section:comudule bialgebra} but using collapse maps instead of contractions.

\begin{propo}\label{proposition:Alefctcomodulebialgebra nonconnected}
$A$ is a left comodule bialgebra over $B$.
\end{propo}

\begin{blanko}
{Directed hereditary species as operadic categories}\label{subsec:dhs as operadic categories}
\end{blanko}

We now work with $\mathbf{Set}$-valued species as in the classical theory and posets whose underlying sets are ordinals. These are necessary to ensure that given a directed hereditary species $H \colon \mathbb{D}_p \to \mathbf{Set}$, its Grothendieck construction $\int H \to \mathbb{D}_p$ is a discrete opfibration and have precise constructions in the context of operadic categories. For a directed hereditary species $H$, the category $\mathbb{H}$ is the pullback of the Grothendieck construction $\int H \to \mathbb{D}_p$ along the inclusion $\mathbb{D}\to \mathbb{D}_p$.

\begin{lem}\label{lemma: GC of DH is lt-nerve}
Let $H \colon \mathbb{D}_p \rightarrow \Set$ be a directed hereditary species. Then $\mathbb{H}$ is a category with local terminal objects.
\end{lem}

\begin{proof}
Since $H$ is a presheaf with values in $\mathbf{Set}$, we have that $\int H \to \mathbb{D}_p$ is a discrete opfibration. This implies by the construction of $\mathbb{H}$ and the stability of discrete opfibration under pullback that $\mathbb{H} \to \mathbb{D}$ is a discrete opfibration.
Combining this with the fact that $\mathbb{D}$ has a terminal object (the poset with one element), it follows that the set $H[1]$ of $H$-structures of the poset with one element is the set of local terminal objects in $\mathbb{H}$ by Lemma \ref{lemma: local terminals and discrete opfribration}. 
\end{proof}

\begin{exa}\label{exa: lt-nerve DH}
Given $H \colon \mathbb{D}_p \rightarrow \Set$ a directed hereditary species, we have that $\mathbb{H}$ is a category with local terminal objects by Lemma \ref{lemma: GC of DH is lt-nerve}. The $\lt$-nerve of $\mathbb{H}$ is described similarly to the connected case (Example \ref{exa: ltnerve H}).
\end{exa}

\begin{lem}\label{lemma: DH is a Dt segal}
Let $H \colon \mathbb{D}_p \rightarrow \Set$ be a directed hereditary species. Then $\HDec_\top \mathsf{S} \nerve^{\lt} \mathbb{H}$ is a $\simplexcategory^{t}$-Segal space.
\end{lem}

\begin{proof}
The proof is the same as Lemma \ref{lemma: H is a Dt segal}.
\end{proof}

Let $\mathbf{DirHerSp}$ denote the category of directed hereditary species. For each directed hereditary species $H$, we have a pseudosimplicial groupoid $\mathbf{H} = \mathsf{S} \nerve^{\lt} \mathbb{H}$. This construction gives a functor $\GDD \colon \mathbf{DirHerSp} \to \mathbf{sGrpd}^{\tps}$ defined by $\GDD(H) = \mathbf{H}$. Furthermore, $\HDec_\top \circ \GDD(H)$ is a $\simplexcategory^{\ter}$-Segal space by Lemma \ref{lemma: DH is a Dt segal}. 

On the other hand, Lemma \ref{lemma: GC of DH is lt-nerve} established that $\mathbb{H}$ is a category with local terminal objects given by the set of $H$-structures of the poset with one element $H[1]$. This gives a functor $\int^{\ast} \colon \mathbf{DirHerSp} \to \mathbf{Cat}_{\lt}$ that sends $H$ to the category $\mathbb{H}$.

\begin{lem}\label{lemma: DHS commutes with GC and GD}
The diagram
\[\begin{tikzcd}
	{\mathbf{DirHerSp}} & {\mathbf{sGrpd}^{\tps}} \\
	{\mathbf{Cat}_{\lt}} & {\mathbf{tsGrpd}}
	\arrow["{\int^{\ast}}"', from=1-1, to=2-1]
	\arrow["{\mathsf{S} \nerve^{\lt}}"', from=2-1, to=2-2]
	\arrow["{\HDec_\top}", from=1-2, to=2-2]
	\arrow["\GDD", from=1-1, to=1-2]
\end{tikzcd}\]
strictly commutes.
\end{lem}

\begin{proof}
The proof is analogous to that Proposition \ref{lemma: CDHS commutes with GC and GD}.
\end{proof}	

\begin{propo}\label{propo: DHS to OpCat functor}
There exists a canonical functor from the category of directed hereditary species ${\mathbf{DirHerSp}}$ to the category of operadic categories $\mathbf{OpCat}$.
\end{propo}

\begin{proof}
The proof is analogous to that Proposition \ref{propo: CDHS to OpCat functor}.
\end{proof}

\addcontentsline{toc}{section}{References}

\begin{footnotesize}

\end{footnotesize}

  \begin{tabular}{@{}l@{}}%
   \scriptsize{DEPARTAMENT DE MAT\`{E}MATIQUES, UNIVERSITAT AUT\`{O}NOMA DE BARCELONA}\\
    \textit{E-mail address}: \texttt{cebrian@mat.uab.cat} \\
    
    \scriptsize{DEPARTAMENT DE MAT\`{E}MATIQUES, UNIVERSITAT AUT\`{O}NOMA DE BARCELONA}\\
    \textit{E-mail address}: \texttt{wforero@mat.uab.cat}
  \end{tabular}
  
\end{document}